\documentclass[letterpaper,11pt,reqno]{amsart}
\usepackage{mlmodern}
\usepackage[utf8]{inputenc}
\usepackage[T1]{fontenc}
\usepackage{amsmath, bm}
\usepackage{amssymb}
\usepackage{graphicx}
\usepackage[counterclockwise]{rotating}
\usepackage{mathtools}
\usepackage{enumerate}
\usepackage{amsfonts}
\usepackage{dirtytalk}
\usepackage{chngcntr}
\counterwithin{figure}{section}
\usepackage{xy}
\usepackage{amsthm}
\usepackage{xfrac}
\usepackage[dvipsnames]{xcolor}
\usepackage[obeyspaces,hyphens,spaces]{url}

\usepackage[
backend=bibtex,
giveninits=true,
style=alphabetic, url=false, isbn=false,
maxbibnames=99,mincitenames=1, maxcitenames=4]{biblatex}
\renewbibmacro{in:}{}
\usepackage[obeyspaces,hyphens,spaces]{url}

\DeclareFieldFormat[article,periodical]{volume}{\mkbibbold{#1}} 

\usepackage[urlcolor=Magenta, 
linkcolor=red, 
citecolor=Green,
colorlinks=true,
linktocpage=true]{hyperref}
\usepackage{xr-hyper}

\usepackage{tikz}
\usepackage{tikz-cd}
\tikzstyle{lien}=[->,>=stealth,rounded corners=5pt,thick]
\usepackage{geometry}
\usepackage{cleveref}

\crefformat{section}{\S#2#1#3} 
\crefformat{subsection}{\S#2#1#3}
\crefformat{subsubsection}{\S#2#1#3}

\newtheorem{thm}{Theorem}[section]
\newtheorem{prop}[thm]{Proposition}
\newtheorem{lem}[thm]{Lemma}

\newtheorem{assum}{Condition}

\theoremstyle{definition}
\newtheorem{defi}[thm]{Definition}
\newtheorem{rem}{Remark}[section]

\numberwithin{equation}{section}

\usepackage{import}
\usepackage{graphicx}
\graphicspath{{images/}}

\usepackage{blindtext}

\usepackage{subfiles}

\author{Julien Moy}
\address{Laboratoire de Mathématiques d'Orsay, Universit{\'e} Paris-Saclay, 91405 Orsay cedex, France}
\email{julien.moy@universite-paris-saclay.fr}
\title[Smooth spectral statistics of random Schr\"odinger operators]{Smooth spectral statistics of random perturbations of Schr\"odinger operators near Anosov energy levels}

\def\lesssimA#1#2{\mathrel{ \vcenter{\offinterlineskip%
			\ialign{\hfil##\hfil\cr \noalign{\vskip 0.15ex} $#1<$\cr \noalign{\vskip 0.25ex} $#1\sim$\cr}%
}}}

\def\lesssim{\mathpalette\lesssimA{}}

\def\gtrsimA#1#2{\mathrel{ \vcenter{\offinterlineskip%
			\ialign{\hfil##\hfil\cr \noalign{\vskip 0.15ex} $#1>$\cr \noalign{\vskip 0.25ex} $#1\sim$\cr}%
}}}

\def\gtrsim{\mathpalette\gtrsimA{}}

\begin{document}
	\maketitle
	
	\begin{abstract} 
		
		We investigate the spectral statistics of random perturbations of semiclassical Schr\"odinger operators on compact manifolds, near energy levels corresponding to chaotic classical dynamics. The prototypical example is that of an operator $h^2\Delta+V(x)$ that is perturbed by a smooth potential $h^\alpha V_\omega$ with $\alpha\in (0,1)$, and $V_\omega$ is a random potential that decorrelates on distances $h^\beta$, with $0<\beta<2\alpha$. We show that for a generic perturbation, the spectral fluctuations of the smoothed counting function of eigenvalues obey a universal behavior at a certain mesoscopic scale, which is coherent with the predictions of random matrix theory.

	\end{abstract}
	\tableofcontents

	\section{Introduction}
A central question in quantum chaos is to understand the behavior in the semiclassical limit of spectra of classically chaotic quantum systems. The guiding example to have in mind is the Laplacian on a closed hyperbolic surface, for which the corresponding classical dynamics is the geodesic flow on the unit cotangent bundle. The Bohigas--Giannoni--Schmit conjecture (BGS) \cite{PhysRevLett.52.1,BGS2} predicts that \emph{generically}, the spectra of such systems display the same fluctuation properties at small energy scales as those predicted by random matrix theory. In the absence of spin, the spectral statistics should match those of the Gaussian Orthogonal Ensemble (GOE) when the system is time-reversal invariant, and the Gaussian Unitary Ensemble (GUE) otherwise. In contrast, for classically integrable systems, the Berry--Tabor conjecture \cite{BerryTabor} predicts that the spectrum displays Poissonian behavior. We refer to the nice review of Ullmo \cite{Ullmo2016BohigasGiannoniSchmitC} for more background on the BGS conjecture.

Although theoretical evidence by Berry \cite{BerryRigidity,Berry1} and numerical experiments support the BGS conjecture, there is no rigorous proof for any individual system. Heuristic arguments typically rely on the Gutzwiller trace formula, which expresses the semiclassical density of states as the sum of a \emph{smooth part} (the Weyl term) and an \emph{oscillating part} involving classical periodic orbits. The energy variance of this oscillating sum has been estimated using a so-called \emph{diagonal approximation} which consists in discarding pairs of distinct periodic orbits. However, the validity of this approximation remains speculative, and is even known to fail dramatically for certain systems exhibiting exceptional symmetries, such as arithmetic hyperbolic surfaces \cite{Luo1994}, for which the length spectrum exhibits exponential degeneracies.

Recent works in the mathematical literature have focused on finding mechanisms to justify this diagonal approximation, by introducing some randomness to the system. In the context of hyperbolic surfaces, instead of studying a fixed surface, one changes perspective and considers an ensemble of surfaces equipped with a probability measure, in the hope of obtaining results that hold with high probability. We mention the works of Rudnick \cite{rudnick2022goe}, Rudnick--Wigman \cite{rudnick2023central,rudnick2023sure} on the Weil--Petersson model of random hyperbolic surfaces, and Naud \cite{naud2022random}, Maoz \cite{Maoz2} and the author \cite{Moy2024} on the discrete model of random covers of a negatively curved surface. These works investigate ensemble fluctuations of the smoothed counting function of eigenvalues of the square root of the Laplacian $\sqrt{\Delta-1/4}$. For both these models, the \emph{ensemble variance} of the smoothed counting function of eigenvalues was shown to be universal in the limit of large genus (for the Weil--Petersson model) or large degree (for the model of random covers), and conformed to the predictions of random matrix theory.

In this paper, we pursue this line of work by studying the spectrum of random perturbations of a fixed semiclassical Schr\"odinger operator near an energy $E$, such that the classical dynamics on the energy shell $E$ is of Anosov type. We consider the model recently introduced by Ingremeau and Vogel \cite{ingremeau2023emergence} to study the \emph{random wave conjecture}.

\subsection{Spectral fluctuations} Let $(M,\operatorname{vol})$ be a smooth closed, connected surface, equipped with a smooth volume form $\operatorname{vol}$. Let $P_h$ be a selfadjoint semiclassical pseudodifferential operator on $M$ depending on a parameter $h\in (0,1]$, with elliptic principal symbol $p$. The example to keep in mind for now is $P_h=h^2\Delta_g+V(x)$ acting on $L^2(M,\operatorname{vol}_g)$ where $\Delta_g$ is the positive Laplace--Beltrami operator associated with a metric $g$, and $V$ is a smooth potential. We will allow for more general operators later on. The operator $P_h$ has discrete spectrum on $L^2(M,\rm{vol})$:
\[\operatorname{Spec}(P_h)=\{E_{j}\}_{j\ge 0}.\]
Throughout the paper, the Fourier transform of a function $f$ is defined by $\widehat f(t)=\frac{1}{2\pi}\int_{\mathbf{R}}\mathrm{e}^{-\mathrm{i}\xi t}f(\xi)\mathrm{d}\xi$. We introduce smoothed counting functions of eigenvalues. 
\begin{defi}[Smoothed counting function] Let $f$ be a test function with smooth Fourier transform supported in $[-1,1]$. Let $\eta>0$. We define
	\begin{equation} \label{eq: counting func} N_{f}(E,\eta):=\operatorname{Tr} f\Big(\frac{P_{h}-E}{\eta}\Big)=\sum_j f\Big(\frac{E_{j}-E}{\eta}\Big). \end{equation}
\end{defi}
The quantity $N_{f}(E,\eta)$ can be thought of as a regularization of the sharp counting function
\begin{equation} \label{eq: sharp counting function}N_{\rm sharp}(E,\eta):=\#\big\{ j~:~ |E_j-E|<\eta\big\}.\end{equation}
We are interested in the fluctuations of $N_f(E',\eta)$, when $E'$ varies in an interval of energies $[E-\delta E,E+\delta E]$. We assume $h^2\ll\eta\ll \delta E\ll h$.\footnote{here $f\ll g$ means that the ratio $f/g$ goes to $0$ as $h\to 0$.} The assumption $h^2\ll \eta$ ensures that a window of width $\eta$ typically contains many eigenvalues (because the mean level spacing is of order $h^2$ by the Weyl law) thus we work in a mesoscopic regime. The assumption $\eta \ll \delta E$ ensures that the window $[E-\delta E,E+\delta E]$ incorporates many disjoint intervals of width $\eta$, which is necessary to compute statistics. The assumption $\delta E\ll h$ is rather technical and ensures that the Weyl term in the trace formula remains constant in the window $[E-\delta E,E+\delta E]$. We denote by $\mathbf P_{E'}$ the normalized Lebesgue measure on $[E-\delta E,E+\delta E]$, and by $\mathbf E_{E'}$ the associated expected value. A particularly interesting statistic is the \emph{spectral number variance}:
\[\mathbf E_{E'}\big[N_{f}(E',\eta)^2\big]-\mathbf E_{E'}\big[N_{f}(E',\eta)\big]^2.\]
This quantity measures the irregularities at scale $\eta$ of the distribution of eigenvalues of $P_h$ in the window $[E-\delta E,E+\delta E]$. A small spectral number variance is an indicator of \emph{rigidity} of the spectrum. This statistic --- with $N_f$ replaced by $N_{\rm sharp}$ --- was studied by Berry \cite{Berry1} using the Gutzwiller trace formula, in the case where the classical flow on the energy shell $E$ is chaotic, say Anosov. One of the goals of the present paper is to study spectral fluctuations of $N_f$ in a model of random Schr\"odinger operators that all share the same classical limit.

\begin{defi}Let $f$ be a real test function, with smooth compactly supported Fourier transform $\widehat f$. Set
	\[\Sigma_{\rm GOE}^2(f):=4\int_{\mathbf R_+} x |\widehat f(x)|^2\mathrm{d}x, \qquad \Sigma_{\rm GUE}^2(f):=2\int_{\mathbf R_+} x |\widehat f(x)|^2\mathrm{d}x. \]
\end{defi} 
These numbers correspond to the ensemble variance of the smoothed counting function of eigenvalues for the GOE and GUE ensembles of random matrices, which were computed by Dyson and Mehta \cite{DysonMehtaIV}. An avatar of the BGS conjecture is that when the classical dynamics on the energy shell $\Sigma_E:=p^{-1}(E)$ is chaotic, then, in the regime $h^2\ll \eta \ll \delta E\ll h$, the random variable $N_{f}(\cdot,\eta)$ should behave like a centered Gaussian random variable with variance $\Sigma_{\rm GOE}^2(f)$ or $\Sigma_{\rm GUE}^2(f)$ depending on whether or not the system is time-reversal invariant. For generic classically integrable systems, the Berry--Tabor conjecture \cite{BerryTabor} instead suggests that the variance of $N_f(\cdot,\eta)$ grows like $\frac{\eta}{h^2}$, which is the average number of energy levels in a window of width $\eta$.

\begin{rem} It is shown in \cite{AnttiHe} that if $X^N$ is a random Wigner matrix of size $N\times N$ (e.g. a matrix of the Gaussian Orthogonal Ensemble GOE) rescaled so that its empirical distribution of eigenvalues converges to the semicircle's law on $(-2,2)$, and $E\in (-2,2)$ is fixed, then for any $\eta$ of order $N^{-\alpha}$ with $\alpha\in (0,1)$, the \emph{ensemble variance} of the smoothed counting function
	\[\operatorname{Tr} f\Big(\frac{X^N-E}{\eta}\Big)=\sum_{E_j\in \operatorname{Spec}(X^N)}f\Big(\frac{E_j-E}{\eta}\Big)\]
	converges to $\Sigma_{\rm GOE}^2(f)$ as $N\to +\infty$.
This is a mesoscopic regime, where the window $[E-\eta,E+\eta]$ is small but still contains a large number of eigenvalues.
	
In random matrix theory, the phenomenon that, for a typical realization of a random matrix, the spectral fluctuations of the number of eigenvalues agree with the corresponding ensemble fluctuations is known as \emph{ergodicity} \cite{PANDEY1979170}.

\end{rem}

\subsubsection{Spectral fluctuations of random operators}

As the BGS conjecture is currently out of reach for a fixed operator $P_h$, we introduce small random perturbations of the operator $P_h$, that we denote $P_h^\omega$, in hope to obtain some results for \say{generic} perturbations. We will describe the random model in detail below, but the example to have in mind is $P_h^\omega=h^2\Delta_g+V(x)+h^\alpha V_\omega(x)$, with $V_\omega=\sum \omega_j V_j$, where the $V_j$'s are smooth potentials localized in balls of radius $h^\beta$ (with $2\beta<\alpha$) that cover the whole space $M$, and the $\omega_j$ are i.i.d. random real variables. The spectrum of $P_h^\omega$ is then a random point process $\{E_j^\omega\}_{j\ge 0}$, and the counting function \eqref{eq: counting func} becomes a random variable denoted $N_f^\omega(E,\eta)$. Our goal is to understand, for typical perturbations $V_\omega$, the fluctuations of $N_f^\omega(E',\eta)$ when $E'$ varies over $[E-\delta E,E+\delta E]$ . 

The random perturbations that we consider are somewhat \emph{ad hoc} but have the following advantage: they are too weak to modify the classical dynamics on the timescales that we consider, but are strong enough to affect the phases in the quantum evolution. 

\subsection{Description of the random operator} We now describe the family of operators that we will consider. We refer to the appendix of \cite{MoyLongTime} for the notions of semiclassical calculus used here. Fix a Weyl quantization procedure $\operatorname{Op}_h^{M}$ on $M$, that maps a symbol $a\in S^m(T^*M)$ to an operator $\operatorname{Op}_h^{M}(a):C^\infty(M)\to C^\infty(M)$. Let $p=p(x,\xi)\in S^m(T^*M)$ be a smooth elliptic real symbol of order $m>0$, that does not depend on $h$, and let $P_h=\operatorname{Op}_h^{M}(p)$.
\begin{assum}\label{assum: convex} The symbol $p$ is strictly convex in $\xi$, meaning that the Hessian $\nabla_\xi^2 p(x,\xi)$ is positive definite everywhere. Further, one has $p(x,\xi)=p(x,-\xi)$. \end{assum} 
The symmetry of $p$ implies that the Hamiltonian flow generated by $p$ is \emph{time-reversal invariant} under the transformation $(x,\xi)\mapsto (x,-\xi)$, see \S \ref{subsec: time reversal}. The convexity assumption will be used to ensure separation properties for periodic orbits of the classical flow, once projected on $M$.

\subsubsection{Dynamical assumptions} 
Fix $E\in \mathbf{R}$. We make the following dynamical assumption. Denote by $\Phi^t$ the Hamiltonian flow on $T^*M$ generated by $p$.
\begin{assum} \label{assum: anosov} The energy level $\Sigma_E:=p^{-1}(E)$ is regular, and the restriction of the Hamiltonian flow $\Phi^t$ to $\Sigma_E$ is Anosov. It means that there is a continuous flow invariant splitting
	\[T\Sigma_E=\mathbf{R}H_p\oplus E_u\oplus E_s,\]
	where
	\begin{itemize}
		\item $H_p$ is the Hamiltonian vector field associated with $p$.
		\item for all $v\in E_u$, for all $t\ge 0$, one has $\|\mathrm{d}\Phi^{-t}\cdot v\|\le C\mathrm{e}^{-Ct}\|v\|$.
		\item for all $v\in E_s$, for all $t\ge 0$, one has $\|\mathrm{d}\Phi^{t}\cdot v\|\le C\mathrm{e}^{-Ct}\|v\|$.
	\end{itemize}
	Here $\|\cdot\|$ is any continuous norm on $T(T^*M)$. The spaces $E_u$ and $E_s$ are called \emph{unstable} and \emph{stable}, and both have dimension $1$.
\end{assum} 

\subsubsection{Perturbations of $P_h$} We now describe the type of perturbations of $P_h$ that we consider. These are slightly more general than those allowed in \cite{ingremeau2023emergence}. In particular, we do not ask the perturbation to depend linearly on the random parameter $\omega$. 

\begin{assum}\label{assum: on alpha beta} We fix two parameters $\alpha,\beta\in [0,1)$ that satisfy
	\[0<\beta<\min(\alpha/2,2-2\alpha).\]\end{assum}
	We consider $h$-dependent real symbols $p_{\delta,\omega}+r_{\delta,\omega}\in S^{m}_\beta(T^*M)$ with $\delta\in [0,h^\alpha]$ and $\omega\in \Omega_h:=[-1,1]^{J_h}$, where $J_h$ is an $h$-dependent set.  The error $r_{\delta,\omega}$ will not affect the spectral statistics on the energy scales that we consider but should be included in the operator in order for the examples of \S \ref{subsec: examples} to fit into our setting. For example, the quantization of the symbol $|\xi|_g^2$ is not exactly the semiclassical Laplacian, and we need to add a lower order term $r\in h^2S^0(T^*M)$ to ensure that $h^2\Delta_g=\operatorname{Op}_h^{M}(|\xi|_g^2+r)$ up to a $\mathcal O(h^\infty)$ remainder.
	
The conditions on $p_{\delta,\omega}+r_{\delta,\omega}$ are quite complicated but have the advantage to include all the examples from \S \ref{subsec: examples} below. The basic example is $p_{\delta,\omega}+r_{\delta,\omega}=p+\delta\sum_{j\in J_h} \omega_jq_j$ where the $q_j$'s are smooth potentials localized on balls of radius $h^\beta$ that cover the whole space $M$.
\begin{assum} \label{assum: condition potential} The perturbations $p_{\delta,\omega}+r_{\delta,\omega}$ satisfy the following properties.
	\begin{enumerate}
		\item \emph{(symbol class of the perturbation)} One has $p_{0,\omega}=p$ for all $\omega\in \Omega_h$. For all $n\ge 1$, the symbol $\partial_\delta^n p_{\delta,\omega}$ belongs to the exotic class $S^m_\beta(T^*M)$, uniformly in $\delta\in [0,h^\alpha]$ and $\omega\in \Omega_h$. Moreover for all $n\ge 1$ and $j\in J_h$, the symbol $\partial_\delta^n \partial_{\omega_j} p_{\delta,\omega}$ belongs to $S^m_\beta(T^*M)$.
		\item \emph{(structure of the linearized perturbation)} Let $q_\omega=\partial_\delta p_{\delta,\omega}|_{\delta=0}$. Then, one has
		\[q_{\omega}=\sum_{j\in J_h} \omega_j q_j,\]
		where the $q_j$'s belong to $S^m_\beta(T^*M)$.
		\item  \emph{(localized bumps)} \label{assum: supp pot Vj} For any $j\in J_h$, the projection of $\operatorname{Supp}(\partial_{\omega_j} p_{\delta,\omega})$ to $M$ is contained in a ball $B(x_j,h^\beta)$, uniformly in $\delta\in [0,h^\alpha]$. Moreover, each $x\in M$ belongs at most to $\mathcal O(1)$ of the balls $B(x_j,h^\beta)$.
		\item \emph{(the bumps cover periodic orbits)} \label{assum: int Vj flow} There are constants $c_0,T_0,\mathfrak c>0$ such that for any periodic orbit $\gamma\subset \Sigma_E$ of the flow, with period $T_0\le T_\gamma\le \mathfrak c|\log h|$, one has
		\[\sum_{j\in J_h} \Big|\int_\gamma q_j\Big|\ge c_0T_\gamma.\]
		This is satisfied e.g. when $q_j\ge 0$ for all $j$ and $\sum_{j\in J_h}q_j\ge c_0$ on $\Sigma_E$.
		\item \emph{(symbol class of the remainder)} One has $r_{\delta,\omega}\in h^{2-2\beta}S_\beta^m(T^*M)$ uniformly in $\delta\in [0,h^\alpha]$ and $\omega\in [-1,1]^{J_h}$.
\end{enumerate}\end{assum}

\begin{defi}[Time-reversal symmetry breaking]\label{defi: preserve T sym} The linearized perturbation $q_\omega=\sum \omega_j q_j$ \emph{preserves time reversal symmetry} if $q_j(x,-\xi)=q_j(x,\xi)$ for all $j\in J_h$. The perturbation \emph{strongly breaks time reversal symmetry} if there are constants $c_0,\mathfrak c>0$ such that for any periodic orbit $\gamma\subset \Sigma_E$ of period $T_0\le T_\gamma\le \mathfrak c|\log h|$, one has
	\begin{equation} \label{eq: defi time sym break}\sum_{j\in J_h}\Big|\int_\gamma \big(q_j(x,\xi)-q_j(x,-\xi)\big)\Big|\ge c_0T_\gamma.\end{equation}
\end{defi}
For example, $q_\omega$ strongly breaks time reversal symmetry whenever $q_j(x,-\xi)=-q_j(x,\xi)$ and item \ref{assum: int Vj flow} above is satisfied.
\subsubsection{The random model} Let $m\in C^2_{\rm comp}((-1,1),\mathbf{R})$ be a compactly supported probability density. We endow the product space $\Omega_h:=[-1,1]^{J_h}$ with the product measure $\mathrm{d}\mathbf P_\omega= (\mathrm{d}m)^{\otimes J_h}$. One can then think of $(\omega_j)_{j\in J_h}$ as a family of i.i.d. real random variables with common density $m$.

We denote by $\mathbf P_\omega(A)$ the probability of an event $A$, and by $\mathbf E_\omega(F)$ the expected value of a function $F$ of $\omega$. Let $p_{\delta,\omega}+r_{\delta,\omega}$ satisfy Condition \ref{assum: condition potential}. The random operator is
\[P_{h,\delta}^\omega=\operatorname{Op}_h^{M}\big(p_{\delta,\omega}+r_{\delta,\omega}\big).\]

\subsection{Main results} Let $p$ satisfy Conditions \ref{assum: convex}--\ref{assum: anosov}. Let $(\alpha,\beta)$ satisfy Condition \ref{assum: on alpha beta}. Consider a family of perturbations $P_{h,\delta}^\omega=\operatorname{Op}_h^{M}(p_{\delta,\omega}+r_{\delta,\omega})$ satisfying Condition \ref{assum: condition potential}. We assume that the perturbations either preserve or break time reversal symmetry, in the sense of Definition \ref{defi: preserve T sym}.
 We let $\Sigma_{\rm GOE/\rm GUE}^2(f)$ be equal to $\Sigma_{\rm GOE}^2(f)$ in the first case, and $\Sigma_{\rm GUE}^2(f)$ in the second case. We take $\delta=h^\alpha$ and set
 \[P_h^\omega:=P_{h,h^\alpha}^\omega.\]
 For $E'\in [E-\delta E,E+\delta E]$, let 
 \[N(\omega,E'):=N_f^\omega(E',\eta)\] denote the smoothed counting function of eigenvalues of $P_{h}^\omega$ in $[E'-\eta,E'+\eta]$, which depends on $h$ and $\omega\in [-1,1]^{J_h}$. Also, let $\widetilde N(\omega,E')=N(\omega,E')-\mathbf E_{E'}(N(\omega,E'))$ denote the centered counting function.
 
 \subsubsection{Gaussian spectral fluctuations for generic perturbations} Our first result shows that for a generic perturbation $\omega$, the spectral fluctuations of $\widetilde N(\omega,E)$ obey a universal behavior. The result will be considerably strengthened in Theorem \ref{thm: conv random fun energy} below, but is simpler to state as is for now.
\begin{thm} \label{thm: as GOE fluctuations} There is a constant $\mathfrak c=\mathfrak c(\alpha,\beta)>0$ that depends additionally on the flow such that the following holds. Consider sequences $(h_n,\delta E_n,\eta_n)_{n\ge 0}$ that satisfy
	\[\frac{h_n}{\mathfrak c|\log h_n|}\le \eta_n\ll \delta E_n\ll h_n\ll 1,\qquad \text{as $n\to +\infty$}.\]
Then, there is a sequence of measurable subsets $B_n\subset \Omega_{h_n}$ such that
	\begin{itemize}
		\item $\mathbf P_\omega(B_n)\longrightarrow 1$.
		\item For any sequence $(\omega_n\in B_n)$, the sequence of centered random variables
		\[\widetilde N(\omega_n,\cdot):=N(\omega_n,\cdot)-\mathbf E_{E'}(N(\omega_n,\cdot)),\]
		defined on $[E-\delta E_n,E+\delta E_n]$, converges in law to a centered Gaussian random variable $\mathcal Z\sim \mathcal N(0,\Sigma_{\rm GOE/GUE}^2) $.
		\item The convergence in law is uniform in $(\omega_n\in B_n)$. For example, one has uniform convergence of the distribution function of $\widetilde N(\omega_n,\cdot)$ to $\mathcal Z$. More precisely,
		\[\sup_{\omega_n\in B_n} \sup_{x\in \mathbf R} \big|\mathbf P_{E'}(\widetilde N(\omega_n,E')\le x)-\mathbf P(\mathcal Z\le x)\big|\underset{n\to +\infty}{\longrightarrow} 0.\]
	\end{itemize}
\end{thm}
\begin{rem}Here is another way to frame the result. For each $\omega\in \Omega_h$, the function $\widetilde N(\omega,\cdot)$ is a random variable on $[E-\delta E,E+\delta E]$, with law $\mu_\omega:=\widetilde N(\omega,\cdot)_*\mathbf P_{E'}$. Then, $\mu_\omega$ is a random variable on $\Omega_h$ taking values in the space of probability measures on $\mathbf R$. Theorem \ref{thm: as GOE fluctuations} establishes that as $h\to 0$, $\mu_\omega$ converges in probability, for the weak topology on probability measures, to the centered Gaussian probability measure with variance $\Sigma_{\rm GOE/GUE}^2$.\end{rem} 
\begin{rem}
Theorem \ref{thm: as GOE fluctuations} is really saying something about (most) individual quantum systems in that we are computing energy fluctuations and not ensemble fluctuations. In this regard, it is close to the original conjecture of \cite{PhysRevLett.52.1}.\end{rem}

The regime $\frac{h}{\mathfrak c|\log h|}\ll \eta\ll h$ is arguably restrictive compared to the expectation of the BGS conjecture, which is that the results should hold in the full regime $h^2\ll \eta\ll h$. This time limitation is common in quantum chaos and is caused by the fact that we are unable to describe the semiclassical evolution of wave-packets beyond times of order $|\log h|$. More importantly, we are bound by the fact that we can only ensure separation of periodic orbits in phase space at scale $h^\beta$ (which is a key point in the proof) for orbits of period $\le \mathfrak c|\log h|$.
\subsubsection{Convergence to a Gaussian random field for generic perturbations} To state our second result, we introduce a new set of variables. Let
\[L=\frac{h}{\eta}, \qquad \Lambda =\frac{\delta E}{\eta}, \qquad \mu = \frac{E-E'}{\eta}.\]
The assumption $\frac{h}{\mathfrak c|\log h|}\le \eta\ll \delta E\ll h$ then becomes $1\ll \Lambda\ll L\le \mathfrak c|\log h|$, and $\mu$ ranges over $[-\Lambda,\Lambda]$ while $E'$ ranged over $[E-\delta E,E+\delta E]$. Then,
\[N_f(E',\eta)=N_f\big(E-\frac{\mu h}{L}, \frac{h}{L}\big)=N_{f_\mu}\big(E,\frac{h}{L}\big),\]
where $f_\mu=f(\cdot-\mu)$. We define
\begin{equation} \label{eq: defi N omega mu}N(\omega,\mu):=N_{f_\mu}\big(E,\frac{h}{L}\big), \qquad \omega\in [-1,1]^{J_h}, \quad \mu\in [-\Lambda,\Lambda].\end{equation}
The set $[-\Lambda,\Lambda]$ is endowed with the uniform probability measure $\mathbf P_\lambda$. We introduce the centered random variable
\[\widetilde N(\omega,\mu):=N(\omega,\mu)-\mathbf E_\lambda[N(\omega,\cdot)].\]
We fix $\omega$ and pick $\lambda\in [-\Lambda,\Lambda]$ at random. We fix another parameter $b>0$ and view the function $\widetilde N(\omega,\lambda+\cdot):\mu\in [-b,b]\to \mathbf{R}$ as a random function on $[-\Lambda,\Lambda]$. That is, we look at the behavior of the function $\widetilde N(\omega,\cdot)$ in a randomly chosen window $[\lambda-b,\lambda+b]$. The BGS conjecture predicts that this random function should behave universally according to random matrix theory. \medskip

Before stating our second result, we recall from \cite[\S 2.3]{ingremeau2023emergence} (see also the appendix of \cite{NazarovSodin2016}) some material on random smooth functions and Gaussian random fields. The space $C^\infty(\mathbf{R})$ is endowed with the topology of convergence of all derivatives on compact sets. That is, $f_n\rightarrow f$ if and only if
\[\forall \ell,k\ge 0, \qquad \|f_n-f\|_{C^\ell([-k,k])}\underset{n\to +\infty}{\longrightarrow}0.\]
With this topology, the space $C^\infty(\mathbf{R})$ is a separable Fr\'echet space. We let $\mathcal B(C^\infty(\mathbf{R}))$ denote the Borel $\sigma$-algebra on $C^\infty(\mathbf{R})$. 
A \emph{random smooth function} on $\mathbf{R}$ is a random variable taking values in $C^\infty(\mathbf{R})$.
\begin{defi}A (centered) \emph{smooth random Gaussian field} on $\mathbf{R}$ is a random smooth function $\mathfrak f:\mathbf{R}\to \mathbf{R}$ such that for any tuple $(x_1,\ldots,x_n)\in \mathbf{R}^n$, the random vector $(\mathfrak f(x_1),\ldots,\mathfrak f(x_n))$ follows a (centered) Gaussian law. The \emph{covariance kernel} is defined by $K(x,y)=\operatorname{Cov}(\mathfrak f(x),\mathfrak f(y))$. If $\mathfrak f$ is centered, then the law of the random vector $(\mathfrak f(x_1),\ldots,\mathfrak f(x_n))$ is given by $\mathcal N(0,A)$, where $A=\big(K(x_i,x_j)\big)_{1\le i,j\le n}$. In particular, the covariance kernel determines the law of $\mathfrak f$. \end{defi}

\begin{defi}A sequence of random smooth functions $(\mathfrak f_n)_{n\ge 1}$ converges in law, or in distribution, to a smooth Gaussian field $\mathfrak f$ if for any bounded continuous functional $F:C^\infty(\mathbf{R})\to \mathbf{R}$, one has
\[\mathbf E[F(\mathfrak f_n)]\underset{n\to +\infty}{\longrightarrow}\mathbf E[F(\mathfrak f)].\]\end{defi}
Typical examples of functional $F$ are $F(f)=\int_{\mathbf{R}}\chi(x,f(x),\ldots,f^{(k)}(x))\mathrm{d}x$, where $\chi$ is a smooth compactly supported function on $\mathbf{R}^{k+1}$, or $F(f)=\varphi(f(x_0))$ for some bounded function $\varphi\in C^\infty(\mathbf{R})$ and $x_0\in \mathbf{R}$. \medskip

Our second result establishes the convergence in law, for generic choices of the parameter $\omega$, of the random function $\widetilde N(\omega,\lambda+\cdot)$ to a smooth Gaussian field with universal covariance.

\begin{thm}[Convergence to a Gaussian random field]\label{thm: conv random fun energy} There exists a constant $\mathfrak c=\mathfrak c(\alpha,\beta)>0$ that depends additionally on the flow such that the following holds. Assume $1\ll b_n\le \Lambda_n\ll L_n\le \mathfrak c|\log h_n|$ as $n\to +\infty$. Then, there exists a sequence of measurable subsets $B_n\subset \Omega_{h_n}$ such that $\mathbf P_{\omega}(B_n)\longrightarrow 1$, and for any sequence $(\omega_n\in B_n)$, the sequence of random functions 
  \[\widetilde N(\omega_n,\lambda+\cdot):\mu\in [-b_n,b_n]\to \mathbf{R}\]
  obtained by picking $\lambda\in [-\Lambda_n,\Lambda_n]$ uniformly at random, converges in distribution to a smooth Gaussian field $\mathfrak f_{\rm GOE/GUE}$ with covariance kernel 
  \[K_{\rm GOE}(\mu,\mu'):=4\int_{\mathbf{R}_+} t|\widehat f(t)|^2 \cos(t|\mu-\mu'|)\mathrm{d}t\]
  	if the perturbations preserve time reversal symmetry, and~$K_{\rm GUE}(\mu,\mu')=\frac{1}{2} K_{\rm GOE}(\mu,\mu')$ if the perturbations strongly break time reversal symmetry.
\end{thm}

\begin{rem}We make a few comments.
	\begin{itemize}
		\item Strictly speaking, in order to deal with functions defined on the whole real line, one should multiply $\widetilde N(\omega_n,\lambda+\cdot)$ by $\chi(\mu/b_n)$ where $\chi$ is a smooth function satisfying $\chi\equiv 1$ on $[-1/2,1/2]$ and $\chi\equiv 0$ outside $[-1,1]$.
		\item  Theorem \ref{thm: as GOE fluctuations} basically computed the on-diagonal kernel $K_{\rm GOE}(\mu,\mu)$ only. Theorem \ref{thm: conv random fun energy} greatly improves upon this by describing the correlations between the values of $\widetilde N(\omega,\lambda+\mu)$ and $\widetilde N(\omega,\lambda+\mu')$ for $\mu\neq \mu'$, as well as correlations of any order.
\item The kernel $K_{\rm GOE}(\mu,\mu')$ depends only on the distance $|\mu-\mu'|$, that is, the limit law is translation invariant. We write $K_{\rm GOE}(\mu,\mu')=k_{\rm GOE}(|\mu-\mu'|)$. 
\item Assuming $\widehat f(0)\neq 0$, one can show by performing two integrations by parts that
\[k_{\rm GOE}(\alpha)\underset{\alpha\to +\infty}\sim -\frac{1}{\alpha^2}\Big(\frac{1}{\pi}\int_{\mathbf{R}}f(x)\mathrm{d}x\Big)^2.\]
In particular $k_{\rm GOE}(\alpha)$ becomes negative for large values of $\alpha$.  This reflects repulsion between eigenvalues: when $|E_1-E_2|/\eta$ is too large, the number of eigenvalues in $[E_1-\eta,E_2+\eta]$ is negatively correlated with the number of eigenvalues in $[E_2-\eta,E_2+\eta]$, which is also a key feature of the Gaussian ensembles of matrices. Although Theorem \ref{thm: conv random fun energy} is restricted to specific test functions $f$, it is instructive to look at the shape of the kernel $k_{\rm GOE}$ for the sharp window $f=\mathbf 1_{[-1,1]}$. Then $k_{\rm GOE}(\alpha)$ measures the correlation between the number of eigenvalues in windows $[E_1-\eta,E_1+\eta]$ and $[E_2-\eta,E_2+\eta]$ for $|E_2-E_1|=\alpha \eta$, and can be explicitly computed as an improper integral:
\[k_{\rm GOE}(\alpha)=\frac{1}{\pi^2}\log\Big|\frac{\alpha^2-4}{\alpha^2}\Big|.\]
Note that $k_{\rm GOE}$ is decreasing on $[0,2]$ and increasing on $[2,+\infty]$, with $k_{\rm GOE}(0)=+\infty$, $k_{\rm GOE}(2)=-\infty$, and $k_{\rm GOE}(+\infty)=0$. We interpret this as follows. For small $\alpha$, the number of eigenvalues in the window $[\lambda-\eta,\lambda+\eta]$ is positively correlated with the number of eigenvalues in the translated window $[\lambda+\alpha \eta-\eta,\lambda+\alpha\eta+\eta]$, which should not be surprising because the two windows have a large overlap. Then $k_{\rm GOE}$ becomes negative and at $\alpha=2$ the two windows stop overlapping. The equality $k_{\rm GOE}(2)=-\infty$ means that the number of eigenvalues in $[\lambda-\eta,\lambda+\eta]$ is very negatively correlated with the number of eigenvalues in the adjacent window $[\lambda+\eta,\lambda+3\eta]$. For larger values of $\alpha$, correlations are still negative, although they tend to $0$ as $\alpha\to +\infty$.
\end{itemize}\end{rem}

\subsection{Ensemble fluctuations of the counting function} In Theorem \ref{thm: as GOE fluctuations} and Theorem \ref{thm: conv random fun energy}, we fixed $\omega$ and studied the fluctuations of the counting functions as the energy varies. We now adopt a different (and arguably, less interesting) point of view where $\omega$ is picked at random, and $\widetilde N(\omega,\cdot)$ is viewed as a random function on $[-\Lambda,\Lambda]$. 

The next result is analogous to Theorem \ref{thm: conv random fun energy}, where the operator is picked at random instead of the energy window.

\begin{thm}[Convergence of the sequence of random functions $\widetilde N(\omega,\cdot)$]\label{thm: convergence random function} There exists a constant $\mathfrak c=\mathfrak c(\alpha,\beta)>0$ that depends additionally on the flow such that the following holds. Assume $1\ll \Lambda_n\ll L_n\le \mathfrak c|\log h_n|$ as $n\to +\infty$. Replace $\widetilde N(\omega,\mu)$ by $\chi(\mu/\Lambda)\widetilde N(\omega,\mu)$ where $\chi$ is a smooth function equal to $1$ on $[-1/2,1/2]$ and vanishing outside $[-1,1]$. Then, as $n\to +\infty$, the sequence of random functions $\widetilde N(\omega,\cdot):[-\Lambda_n,\Lambda_n]\to \mathbf{R}$ (where $\omega$ is picked at random in $\Omega_{h_n}$) converges in distribution to a smooth Gaussian field with covariance kernel $K_{\rm GOE}(\mu,\mu')$ if the perturbations preserve time reversal symmetry, and $K_{\rm GUE}(\mu,\mu')=\frac{1}{2} K_{\rm GOE}(\mu,\mu')$ if the perturbations strongly break time reversal symmetry.
\end{thm}
\subsection{Examples} \label{subsec: examples} We describe three models that all fit in the above picture. The first one was studied in \cite{ingremeau2023emergence} in the case $p=|\xi|_g^2$ on a negatively curved manifold, where the authors investigated the long-time evolution of Lagrangian states by the quantum dynamics generated by $h^2\Delta_g+\delta q_\omega(x)$.

\subsubsection{Random electric potential model} 
We consider a random Schr\"odinger operator 
\begin{equation} \label{tag: random potential} P_{h,\delta}^\omega=h^2\Delta_g+V(x)+\delta q_\omega(x), \tag{random electric potential}\end{equation}
where $q_\omega(x)=\sum_{j\in J_h} \omega_j q_j(x)$ is a sum of localized random bumps satisfying Condition \ref{assum: condition potential}. Then we know that
\[P_{h,\delta}^\omega=\operatorname{Op}_h^{M}(|\xi|_g^2+V(x)+\delta q_\omega +h^2 r)+\mathcal O(h^\infty)_{L^2\to L^2}\]
with $r\in S^0(T^*M)$. We let $p=|\xi|_g^2+V(x)$, $p_{\delta,\omega}=p+\delta q_\omega$ and $r_{\delta,\omega}=h^2 r$. These perturbations preserve time reversal symmetry.

\subsubsection{Random metric model}
We consider a conformal metric $g_{\varphi}=\mathrm{e}^{\varphi}g$ with $\varphi\in C^\infty(M)$. Since we work in dimension $2$, the Laplacian $\Delta_{g_\varphi}$ acting on $L^2(M,\mathrm{dvol}_{g_\varphi})$ is unitarily conjugated to the operator $\mathrm{e}^{-\varphi/2}\Delta_g\mathrm{e}^{-\varphi/2}$ acting on $L^2(M,\mathrm{dvol}_g)$, therefore they have the same spectrum. We consider a random conformal factor $\varphi_{\delta,\omega}(x)=\delta\sum_{j\in J_h}\omega_j q_j(x)$ and consider the random operator
\begin{equation} \label{tag: random metric} P_{h,\delta}^\omega=h^2\big(\mathrm{e}^{-\frac{\varphi_{\delta,\omega}}{2}}\Delta_g\mathrm{e}^{-\frac{\varphi_{\delta,\omega}}{2}}\big). \tag{random metric} \end{equation}
Using the results from the appendix of \cite{MoyLongTime}, one can check that
\[P_{h,\delta}^\omega=\operatorname{Op}_h^{M}(\mathrm{e}^{-\varphi_{\delta,\omega}}|\xi|_g^2+ r_{\delta,\omega})+\mathcal O(h^\infty)_{L^2\to L^2},\]
with $r_{\delta,\omega}\in h^{2-2\beta}S_\beta^0(T^*M)$ uniformly in $\delta\in [0,h^\alpha]$ and $\omega\in [-1,1]^{J_h}$. This fits into the above assumptions with $p_{\delta,\omega}=\mathrm{e}^{-\varphi_{\delta,\omega}}|\xi|_g^2$. These perturbations preserve time-reversal symmetry.

\subsubsection{Random magnetic potential model} Let $(U_i,f_i)_{i\in I}$ be a finite atlas of $M$, and let $(\chi_i)_{i\in I}$ be a subordinated partition of unity.  Let $(q_j')_{j\in J_h'}$ satisfy Condition \ref{assum: condition potential} (with $J_h',q_j'$ in place of $J_h,q_j$). Let $(\mathbf e_1,\mathbf e_2)$ denote the standard basis of $T^*\mathbf{R}^2$. Let $\{\omega_{k,i,j}\}_{k\in\{1,2\},i\in I,j\in J_h'}$ be i.i.d. real random variables, and define
\[\mathbf A_\omega=\sum_{i\in I} \chi_i f_i^*\Big(\sum_{j\in J_h'}q_j'(x)\big( \omega_{1,i,j}\mathbf e_1+\omega_{2,i,j}\mathbf e_2\big)\Big).\]
We set $J_h=\{1,2\}\times I\times J_h'$, and $q_{\bm j}(x,\xi)=-2q_j' \chi_i \langle \xi, f_i^*\mathbf e_k\rangle$ for $\bm j=(k,i,j)\in J_h$. We define a random magnetic operator by
\begin{equation}\label{tag: magnetic} P_{h,\delta}^\omega=(\nabla-\mathrm{i}h\delta\mathbf A_\omega)^*(\nabla-\mathrm{i}h\delta\mathbf A_\omega)+V(x).\tag{random magnetic potential}\end{equation}
One can check using the results from the appendix of \cite{MoyLongTime} that
\[P_{h,\delta}^\omega=\operatorname{Op}_h^{M}\big(|\xi-\delta\mathbf A_\omega(x)|_g^2+V(x)+r_{\delta,\omega}\big)+\mathcal O(h^\infty)_{L^2\to L^2},\]
with $r_{\delta,\omega}\in h^2h^{-2\beta}S^0_\beta(T^*M)$. We set $p_{\delta,\omega}=|\xi-\delta\mathbf A_\omega(x)|_g^2+V(x)$. Note that
\[q_\omega(x,\xi)=\partial_\delta p_{\delta,\omega}\big|_{\delta=0}=-2\langle \xi,\mathbf A_\omega(x)\rangle_g=\sum_{\bm{j}\in J_h} \omega_{\bm j}q_{\bm j}.\]
We will show in Proposition \ref{prop: random mag checks assumptions} that $p_{\delta,\omega}$ satisfies Condition \ref{assum: condition potential}. Since $q_{\bm j}(x,\xi)-q_{\bm j}(x,-\xi)=2q_{\bm j}(x,\xi)$, these perturbations strongly break time reversal symmetry.

\subsubsection{Other models} One can build many other models of random operators that preserve or break time reversal symmetry. For example, one could use the model of random semiclassical pseudodifferential operators ($\Psi$DOs) studied by Ingremeau and Vogel \cite{ingremeau2023emergence}. These operators are \say{more random} than those described above as they introduce randomness in the $\xi$ variable too. Indeed, in the random electric potential/magnetic potential/metric models described above, the random perturbation $q_\omega(x,\xi)$ is correlated with $q_{\omega}(x',\xi')$ if and only if $d(x,x')\lesssim h^{\beta}$. In the random $\Psi$DO model, $q_\omega(x,\xi)$ is correlated with $q_\omega(x',\xi')$ if and only if $d(x,x')\lesssim h^\beta$ \emph{and} $d(\xi,\xi')\lesssim h^\beta$.

\subsection{Related works}

	\subsubsection{Propagation of Lagrangian states by random perturbations of Schr\"odinger operators} The model studied in the present paper was introduced by Ingremeau and Vogel in \cite{ingremeau2023emergence}, where they suggested that it is a good toy model to study quantum chaos. Their object of interest is the long-time evolution of a monochromatic Lagrangian state $f_h=\mathrm{e}^{\frac{\mathrm{i}}{h} \phi}a$ under the perturbed quantum dynamics, that is the function $\mathrm{e}^{-\frac{\mathrm{i}}{h} tP_h^\omega}f_h(x)$. There are two notions of randomness in their work: one can pick $\omega\in [-1,1]^{J_h}$ at random, or pick $x\in M$ at random. The authors show that for a generic perturbation $\omega$, the random function $\mathrm{e}^{-\frac{\mathrm{i}}{h} tP_h^\omega}f_h(\exp_x(\cdot))$ (where $x$ is picked at random in $M$) converges in law to a Gaussian random field after suitable rescaling \cite[Theorem 2.10]{ingremeau2023emergence}. This is analogous to Theorems \ref{thm: as GOE fluctuations}-\ref{thm: conv random fun energy}. They also show that for almost all $x\in M$, the random function $\mathrm{e}^{-\frac{\mathrm{i}}{h} tP_h^\omega}f_h$ (where $\omega$ is picked at random in $[-1,1]^{J_h}$), after suitable rescaling to become a function on $T_xM\simeq \mathbf{R}^d$, converges in law to a monochromatic Gaussian free field \cite[Theorem 2.11]{ingremeau2023emergence}. This is analogous to Theorem \ref{thm: convergence random function}.
	
	In their work, Ingremeau and Vogel modify the random variables that enter in play in order to obtain independence and thereby apply a multivariate central limit Theorem. Our approach avoids this step and directly applies the method of moments to establish convergence in distribution. This allows us to work in the full range $0<\beta<\min(\alpha/2,2-2\alpha)$, while \cite{ingremeau2023emergence} was constrained to a smaller range of parameters in the random potential model.
	
	In \cite{ingremeau2025improvedlinftyboundseigenfunctions}, they obtained improved $L^\infty$ upper bounds on eigenfunctions for the conjugated operator $\widetilde P_h:= \mathrm{e}^{-\frac{\mathrm{i}}{h}P_h^\omega}P_h \mathrm{e}^{\frac{\mathrm{i}}{h}P_h^\omega}$, in the case where $P_h=h^2\Delta_g$ on a compact negatively curved manifold.

\subsubsection{Perturbation of quantized diffeomorphisms of the torus} The precise meaning of \say{generic} in the BGS conjecture is unclear, but it is necessary in order to discard exceptional systems for which BGS is known to fail. The idea of studying BGS by averaging over an ensemble of quantum systems that share the same classical limit was suggested by Zirnbauer \cite{zirnbauer1999pair}. He considered the $2$-point correlation function for perturbations of Berezin--Toeplitz operators that quantize chaotic diffeomorphism of the torus. If $U_N$ quantizes a chaotic map $f:\mathbf T^2\to \mathbf T^2$ (here $N=h^{-1}$ and one works in the limit $N\to +\infty$), the perturbations take the form
\[U_N^\omega=\exp\Big(\frac{\mathrm{i}}{h}\sum_{0\le k\le s} h^\alpha\xi_k Q_{k,N}\Big)U_N,\]
where $s$ is a fixed integer, $\alpha\in (0,1)$, $(\xi_k)_{0\le k\le s}$ are i.i.d. Gaussian random variables and the operators $Q_{k,N}$ are quantizations of Hamiltonians $q_k$ on $\mathbf T^2$. This model is similar to the one considered in the present paper, except that we let $s\sim h^{-2\beta}$ grow as $h\to 0$, and the functions $q_k$ are supported in balls of radius $h^\beta$. Since there are available semiclassical and classical trace formulae for quantized diffeomorphisms of the torus, see e.g. \cite{Faure2007}, it is likely that Theorems \ref{thm: as GOE fluctuations}--\ref{thm: convergence random function} can be adapted to this setting.

\subsubsection{Random hyperbolic surfaces} The development of new integration techniques on moduli space of random surfaces in the recent years has made possible the study of spectral statistics for models of random surfaces. Rudnick \cite{rudnick2022goe} considered the \emph{ensemble variance} of the smoothed counting function \eqref{eq: counting func} for eigenvalues of the operator $\sqrt{\Delta-1/4}$ on the moduli space $\mathcal M_g$ of hyperbolic surfaces of genus $g$. In the iterated limit $g\to +\infty$ then $L\to +\infty$, this variance converges to an expression consistent with GOE. Rudnick--Wigman \cite{rudnick2023sure} showed that for a typical surface $M\in \mathcal M_g$, the spectral variance of the counting function matches the ensemble variance, in the iterated limit $g\to +\infty$ then $L\to +\infty$. This is analogous to the result of Proposition \ref{prop: moments N_osc}, for $n=2$. Rudnick--Wigman \cite{rudnick2023central} showed that the oscillating term in the trace formula converges in law to a centered Gaussian random variable with variance $\Sigma_{\rm GOE}^2$, in the iterated limit $g\to +\infty$ then $L\to +\infty$. This is analogous to a special case of Theorem \ref{thm: convergence random function} where the considered functionals $F:C^\infty(\mathbf{R})\to \mathbf{R}$ are of the form $F(f)=\varphi(f(0))$ with $\varphi\in C^\infty_{\rm comp}(\mathbf{R})$. In the recent preprint \cite{meriaz2026localweyllawlengthminimising}, Meriaz computed the variance of the local Weyl law over the moduli space of \emph{marked hyperbolic surfaces} of genus $g$. His results are consistent with the predictions of Berry's random wave model for eigenfunctions. In our setting, letting $(\psi_j^\omega)_j$ denote an orthonormal basis of eigenfunctions of $P_h^\omega$, it would also be interesting to obtain information about the fluctuations of the smoothed local density of states
\[\sum_j f\Big(\frac{E-E_j^\omega}{\eta}\Big) |\psi_j^\omega(x)|^2,\]
either as $E,x$, or $\omega$ varies.

There have been parallel studies in the model of random covers of a closed hyperbolic surface introduced by Magee--Naud--Puder \cite{mageepudernaud,magee_puder_2023}. The limit $g\to +\infty$ is replaced by $n\to +\infty$, where $n$ is the degree of the covering surface. The analogue of \cite{rudnick2022goe} was obtained by Naud \cite{naud2022random}, who was also able to incorporate unitary twists that break time-reversal invariance in order to observe a behavior consistent with GUE --- see Marklof--Monk \cite{marklof2024modulispacetwistedlaplacians} for a similar result in the Weil--Petersson model. The analogue of the results from \cite{rudnick2023central,rudnick2023sure} were obtained by Maoz \cite{Maoz1,Maoz2}. All the above results were adapted to random covers of surfaces of \emph{variable} negative curvature by the author in \cite{Moy2024}.

\subsection{Extension to higher dimension} There are analogues of Theorems \ref{thm: as GOE fluctuations} and \ref{thm: convergence random function} in dimension $\ge 3$, but the results are nicer to express in dimension $2$. Indeed, in dimension $d\ge 3$, by the trace formula of \cite[Theorem 1.1]{MoyLongTime} applied to $f_\mu=f(\cdot-\mu)$, the Weyl term in the trace formula \eqref{eq: GUTZ} below is given by
\[N_{\rm Weyl}=\frac{1}{h^{d-1}L}\sum_{k\le N}\sum_{i+j=k}c_{k,i,j,\delta} L^{-k}\mu^i\widehat f^{(j)}(0)+\mathcal O_f(h^{N(1-2\beta)-C_d}(1+|\mu|)^{2N+C_d}).\]
The coefficients $c_{k,i,j,\delta}$ satisfy $|c_{k,i,j,\delta}|\lesssim h^{\lfloor \frac{k+1}{2}\rfloor}$. In dimension $d=2$, for $|\mu|\le \Lambda\le L\le \mathfrak c|\log h|$ one can simplify
\[N_{\rm Weyl}=\frac{C_\omega}{hL}\widehat f(0)+\mathcal O\Big(\frac{\Lambda}{L}\Big),\]
where we denoted $C_\omega=c_{0,0,0,\delta}$. In particular, the Weyl term is independent of the spectral window centered at $E-\mu\frac{h}{L}$, up to a small error. Consequently, subtracting $\mathbf E_\mu[N]$ from the counting function allows suppressing the Weyl term. In higher dimensions, since we are constrained to the regime $L\le \mathfrak c|\log h|$, the term $N_{\rm Weyl}$ depends polynomially on $\mu$ --- note that if we could reach extremely long times $L\sim h^{-(d-1)}$, the Weyl term would still be independent of $\mu$, but here this is not the case.

To suppress the Weyl term, we consider the polynomial function $\overline N(\omega,\cdot)$ defined on $[-\Lambda,\Lambda]$, of degree $\le C_d$ large enough depending on the dimension (when $d=2$ one can take $C_d=0$) that minimizes the $L^2$-norm
\[\mathbf E_\mu[|N(\omega,\mu)-\overline N(\omega,\mu)|^2].\]
The polynomial function $\overline N(\omega,\cdot)$ encodes the slowly varying fluctuations of the counting function $N(\omega,\cdot)$ --- this is somehow analogous to the \emph{local mean density} considered by Berry \cite{BerryRigidity,Berry1}---, and one is really interested in the difference
\[\widetilde N=N-\overline N,\]
which encodes the small scale behavior of the counting function.

I expect that Theorems \ref{thm: as GOE fluctuations} and \ref{thm: convergence random function} still hold for this function $\widetilde N$; the proof is not modified once the Weyl term has been dealt with.

\subsection{Overview of the proof} 

The proof of Theorem \ref{thm: as GOE fluctuations} relies on the semiclassical trace formula for the operator $P_{h}^\omega$ applied to the test function $f_\mu:=f(\cdot-\mu)$, that we proved in \cite{MoyLongTime}.
\begin{prop}[{\cite[Theorem 1.1]{MoyLongTime}}]\label{prop: long time gutz} Assume $M$ has dimension $2$. Assume that Conditions \ref{assum: convex}--\ref{assum: condition potential} are satisfied. Let $\omega\in [-1,1]^{J_h}$ and $\mu\in [-\Lambda,\Lambda]$. Then, for $1\le L\le \mathfrak c(\alpha,\beta)|\log h|$, the smoothed counting function from \eqref{eq: defi N omega mu} satisfies
	\begin{equation} \label{eq: GUTZ}N(\omega,\mu)= \frac{C_\omega}{hL}\widehat f(0)+N_{\rm osc}(\omega,\mu)+R(\omega,\mu),\end{equation}
where
\begin{itemize}
	\item $C_\omega$ is a constant depending on $h,\omega$ but not on $\mu$, that converges as $h\to 0$ to the Liouville measure of the energy shell $p^{-1}(E)$.
	\item $N_{\rm osc}=N_{\rm osc}(\omega,\mu)$ depends on $\omega$ and $\mu$, and is given by
	\begin{equation} \label{eq: Nosc in traceformula}N_{\rm osc}={\frac{1}{L}\sum_{\gamma\in \mathcal P}\sum_{\pm} \frac{T_\gamma^\sharp}{|\det(I-P_\gamma)|^{1/2}} \mathrm{e}^{\pm\frac{\mathrm{i}}{h} S_\omega(\gamma)}\mathrm{e}^{\pm\mathrm{i}\mu T_\gamma/L}\widehat f\Big(\pm\frac{T_\gamma}{L}\Big)}.\end{equation}
	Here:
	\begin{itemize}
		\item $\mathcal P$ is the set of periodic orbits of the unperturbed flow $\Phi^t:p^{-1}(E)\to p^{-1}(E)$.
		\item $T_\gamma$ and $T_\gamma^\sharp$ denote the period and primitive period of a periodic orbit $\gamma$, respectively.
		\item $P_\gamma:E_u\oplus E_s\to E_u\oplus E_s$ is the linearized Poincar\'e map of the orbit, whose conjugacy class is independent of the point of the orbit.
		\item $S_\omega(\gamma)$ is the action of a perturbed orbit, given by
		\[S_\omega(\gamma)=S(\gamma)-h^{\alpha}\int_{\gamma} q_\omega+r_\omega(\gamma),\]
		where $S(\gamma)$ is the action of $\gamma$ and the error satisfies 
		\begin{equation} \label{eq: bound derivee r omega} |\partial_{\omega_i}r_{\omega}(\gamma)|\lesssim h^{2\alpha} h^{-\beta} \mathrm{e}^{CT_\gamma}\end{equation} 
		for some constant $C>0$.
	\end{itemize} 
	\item $R=R(\omega,\mu)$ is bounded uniformly in $\omega,\mu$ by
	\[|R(\omega,\mu)|\lesssim \frac{\Lambda}{L}.\]
\end{itemize} \end{prop}

\begin{rem}Due to the particular structure of the Hamiltonian $p$, there are no Maslov indices in the trace formula, by a result of Paternain--Paternain \cite{Paternain1994}.\end{rem}
\begin{rem}Since the constant $C_\omega$ in \eqref{eq: GUTZ} is independent of $\mu$, it will vanish when computing the variance of $N(\omega,\mu)$ over $\mu\in [-\Lambda,\Lambda]$. \end{rem}
The bound \eqref{eq: bound derivee r omega} does not appear in \cite{MoyLongTime} but will be proved below in Proposition \ref{prop: bounds derivatives omega r}. Theorem \ref{thm: as GOE fluctuations} will then follow from the next proposition. 

\begin{prop}[$L^2_\omega$-convergence of the $\mu$-moments of $N_{\rm osc}$] \label{prop: moments N_osc}  Denote by $(m_n)_{n\ge 1}$ the moments of a centered Gaussian random variable with variance $\Sigma_{\rm GOE/GUE}^2$ depending on whether the perturbations preserve or break time reversal symmetry, which are given by
	\[m_n:=\left\{\begin{array}{ll} 0 & \text{if $n$ is odd,} \\ \frac{(2p)!}{2^pp!}\Sigma_{\rm GOE/GUE}^{2p} & \text{if $n=2p$ is even.}\end{array}\right.\]
	 Assume $1\le \Lambda \le L\le \mathfrak c|\log h|$ for some small enough constant $\mathfrak c(\alpha,\beta)>0$. Let $N_{\rm osc}$ be as in \eqref{eq: Nosc in traceformula}. Then, for all $n\ge 1$, one has
	\[\mathbf E_\omega(|\mathbf E_\mu(N_{\rm osc}^n)-m_n|^2)=\mathcal O_n(\Lambda^{-1}).\]
\end{prop}
After applying the Markov inequality, Proposition \ref{prop: moments N_osc} tells us that for most $\omega\in \Omega_h$, the $\mu$-moments of $N_{\rm osc}$ approach the moments of a centered Gaussian random variable, which in turn allows showing convergence in law of $N_{\rm osc}$ to a Gaussian along most sequences $(\omega_n\in \Omega_{h_n})$.

\subsubsection{The diagonal approximation} Let us explain informally how to prove a basic version of Theorem \ref{thm: convergence random function}, by studying the $\omega$-variance of $N_{\rm osc}(\omega,0)$, as this already contains the mechanisms that allows proving Proposition \ref{prop: moments N_osc}. Suppose for simplicity that the test function $f$ is even and nonnegative. Introducing the shorthand $A_\gamma:=\frac{T_\gamma^\sharp}{|\det(I-P_\gamma)|^{1/2}}\widehat f\Big(\frac{T_\gamma}{L}\Big)$, one can rewrite
\[N_{\rm osc}(\omega,0)=\frac{2}{L}\sum_{\gamma\in \mathcal P} A_\gamma \cos(h^{-1}S_\omega(\gamma)).\] 
Then, the $\omega$-average of the square of the oscillating part is given by
\begin{equation} \label{eq: square N osc}\mathbf E_\omega\big[|N_{\rm osc}(\omega,0)|^2\big]=\frac{4}{L^2}\sum_{(\gamma,\gamma')} A_\gamma A_{\gamma'} \mathbf E_\omega\big[\cos(h^{-1}S_\omega(\gamma))\cos(h^{-1}S_\omega(\gamma'))\big].\end{equation}
The random perturbations are designed so that different periodic orbits see different bumps (unless one orbit is obtained by reversing the other). This relies on separation properties of periodic orbits in the position space, see \S \ref{subsec: separation}.  In particular, if $\gamma'\notin \{\gamma,\gamma^{-1}\}$ (here $\gamma^{-1}$ denotes the time reversed orbit), then $\int_\gamma q_\omega$ is almost uncorrelated with $\int_{\gamma'} q_\omega$ and $\mathbf E_\omega\big[\cos(h^{-1}S_\omega(\gamma))\cos(h^{-1}S_\omega(\gamma'))\big]\approx 0$. This provides a probabilistic mechanism to suppress nondiagonal contributions. Since $\gamma$ and $\gamma^{-1}$ have the same contribution, one ends up with
\[\mathbf E_\omega\big[|N_{\rm osc}(\omega,0)|^2\big]\approx\frac{8}{L^2}\sum_{\gamma\in \mathcal P} A_\gamma^2 \mathbf E_\omega\big[\cos(h^{-1}S_\omega(\gamma))^2\big],\]
where we have gained a factor $2$ compared to \eqref{eq: square N osc}. The cosines squared average to $\frac 12$, and by using equidistribution results for periodic orbits from \cite[\S 5]{Moy2024} (which are analogous to the sum rule of Hannay--Ozorio de Almeida \cite{JHHannay1984}), one finds
\[\mathbf E_\omega\big[|N_{\rm osc}(\omega,0)|^2\big]\approx \frac{4}{L^2}\sum_{\gamma\in \mathcal P} A_\gamma^2\underset{L\to +\infty}{\longrightarrow} \Sigma_{\rm GOE}^2(f).\]
Similar cancellation mechanisms were exploited by Ingremeau and Vogel to study the long time evolution of plane waves by the operator $\mathrm{e}^{-\frac{\mathrm{i}}{h} tP_{h}^\omega}$, see \cite{ingremeau2023emergence}. Informally, they relied on the fact that on a negatively curved surface $M$, for most points $(x,\xi)\in S^*M$, the geodesic segment of length $T$ issuing from $(x,\xi)$ does not pass $h^\beta$-close to itself on the base too often, when $T\le \mathfrak c(\beta)|\log h|$.

\subsection{Notation}  We use the letters $C,c$ to denote positive constants that may change from line to line. If the constant depends on a relevant set of parameters $\bullet$, we denote it by $C_\bullet$. If $f,g$ are two nonnegative quantities, we write $f\lesssim g$ if there is a constant $C>0$ such that $f\le Cg$. We write $f=\mathcal O(g)$ if $|f|\lesssim |g|$.

\subsection*{Acknowledgments} I am grateful to Maxime Ingremeau and Martin Vogel for discussions around this work. I thank St\'ephane Nonnenmacher for his guidance during the writing of this paper, and his feedback on a first version of this manuscript.
\section{Dynamical preliminaries} In this section, we first recall some equidistribution results for periodic orbits of the flow $\Phi^t:\Sigma_E\to \Sigma_E$. Then, we discuss time reversal symmetry before proving some quantitative estimates on the separation of periodic orbits in phase space and position space. These will be used in the next section to obtain decorrelation results for the actions of distinct periodic orbits of the perturbed classical dynamics.

\subsection{Equidistribution results} Let $p$ be a Hamiltonian on $T^*M$ satisfying Conditions \ref{assum: convex}--\ref{assum: anosov}. We recall some results from \cite[\S 5]{Moy2024}. The first one is a consequence of the local trace formula of Jin--Zworski and Jin--Tao \cite{Jin2016,jin2023number} for Anosov flows.
\begin{prop}[{\cite[Theorem 5.1 and Lemma 5.2]{Moy2024}}]\label{prop: equidistribution result} Denote by $\mathcal P$ the set of periodic orbits of the flow $\Phi^t:\Sigma_E\to \Sigma_E$. Let $\varphi$ be a smooth real compactly supported function. Then, the following holds for $L\ge 1$:
	\[\frac 1L\sum_{\gamma\in \mathcal P} \frac{T_\gamma^\sharp \varphi(T_\gamma/L)}{|\det(I-P_\gamma)|}=\int_{\mathbf{R}_+} \varphi(t)\mathrm{d}t+\mathcal O\Big(\frac 1L\Big), \]
	moreover the error improves to $\mathcal O(L^{-2})$ if $\varphi(0)=0$. Also, as $T\to +\infty$:
	\[\sum_{\gamma\in \mathcal P} \frac{T_\gamma^\sharp \varphi(T_\gamma-T)}{|\det(I-P_\gamma)|}=\int_{\mathbf{R}_+} \varphi(t)\mathrm{d}t+o(1).\]
Also, both results still hold when the sums are taken over \emph{primitive} periodic orbits. \end{prop}
\begin{rem}Proposition \ref{prop: equidistribution result} is a smoothed version of the sum rule derived by Hannay--Ozorio de Almeida \cite{JHHannay1984}.\end{rem}
Introduce the shorthand $A_\gamma:=\frac{T_\gamma^\sharp}{|\det(I-P_{\gamma})|} |\widehat f(T_\gamma/L)|$. By applying Proposition \ref{prop: equidistribution result} to $\varphi(t)=t |\widehat f(t)|^2$, we obtain
\begin{equation} \label{eq: formula for GOE}\frac{4}{L^2}\sum_{\gamma\in \mathcal P} A_\gamma^2=\Sigma_{\rm GOE}^2(f)+\mathcal O(L^{-2}).\end{equation}
We will use repeatedly the following elementary inequality, which is a variation of \cite[Lemma 6.3]{Moy2024}. It allows controlling the Poincar\'e determinant of repetitions of periodic orbits.
\begin{lem}\label{lem: poincare map powers} There is a constant $c>0$ and constants $C_j>0$ for $j\ge 0$ such that for any $\gamma\in \mathcal P$ and $k\ge 1$, we have
	\[\forall j\le k, \ |\det(I-P_{\gamma^k})|\ge C_{j}|\det(I-P_\gamma)|^j \mathrm{e}^{c(k-j)T_\gamma}.\]\end{lem}

\subsubsection{Time-reversal} \label{subsec: time reversal}
Let $\iota$ denote the involution $(x,\xi)\in T^*M\mapsto (x,-\xi)\in T^*M$. Since $p(x,\xi)=p(x,-\xi)$, one has
\begin{equation}\label{eq: time reversal}\iota\circ \Phi^t\circ \iota=\Phi^{-t}.
\end{equation}
We say that the flow is \emph{time-reversal invariant} under the transformation $\iota$. Each periodic orbit $\gamma\subset \mathcal P$ of the flow admits a time reversed version $\gamma^{-1}:=\iota(\gamma)$, obtained by reversing the momentum direction. One has the following easy lemma.
\begin{lem}\label{lem: iota swaps}The involution $\iota$ exchanges the stable and unstable spaces, i.e.
	\[\iota_*(E^u)=E^s, \qquad \iota_*(E^s)=E^u.\]
	Thus, $\iota$ preserves the direct sum $E^s\oplus E^u\subset T\Sigma_E$. \end{lem}
This allows relating the Poincar\'e determinants of an orbit and its time-reversed version.
\begin{lem}Let $\gamma$ denote a periodic orbit of the flow, and let $\gamma^{-1}:=\iota(\gamma)$ denote the time-reversed orbit. Then, one has
	\[|\det(I-P_{\gamma^{-1}})|=|\det(I-P_\gamma)|.\] \end{lem}
\begin{proof}By \eqref{eq: time reversal} and Lemma \ref{lem: iota swaps}, up to conjugation one has
	\[P_{\gamma^{-1}}=\mathrm{d}\iota\circ (P_\gamma)^{-1}\circ \mathrm{d}\iota.\]
	Since $P_\gamma$ is symplectic, it has determinant $1$, thus
	\[|\det(I-P_{\gamma^{-1}})|=|\det(I-P_\gamma^{-1})|=|\det(I-P_\gamma)|.\]
\end{proof}

\subsection{Separation of periodic orbits}\label{subsec: separation} Let $p$ satisfy Conditions \ref{assum: convex}--\ref{assum: anosov}. Consider an Anosov energy level $E$. We show that distinct periodic orbits of the flow on $\Sigma_E$ do not spend too much time close to each other in position space (Proposition \ref{prop:separation orbite base 2}), and that a periodic orbit does not intersect itself too often on the base (Proposition \ref{prop: self intersect base}). We first quote a result of \cite{Schenk}, which deals with these questions in phase space. Its proof relies on a quantitative version of the Anosov closing lemma.
\begin{prop}[Separation of periodic orbits in phase space, {\cite[Lemma 9]{Schenk}}] \label{prop: separation orbits phase space} There are positive constants $c,C>0$ such that for any distinct primitive periodic orbits $\gamma_1,\gamma_2$ with periods $\le L$, one has
	\[\operatorname{dist}_{T^*M}(\gamma_1,\gamma_2)> c\mathrm{e}^{-CL}.\]
\end{prop}
With the same arguments, one finds that periodic orbits do not refocus too much.
\begin{prop}[Non refocusing of periodic orbits]\label{prop: nonfocusing}Let $\gamma$ be a primitive periodic orbit of the flow on $\Sigma_E$, of period $L$. There is a constant $C_\star>0$ such that the following holds. For any $0\le \varepsilon\le \mathrm{e}^{-C_\star L}$, if $\rho,\rho'$ are two points of the orbit with $d_{T^*M}(\rho,\rho')\le \varepsilon$, then there is a time $\tau$ satisfying $|\tau|\le C_\star\varepsilon$ such that $\rho'=\Phi^{\tau}(\rho)$. \end{prop}
\begin{proof} Fix $C_\star>0$, to be chosen below. Let $\rho,\rho'$ be two points of the orbit satisfying $d_{T^*M}(\rho,\rho')\le \varepsilon$. Up to exchanging $\rho$ and $\rho'$ we can write $\rho'=\Phi^\tau(\rho)$ for some $\tau\in [0,L/2]$. By contradiction, assume that $|\tau|\ge C_\star\varepsilon$. For small times, by injectivity radius considerations we have $d(\rho,\Phi^t(\rho))\gtrsim |t|$. Thus the fact that $d_{T^*M}(\rho,\rho')\le \varepsilon$ forces $\tau\ge \tau_0$ for some uniform positive constant $\tau_0>0$. We can now apply the Anosov closing lemma \cite[Lemma 8.3.13]{Lefeuvre} to the pseudo-orbit $\{\Phi^t(\rho)\}_{t\in [\tau,L]}$ to construct another periodic orbit $\gamma'$ of period $L'$ satisfying $|L-\tau-L'|\le \mathrm{e}^{CL}\varepsilon$ and such that $\operatorname{dist}(\gamma,\gamma')\le \varepsilon \mathrm{e}^{CL}$. By taking $C_\star$ large enough depending on $C$, we find that $\gamma\neq \gamma'$, and the estimate $\operatorname{dist}(\gamma,\gamma')\le \varepsilon \mathrm{e}^{CL}$ then contradicts Proposition \ref{prop: separation orbits phase space}.
\end{proof}
We need analogous statements of the above propositions where we replace distances on $T^*M$ by distances on $M$. A result as Proposition \ref{prop: separation orbits phase space} with $\operatorname{dist}_{T^*M}$ replaced by $\operatorname{dist}_M$ is obviously false as many periodic orbits cross each other on the base manifold. However, the key point is that due to the above fact and the particular structure of the Hamiltonian, the angle made between the curves at such crossings cannot be too small, implying that two distinct periodic orbits cannot spend too much time close to each other. This is made precise in the following proposition.

\begin{prop}[Crossings of periodic orbits on the base] \label{prop:separation orbite base 2} Let $p$ satisfy Conditions \ref{assum: convex}--\ref{assum: anosov}. There are constants $C,c>0$ such that the following holds. 
	Take two primitive periodic orbits $\gamma_1,\gamma_2$ of the flow, with periods $T_1,T_2\le L$, and such that $\gamma_2\notin \{\gamma_1,\gamma_1^{-1}\}$. Define
	\[I_\varepsilon=\{t\in [0,T_1], \ \exists s\in [0,T_2], \ d_M(\gamma_1(t),\gamma_2(s))< \varepsilon\}.\]
	Assume that $\varepsilon\le c\mathrm{e}^{-CL}$. Then, the measure of $I_\varepsilon$ is smaller than $CL^{2}\mathrm{e}^{C L}\varepsilon$.
\end{prop} 
\begin{rem}\label{rem: shell is trapped}Since $p$ is elliptic, the energy shell $\Sigma_E$ is contained in $\{|\xi|\le r_1\}$ for some $r_1>0$. The symmetry property of $p$ implies $\nabla_\xi p(x,0)=0$, and in turn the convexity assumption implies the existence of $\alpha>0$ depending on $r_1$ such that $|\nabla_\xi p(x,\xi)|\ge \alpha|\xi|$ when $|\xi|\le r_1$. Since $E$ is an Anosov energy level, by \cite{Paternain1994}, one has $E>\max_{x\in M}p(x,0)$. It follows that there exists another constant $r_0$ such that
	\[\Sigma_E\subset \{(x,\xi)\in T^*M~:~ r_0\le |\xi|\le r_1\}.\]
	Thus, $|\nabla_\xi p(x,\xi)|\ge \alpha r_0$ on $\Sigma_E$. By the Hamilton equations, it follows that classical trajectories on $\Sigma_E$ have speed $\ge \alpha r_0$ on the base.
\end{rem}
Before proving Proposition \ref{prop:separation orbite base 2}, we state a lemma about convex functions.
\begin{lem}\label{lem: lem convexe}Let $\alpha>0$ and let $f:\mathbf{R}^d\to \mathbf{R}$ be a $C^2$, $\alpha$-convex function (i.e. $\nabla^2 f\ge \alpha\mathrm{Id}$), such that $\nabla f(0)=0$. Let $R\ge r>0$. Let $a,b\in B(0,R)\backslash B(0,r)$ satisfy $f(a)=c$ and $f(b)=c+\varepsilon$ with $\varepsilon\in [0,1]$. Then, the angle $\theta\in [0,\pi]$ between $\nabla f(a)$ and $\nabla f(b)$ satisfies 
\[|b-a|\lesssim_{\!\! f,r,R}(\theta+ \varepsilon).\]\end{lem}
In the case $\varepsilon=0$, this lemma basically states that the angle between the normal vectors to a given level set of $f$ at the points $a$ and $b$ is controlled by the distance between $a$ and $b$. This relies on the convexity of the set $\{f\le c\}$.
\begin{proof}Let $u=b-a$. Then, by convexity,
	\[f(b)-f(a)-\langle \nabla f(a),u\rangle\ge \frac{\alpha}{2}|u|^2, \qquad f(a)-f(b)+\langle \nabla f(b),u\rangle\ge \frac{\alpha}{2}|u|^2.\]
	Then,
	\[\Big\langle \frac{\nabla f(a)}{|\nabla f(a)|},u\Big\rangle\le \frac{1}{|\nabla f(a)|}\Big(\varepsilon- \frac{\alpha}{2}|u|^2\Big), \quad \text{and}\quad \Big\langle \frac{\nabla f(b)}{|\nabla f(b)|},u\Big\rangle\ge \frac{1}{|\nabla f(b)|}\Big(\varepsilon+\frac{\alpha}{2}|u|^2\Big).\]
	Summing these two inequalities and dividing both sides by $|u|=|b-a|$ leads to the result. We have used the fact that $|\nabla f(a)|$ is bounded from above and below by positive constants depending on $f,r,R$, and that the difference between $\frac{1}{|\nabla f(a)|}$ and $\frac{1}{|\nabla f(b)|}$ is bounded by $|b-a|$ times a constant depending on $f,r,R$.
\end{proof}
We turn to the proof of the proposition.
\begin{proof}[Proof of Proposition \ref{prop:separation orbite base 2}]
	Introduce the open set 
	\[\mathcal A_\varepsilon=\{(t,s)\in \mathbf R/T_1\mathbf Z\times \mathbf R/T_2\mathbf Z \ : \ d_M(\gamma_1(t),\gamma_2(s))< \varepsilon\}\]
	and note that $I_\varepsilon$ is nothing but the projection of $\mathcal A_\varepsilon$ on the first coordinate. 
	We claim that the connected components of $\mathcal A_\varepsilon$ have diameter $\lesssim \mathrm{e}^{CL}\varepsilon$, and are distant from each other by some uniform distance $d_0>0$. This will imply the result on $I_\varepsilon$ by projecting on the first coordinate.
	
	To begin, consider the distance function between the two orbits on the base
	\[f:(t,s)\mapsto d_M(\gamma_1(t),\gamma_2(s))\]
	Then, $\mathcal A_\varepsilon=f^{-1}([0,\varepsilon[)$ is open, while the closure $\overline{\mathcal A_\varepsilon}$ is included in $f^{-1}([0,\varepsilon])$. Take a nonempty connected component $U$ of $\mathcal A_\varepsilon$. Then, the restriction $f_{|\overline U}$ admits a local minimum at some point $(t_\star,s_\star)$ that actually belongs to $U$, since $f=\varepsilon$ on $\partial U$. There are two cases to distinguish:
	\begin{itemize}
		\item $f(t_\star,s_\star)=0$, meaning that the curves $\gamma_1(t)$ and $\gamma_2(s)$ intersect at $(t_\star,s_\star)$ on the base, with an angle $\theta\in ]-\pi,\pi]$.
		\item $f(t_\star,s_\star)>0$.
	\end{itemize}
	
	\textbf{1.} We first rule out the second possibility. By contradiction, assume that $f(t_\star,s_\star)>0$. Let $\pi:T^*M\to M$ denote the projection. Then, the vectors $\partial_t \pi(\gamma_1(t_\star))$ and $\partial_s \pi(\gamma_2(s_\star))$ are parallel (however, they need not have the same norm). Since the points $\pi(\gamma(t_\star))$ and $\pi(\gamma_2(s_\star))$ are $\varepsilon$-close, we can fix a local chart $U$ of $M$ near these points, that induces local symplectic coordinates on $T^*U$. We write $\gamma_1(t)=(x_1(t),\xi_1(t))$ and $\gamma_2(s)=(x_2(s),\xi_2(s))$ in these coordinates. By what precedes, the vectors $\dot x_1(t_\star)$ and $\dot x_2(s_\star)$ have positive norms and belong to a cone of aperture $\le C\varepsilon$, where $C$ depends on the constants $\alpha,r_0,r_1$ from Remark \ref{rem: shell is trapped}. Up to replacing $\gamma_2$ by the momentum-reversed orbit, we can assume that this cone is \emph{positive}. 
	
	By the Hamilton equations, one has in local coordinates
	\[\dot x_1(t)=\partial_\xi p(x_1(t),\xi_1(t)), \qquad \dot x_2(s)=\partial_\xi p(x_2(s),\xi_2(s)).\]
	Since $x_1(t_\star)$ is $C\varepsilon$-close to $x_2(s_\star)$, it follows that $\partial_\xi p(x_2(s_\star),\xi_1(t_\star))$ and $\partial_\xi p(x_2(s_\star),\xi_2(s_\star))$ belong to a positive cone of aperture $\le C \varepsilon$, where $C$ might have increased. Note also that $p(x_2(s_\star),\xi_1(t_\star))=E+\mathcal O(\varepsilon)$ because $p(x_1(t_\star),\xi_1(t_\star))=E$ and $x_1(t_\star)$ is $\varepsilon$-close to $x_2(s_\star)$. By Lemma \ref{lem: lem convexe}, we deduce 
	\[|\xi_1(t_\star)-\xi_2(s_\star)|\le C\varepsilon,\]
	implying $d_{T^*M}(\gamma_1(t_\star),\gamma_2(s_\star))\le C\varepsilon$. This contradicts the orbit separation in phase space when $\varepsilon\le c\mathrm{e}^{-CL}$.\medskip
	
	\textbf{2.}	We can thus consider the case $f(t_\star,s_\star)=0$, meaning that $x_1(t_\star)=x_2(s_\star)=x$ for some $x\in M$. We work again in local symplectic coordinates $(x,\xi)$. By Proposition \ref{prop: separation orbits phase space}, we know that $|\xi_1(t_\star)-\xi_2(s_\star)|\ge c\mathrm{e}^{-CL}$. Write $v_1:=\dot x_1(t_\star)=\partial_\xi p(x,\xi_1(t_\star))$ and $v_2:=\dot x_2(s_\star)=\partial_\xi p(x,\xi_2(t_\star))$. Since $p(x,\xi_1(t_\star))=p(x,\xi_2(s_\star))=E$, Lemma \ref{lem: lem convexe} implies that the angle $\theta\in [0,\pi/2]$ between the lines $\mathbf{R}v_1$ and $\mathbf{R}v_2$ satisfies
	\begin{equation} \label{eq: lower bound theta} \theta\gtrsim \theta_0:=c\mathrm{e}^{-CL}.\end{equation}
	Now, provided that $u,u'$ are small enough, we may write in coordinates using Taylor's formula:
	\[x_1(t_\star+u)=x+ uv_1 +\mathcal O(u^2), \qquad x_2(s_\star+u')= x +u'v_2+\mathcal O(u'^2).\]
	The implied constant in $\mathcal O(\cdot)$ is uniform in $x$ because $M$ is compact. By Remark \ref{rem: shell is trapped}, $v_1$ and $v_2$ have positive norm, it then follows from \eqref{eq: lower bound theta} that
	\[d_M(x_1(t_\star+u),x_2(s_\star+u'))\gtrsim \theta_0 \min(|u|,|u'|)+\mathcal O(|u|^2+|u'|^2).\]
	Therefore, there is a constant $\eta>0$ such that for any $u,u'\in ]-\eta,+\eta[$, one has
	\[d_M(x_1(t_\star+u),x_2(s_\star+u'))\gtrsim c\max(|u|,|u'|) \mathrm{e}^{-CL}.\]
	It follows that when $(u,u')$ satisfies
	\[\eta\ge \max(|u|,|u'|) \ge c\mathrm{e}^{CL}\varepsilon,\]
	then $(t_\star+u,s_\star+u')\notin \mathcal A_\varepsilon$. This implies the following two facts:
	\begin{itemize}
		\item Connected components of $\mathcal A_\varepsilon$ have diameter $\lesssim \varepsilon \mathrm{e}^{CL}$.
		\item Disjoint connected components of $\mathcal A_\varepsilon$ are at distance $\ge \eta$ from each other, for some positive constant $\eta$ that is independent of $\varepsilon,L$. Thus, $\mathcal A_\varepsilon$ has at most $\mathcal O(L^2)$ disjoint nonempty connected components.
	\end{itemize}
	
	\textbf{3.} Since $\mathcal A_\varepsilon$ has $\mathcal O(L^2)$ disjoint nonempty connected components, each of which has diameter $\lesssim \varepsilon \mathrm{e}^{CL}$, and since $I_\varepsilon$ is the projection of $\mathcal A_\varepsilon$ on the first coordinate, we infer:
	\[|I_\varepsilon| \lesssim L^2\mathrm{e}^{CL}\varepsilon,\]
	as we wished.
\end{proof}

Similarly, we obtain an analogue of Proposition \ref{prop: nonfocusing} on the base.
\begin{prop}[Self-intersections on the base] \label{prop: self intersect base} Under the same assumptions as in Proposition \ref{prop:separation orbite base 2}, there are constants $C,c>0$ such that the following holds. Take a primitive orbit $\gamma(t)$ of period $L$, and define
	\[I_\varepsilon=\big\{t\in \mathbf{R}/L\mathbf{Z}~:~ \exists s\in \mathbf{R}/L\mathbf{Z}, \ \text{$|s-t|>C \varepsilon$ and $ d_M(\gamma(t),\gamma(s))< \varepsilon$}\big\}.\]
	Assume that $\varepsilon\le c\mathrm{e}^{-CL}$. Then, the measure of $I_\varepsilon$ is smaller than $CL^{2}\mathrm{e}^{C L}\varepsilon$.
\end{prop} 
\begin{proof} This is the same proof as above, but we instead work with
	\[\mathcal A_\varepsilon'=\big\{(t,s)\in \mathbf R/L\mathbf Z\times \mathbf R/L\mathbf Z~:~ |t-s|>C \varepsilon ~\text{and}~ d_M(\gamma(t),\gamma(s))< \varepsilon\big\}\]
	with a large enough constant $C$, and invoke Proposition \ref{prop: nonfocusing} in place of Proposition \ref{prop: separation orbits phase space}.
\end{proof}
Proposition \ref{prop: self intersect base} shows that if $\gamma$ has period $T_\gamma\le \mathfrak c(\beta)|\log h|$, then the intersection of $\gamma$ with any ball of radius $h^\beta$ is either empty or consists of a single segment. This allows proving the next result.

\begin{prop}\label{prop: random mag checks assumptions} The random magnetic potential model satisfies Condition \ref{assum: condition potential}.
\end{prop}
\begin{proof}We merely need to check item (\ref{assum: int Vj flow}), that is, for any periodic orbit $\gamma$ with period $T_0\le T_\gamma\le \mathfrak c|\log h|$, one has
	\[\sum_{\bm j\in J_h} \Big|\int_{\gamma} q_{\bm j}\Big|=\sum_{i\in I}\sum_{j\in J_h'} \Big|\int_\gamma \chi_i q_j \langle \xi,f_i^*\mathbf e_1\rangle\Big|+\Big|\int_\gamma \chi_i q_j \langle \xi,f_i^*\mathbf e_2\rangle\Big|\ge c_0T_\gamma,\]
	where the sum on the left runs over $\bm j=(k,i,j)\in J_h=\{1,2\}\times I\times J_h'$. By the remark after Proposition \ref{prop: self intersect base}, we know that the intersection of $\gamma$ with $\operatorname{Supp}(q_j')$ consists of a single segment of orbit of length $\le h^\beta$. On this segment, the quantities $\langle \xi,f_i^*\mathbf e_k\rangle$ for $k\in \{1,2\}$ are almost constant, more precisely
	\[\langle \xi,f_i^*\mathbf e_k\rangle=u_{i,k}+\mathcal O(h^\beta),\]
	with $u_{i,k}$ independent of the point. Since the orbit $\gamma$ spends a time $\lesssim T_\gamma h^\beta$ in the ball $B(x_j,h^\beta)$ (this follows from item (\ref{assum: supp pot Vj}) in Condition \ref{assum: condition potential}), one obtains
	\[\int_\gamma \chi_i q_j'\langle \xi,f_i^* \mathbf e_k\rangle=u_{i,k}\int_\gamma \chi_i q_j'+\mathcal O(T_\gamma h^{2\beta}\mathbf 1_{\gamma\cap \operatorname{Supp}(q_j')\neq \varnothing}).\]
	moreover,
	\[u_{i,1}^2+u_{i,2}^2=\|(f_i)_*\xi\|^2 +\mathcal O(h^\beta)\gtrsim 1,\]
	thus $|u_{i,1}|+|u_{i,2}|\gtrsim 1$. Since $\sum \chi_i=1$, the triangle inequality implies
	\[\sum_{i\in I}\Big|\int_\gamma \chi_i q_j'\Big|\ge \Big|\int_\gamma q_j'\Big|.\]
	Since $(q_j')_{j\in J_h'}$ satisfies Condition \ref{assum: condition potential}, and since there are at most $\mathcal O(T_\gamma h^{-\beta})$ balls $B(x_j,h^\beta)$ that intersect $\gamma$, one finds
	\begin{equation} \begin{aligned}\sum_{\bm j\in J_h} \Big|\int_{\gamma} q_{\bm j}\Big| & \ge \sum_{i\in I}\sum_{j\in J_h'} \Big[\Big|u_{i,1}\int_\gamma \chi_i q_j'\Big|+\Big|u_{i,2}\int_\gamma \chi_i q_j'\Big|+\mathcal O(T_\gamma h^{2\beta}\mathbf 1_{\gamma\cap \operatorname{Supp}(q_j')\neq \varnothing})\Big] \\
	& \gtrsim \sum_{j\in J_h'} \Big[\Big|\int_\gamma q_j'\Big|+\mathcal O(T_\gamma h^{2\beta}\mathbf 1_{\gamma\cap \operatorname{Supp}(q_j')\neq \varnothing})\Big]\ge c_0T_\gamma +\mathcal O(T_\gamma^2 h^{\beta})\ge \frac{c_0}{2}T_\gamma. \end{aligned}\end{equation}
This finishes the proof.
\end{proof}

\section{Cancellations between periodic orbits} The following proposition is the main input in the proof of Proposition \ref{prop: moments N_osc}, as it provides a probabilistic cancellation mechanism between the contributions of distinct periodic orbits to the trace formula. Recall that we assume $2\beta<\alpha<1$ and $\beta<2-2\alpha$.

\begin{prop}\label{prop: averages actions}Let $0<\varepsilon<\min(1-\alpha-\beta/2,\alpha-2\beta)$. Then there is a constant $\mathfrak c=\mathfrak c(\alpha,\beta,\varepsilon)>0$ depending additionally on the flow such that the following holds. For all $n\ge 1$, there is a constant $C_n>0$ such that for all sufficiently small $h>0$ depending on $n$, for any tuple $(\gamma_1,\ldots,\gamma_n)$ of periodic orbits with periods $\le \mathfrak c|\log h|$ such that $\gamma_j\notin \{\gamma_i,\gamma_i^{-1}\}$ when $i\neq j$, one has:
	\begin{itemize}
  \item if the random perturbation preserves time reversal symmetry, then for any nonzero integers $k_1,\ldots,k_n$, one has
  \[\Big|\mathbf E_\omega\Big[\exp\Big(\frac{\mathrm{i}}{h}\sum_{1\le i\le n} k_i S_{\omega}(\gamma_i)\Big)\Big]\Big|\le C_n h^{\varepsilon n}.\]
  \item if the random perturbation strongly breaks time reversal symmetry, then for any integers $k_{1}^{\pm},\ldots,k_n^{\pm}$, one has
   \[\Big|\mathbf E_\omega\Big[\exp\Big(\frac{\mathrm{i}}{h}\sum_{1\le i\le n} k_i^+ S_{\omega}(\gamma_i)+k_i^-S_\omega(\gamma_i^{-1})\Big)\Big]\Big|\le C_n h^{\varepsilon N(\vec k^{\pm})},\]
   where $N(\vec k^{\pm})$ is the number of pairs $\{k_i^+,k_i^-\}$ of the form $\{-1,1\}$ or $\{0,k\}$ with $k\neq 0$.
\end{itemize}\end{prop}
We recall from Proposition \ref{prop: long time gutz} that
\[S_\omega(\gamma_i)=S_0(\gamma_i)-h^\alpha\int_{\gamma_i}q_\omega +r_\omega(\gamma),\]
where $S_0(\gamma_i)$ is independent of $\omega$ and $|\partial_{\omega_i} r_\omega(\gamma)|\lesssim h^{2\alpha-\beta}\mathrm{e}^{CL}$. \medskip

We begin with some preliminary results. By an adaptation of \cite[eq. (9.8)]{ingremeau2023emergence}, one can show that Condition \ref{assum: condition potential} implies the following lemma.
\begin{lem}\label{lem: Jh(gamma)} Let $(q_j)_{j\in J_h}$ be a family of potentials satisfying Condition \ref{assum: condition potential}. For $j\in J_h$ set $\theta_j(\gamma)=\int_{\gamma}q_j$. Then, there are constants $c_1,c_2, \varepsilon_0>0$ such that for any closed orbit $\gamma$, there is a subset $J_h(\gamma)\subset J_h$ such that $c_1h^{-\beta}T_\gamma\le \#J_h(\gamma)\le c_2 h^{-\beta} T_\gamma$, and
	\[\forall j\in J_h(\gamma), \quad |\theta_j(\gamma)|\ge \varepsilon_0 h^\beta.\]
	Moreover, there is a constant $c_3>0$ such that
	\[\#\{j\in J_h~:~ \theta_j(\gamma)\neq 0\}\le c_3 h^{-\beta}T_\gamma.\]
\end{lem}

\begin{lem}\label{lem: Jh(gamma gamma')}Let $\gamma_1,\gamma_2$ be two primitive periodic orbits of period $\le L$, with $\gamma_2\notin\{\gamma_1,\gamma_1^{-1}\}$. Fix $C_0>0$ and set 
	\[J_h(\gamma_1,\gamma_2):=\{j\in J_h~:~ \operatorname{dist}_M(x_j,\gamma_1)<C_0h^\beta\ \text{and} \ \operatorname{dist}_M(x_j,\gamma_2)<C_0h^\beta\}.\] Then,
	\[\#J_h(\gamma_1,\gamma_2)\lesssim_{\!\! C_0} \mathrm{e}^{C L}.\]
\end{lem}
\begin{proof} We view $\gamma_1$ and $\gamma_2$ as curves on $M$. For all $j\in J_h(\gamma,\gamma')$, we can find $(t_j,s_j)$ such that $\gamma_1(t_j)$ and $\gamma_2(s_j)$ belong to $B(x_j,C_0h^\beta)$. It follows that for all $t\in \mathcal I_j:=[t_j-h^\beta,t_j+h^\beta]$ we have $d_M(\gamma_1(t),\gamma_2(s_j))<(C_0+C_1)h^\beta=:C_2h^\beta$ for some constant $C_1$ depending on the flow. Therefore, in the notation of Proposition \ref{prop:separation orbite base 2}, we have $\mathcal I_j\subset I_{C_2h^\beta}$. Now, set
	\[\mathcal I:=\bigcup_{j\in J_h(\gamma_1,\gamma_2)} \mathcal I_j\subset I_{C_2h^\beta}.\]
	Note that each element of $\mathcal I$ belongs at most to $C_3L$ of the $\mathcal I_j$'s, where $C_3$ depends on item (\ref{assum: supp pot Vj}) in Condition \ref{assum: condition potential}. Using Proposition \ref{prop:separation orbite base 2}, we infer that if $\mathfrak c=\mathfrak c(\alpha,\beta)$ is small enough, then for $L\le \mathfrak c|\log h|$ we have, denoting by $\operatorname{Leb}$ the Lebesgue measure:
	\[C_2 \mathrm{e}^{CL} h^\beta \ge \operatorname{Leb}(I_{C_2h^\beta})\ge \operatorname{Leb}(\mathcal I)\ge \frac{1}{C_3L}\sum_{j\in J_h(\gamma,\gamma')}\operatorname{Leb}(\mathcal I_j)\ge \frac{2}{C_3L}h^\beta \#J_h(\gamma,\gamma').\]
	The result follows by increasing slightly $C$.
\end{proof}
We now consider bunches of orbits.
\begin{lem}\label{lem: bunches periodic orbits} There are constants $C_0,C_1,C>0$ such that the following holds. Let $\gamma_1,\ldots,\gamma_n$ be primitive periodic orbits of period $\le L$, such that $\gamma_j\notin\{\gamma_i,\gamma_i^{-1}\}$ if $i\neq j$. Then, there are pairwise disjoint subsets $\widetilde J_h(\gamma_i)\subset J_h$ for $i=1,\ldots,n$ such that
	\begin{itemize}
		\item $\widetilde J_h(\gamma_i)\subset J_h(\gamma_i)$.
		\item $\# \widetilde J_h(\gamma_i)\ge c_1 T_{\gamma_i} h^{-\beta}-C_1 n \mathrm{e}^{CL}$
		\item If $j\in \widetilde J_h(\gamma_{i})$ then $\operatorname{dist}(x_j,\gamma_{i'})\ge C_0 h^\beta$  for all $i'\neq i$.
\end{itemize}\end{lem}
\begin{proof}With the notation of Lemma \ref{lem: Jh(gamma)} and Lemma \ref{lem: Jh(gamma gamma')}, it suffices to take
	\[\widetilde J_h(\gamma_k):=J_h(\gamma_k)\backslash \bigcup_{k'\neq k}J_h(\gamma_k,\gamma_{k'}).\]
\end{proof}
The next lemma will be crucial to simplify computations in the proof of Proposition \ref{prop: averages actions}. Recall that $r_\omega$ has been defined in Proposition \ref{prop: long time gutz}.
\begin{lem}\label{lem: derivees r omega zero} Consider a tuple $(\gamma_1,\ldots,\gamma_n)$ as in Lemma \ref{lem: bunches periodic orbits}. Assume that $j\in \widetilde J_h(\gamma_i)$. Then, for $i'\neq i$ we have
	\[\forall \omega_j\in [-1,1], \quad \partial_{\omega_j} r_\omega(\gamma_i')=0.\]
\end{lem}

\begin{proof} To prove this result, we need to make the definition of $r_\omega$ explicit. Fix $\omega\in [-1,1]^{J_h}$. Let $\Phi_{\delta,\omega}^t$ denote the Hamiltonian flow generated by $p_{\delta,\omega}$. We proved in \cite{MoyLongTime} that for every periodic orbit $\gamma=\gamma_0\subset p^{-1}(E)$ of the flow $\Phi^t_0$, there is a smooth variation of periodic orbits $\gamma_{\delta,\omega}\subset p_{\delta,\omega}^{-1}(E)$ of the flow $\Phi_{\delta,\omega}^t$. Moreover for every point $\rho_0\in \gamma_0$, one can find a smooth variation of points $\rho_{\delta,\omega}\in \gamma_{\delta,\omega}$. Denoting by $T_{\delta,\omega}$ the period of $\gamma_{\delta,\omega}$ and introducing $z_{\delta,\omega}:=(T_{\delta,\omega},\rho_{\delta,\omega})$, we know from \cite{MoyLongTime} (see the proof of Proposition \ref{prop: bounds derivatives omega r} below) that $|\partial_{\delta}z_{\delta,\omega}|\lesssim h^{-\beta} \mathrm{e}^{CT_\gamma}$ uniformly in $\delta\in [0,h^\alpha]$.
There is also an explicit formula for the variation of the action of $\gamma_{\delta,\omega}$, given by
\begin{equation} \label{eq: partial delta S delta}\partial_\delta S(\gamma_{\delta,\omega})=-\int_{\gamma_{\delta,\omega}}\partial_\delta p_{\delta,\omega}.\end{equation}
A consequence of the estimates on $z_{\delta,\omega}$ is that for $\delta\in [0,h^\alpha]$, the orbit $\gamma_{\delta,\omega}$ is contained in a neighborhood of size $\mathcal O(h^{\alpha-\beta}\mathrm{e}^{CL})$ of $\gamma_0$. The function $r_\omega(\gamma)$ is defined by
\[S(\gamma_{h^\alpha,\omega})=S(\gamma_0)-h^\alpha \frac{\partial S(\gamma_{\delta,\omega})}{\partial \delta}\Big|_{\delta=0}+r_{\omega}(\gamma).\]
By Taylor's formula, one has
\begin{equation} \label{eq: express r omega}r_\omega(\gamma)=\int_0^{h^\alpha} (h^\alpha-\delta)\partial_{\delta}^2 S(\gamma_{\delta,\omega}) \mathrm{d}\delta.\end{equation}
By differentiating under the integral in \eqref{eq: express r omega}:
\begin{equation} \label{eq: partial omega j r}\partial_{\omega_j}r_\omega(\gamma)=\int_0^{h^\alpha} (h^\alpha-\delta)\partial_{\delta}^2 \partial_{\omega_j}S(\gamma_{\delta,\omega}) \mathrm{d}\delta.\end{equation}
In the notation of the lemma, starting from $\gamma=\gamma_{i'}$, it is enough to prove that $\partial_{\omega_j}S(\gamma_{\delta,\omega})=0$ for all $\delta\in [0,h^\alpha]$ to finish the proof. By \eqref{eq: partial delta S delta}, one has
\[\partial_{\omega_j}S(\gamma_{\delta,\omega})=-\int_{\gamma_{\delta,\omega}}\partial_{\omega_j}p_{\delta,\omega}.\]
Let $i\neq i'$ and take $j\in \widetilde J_h(\gamma_i)$. We claim that $\gamma_\delta$ does not intersect the support of $\partial_{\omega_j}p_{\delta,\omega}$, for all $\delta\in [0,h^\alpha]$. By construction, we have $\operatorname{dist}(x_j,\gamma_i')\ge C_0h^\beta$. Since $\gamma_{\delta,\omega}$ is contained in a $\mathcal O(h^{\alpha-\beta}\mathrm{e}^{CL})$-neighborhood of $\gamma_{i'}$, we deduce
\[\forall \delta\in [0,h^\alpha], \qquad\operatorname{dist}(x_j,\gamma_{\delta,\omega})\ge C_0h^{\beta}-C_1 h^{\alpha-\beta}\mathrm{e}^{CL}\]
for another constant $C_1>0$. We choose $L\le \mathfrak c|\log h|$ with $\mathfrak c$ small enough so that $h^{\alpha-2\beta}\mathrm{e}^{CL}\le \frac{C_0}{2C_1}$. Then, for $h$ small enough, one has
\[\forall \delta\in [0,h^\alpha], \qquad \gamma_\delta\cap B(x_j,h^\beta)=\varnothing.\]
This allows us to conclude.
	
\end{proof}
Before turning to the proof of Proposition \ref{prop: averages actions}, we give estimates on the derivatives of $r_\omega(\gamma)$.
\begin{prop}\label{prop: bounds derivatives omega r} One has
	\begin{equation} \label{eq: prop partial omega i r}	|\partial_{\omega_i}r_{\omega}(\gamma)|\lesssim h^{2\alpha} h^{-\beta} \mathrm{e}^{CT_\gamma}.\end{equation}
\end{prop}
\begin{proof} 
We recall some results from \cite[\S 3,\S 6.7]{MoyLongTime}. There, we built a long time generating function for the flow $\Phi^t_{\delta,\omega}$ in coordinates. More precisely, let $\rho_0\in T^*M$ and set $\rho_1=\Phi_0^{t_0}(\rho_0)$. Then, there are local symplectic coordinates $(y,\eta)$ near $\rho_0$ and $(x,\xi)$ near $\rho_1$ (that map $\rho_0$ and $\rho_1$ to $(0,0)$, respectively), and a local generating function $\varphi_{\delta,\omega}(t,x,\eta)$ for the flow $\Phi^t_{\delta,\omega}$ such that 
\[\Phi^t_{\delta,\omega}(y,\eta)=(x,\xi)\Longleftrightarrow (y,\xi)=(\nabla_\eta \varphi_{\delta,\omega},\nabla_x\varphi_{\delta,\omega}).\]
Let $\gamma$ be a periodic orbit of the flow $\Phi_0^t$ of period $T_\gamma \le L\le \mathfrak c|\log h|$. Let $\rho_0\in \gamma$ and $t_0=T_\gamma$, so that $\rho_1=\rho_0$ in the above notation. We can thus take the same chart near $\rho_0$ and $\rho_1$. We can also design the local coordinates to ensure that the short-time flow $\Phi_0^s$ becomes the $s$-translation along the $x_1$ coordinates. We write $x=(x_1,x')\in \mathbf{R}\times \mathbf{R}$ and set
\[\Psi_{\delta,\omega}(t,x,\xi)=\varphi_{\delta,\omega}(t,x,\xi)-\langle x,\xi\rangle+tE.\]
The critical points of the phase correspond to periodic points of the flow $\Phi_{\delta,\omega}^t$ in the energy shell $p_{\delta,\omega}^{-1}(E)$. We fix $x_1=0$, and introduce the variable $z=(t,x',\xi)$, which parameterizes a Poincar\'e section for the flow. For $\delta=0$, one has $\nabla_z \Psi_0(0)=0$, because $\rho_0\equiv (0,0)$ is a periodic point of the flow $\Phi^t_0$. Then, by perturbation, for any $\delta\in [0,h^\alpha]$, there is a unique critical point $z_{\rm c}$ of $\Psi_{\delta,\omega}$ (that depends on $\delta,\omega$), which corresponds to a periodic point of the perturbed flow $\Phi_{\delta,\omega}^t$. We drop the $\delta,\omega$ subscripts to ease notation. The variations of $z_{\rm c}$ with respect to $\delta$ are given by
\begin{equation} \label{eq: partial delta z delta in proof} \partial_\delta z_{\rm c}=-(\nabla_z^2\Psi)^{-1}\partial_\delta \nabla_z \Psi.\end{equation}
Moreover, the value of the phase $\Psi(z_{\rm c})$ corresponds to the action of an orbit of the perturbed flow, and is thus independent of the point $\rho_0\in \gamma$. We denote this action by $S(\gamma_{\delta,\omega})$. We recall from \cite[Proposition 3.31]{MoyLongTime} the following estimates on the phase function $\Psi$. Assume $T_\gamma\le L$. Then, uniformly in $\delta\in [0,h^\alpha]$ and $\omega\in [-1,1]^{J_h}$, one has
\begin{equation} \label{eq: we recall the estimates}\forall n\ge 1,\forall k\ge 0 \quad |\partial_\delta^n\partial_z^k \Psi|\lesssim h^{-\beta(k+2n-2)}\mathrm{e}^{CL(2n+k-2)_+}.\end{equation}
Moreover, for all $j\in J_h$, one has
\begin{equation}\label{eq: we recall 2} \qquad \forall n\ge 1,\forall k\ge 0 \ |\partial_\delta^{n-1}\partial_{\omega_j}\partial_z^k\Psi|\lesssim h^{-\beta(k+2n-5)}\mathrm{e}^{CL(2n+k-2)_+}.\end{equation}
From the formula \eqref{eq: partial omega j r}, to obtain the sought estimate \eqref{eq: prop partial omega i r}, it is enough to prove that $|\partial_{\omega_j}\partial_\delta^2 S(\gamma_{\delta,\omega})|\lesssim h^{-\beta}$ uniformly in $0\le \delta\le h^\alpha$. By differentiating three times the identity $S(\gamma_{\delta,\omega})=\Psi(z_{\rm c})$, one finds
\begin{equation}\label{eq: d omega delta delta S}\begin{aligned}
		\partial^3_{\delta\delta \omega_i}S(\gamma_{\delta,\omega})& ={\color{black} (\partial^3_{\delta \delta\omega_i} \Psi)(z_{\rm c})+(\partial^3_{\delta\omega_iz}\Psi)\partial_\delta z_{\rm c}} \\
		& +{\color{black}\Big[\partial_{z \delta\omega_i}^3 \Psi +\partial^3_{zz\omega_i}\Psi \partial_{\delta}z_{\rm c}\Big]\partial_\delta z_{\rm c}+\partial_{z\omega_i}^2 \Psi \partial_{\delta\delta}^2z_{\rm c}} \\
		& +{\color{black} \Big[\partial_{z\delta\delta}^3 \Psi +\partial^3_{zz\delta}\Psi \partial_\delta z_{\rm c}\Big]\partial_{\omega_i}z_{\rm c}+\partial_{z\delta}^2\Psi \partial_{\delta\omega_i}^2z_{\rm c}} \\
		& +{\color{black}\Big[\partial_{\delta zz}^3\Psi +\partial_{zzz}^3\Psi \partial_{\delta}z_{\rm c}\Big]\partial_{\delta}z_{\rm c}\partial_{\omega_i}z_{\rm c} + \partial_{zz}^2\Psi\Big[\partial_{\delta\delta}^2z_{\rm c}\partial_{\omega_i}z_{\rm c}+\partial_{\delta}z_{\rm c}\partial_{\delta\omega_i}^2z_{\rm c}\Big]} \\
		& + {\color{black}\partial^2_{\delta z}\Psi\partial_{\delta\omega_i}^2z_{\rm c}+\partial_z\Psi\partial_\delta z_{\rm c}\partial_{\delta\omega_i}^2z_{\rm c}+\partial_z\Psi\partial^3_{\delta\delta\omega_i}z_{\rm c}}.
\end{aligned}\end{equation}
Here $\Psi$ and its derivatives are evaluated at $z_{\rm c}$. In particular, since $\partial_z\Psi(z_{\rm c})=0$, the last term vanishes, so we do not need to control $\partial^3_{\delta\delta\omega_i}z_{\rm c}$. The variations of the critical point $z_{\rm c}$ are estimated as follows, using \eqref{eq: partial delta z delta in proof} and \eqref{eq: we recall the estimates}--\eqref{eq: we recall 2}. There is a constant $C>0$ such that one has uniformly in $0\le \delta\le h^\alpha$:
	\begin{itemize}
		\item $|\partial_{\omega_i} z_{\rm c}|\lesssim h^{2\beta}$,
		\item $|\partial_{\delta}z_{\rm c}|\lesssim h^{-\beta}\mathrm{e}^{CL}$,
		\item $|\partial_{\delta\omega}z_{\rm c}|\lesssim \mathrm{e}^{CL}$,
		\item $|\partial_{\delta\delta}^2z_{\rm c}|\lesssim h^{-3\beta}\mathrm{e}^{CL}$.
	\end{itemize}
Finally, we evaluate each term of \eqref{eq: d omega delta delta S} separately using the estimate \eqref{eq: we recall the estimates} on the function $\Psi$ with the above estimates of $z_{\rm c}$. Each of them is $\mathcal O(h^{-\beta}\mathrm{e}^{CL})$, except the last one that simply vanishes because $\nabla_z\Psi(z_{\rm c})=0$, as we mentioned above.
\end{proof}

We are ready to prove Proposition \ref{prop: averages actions}.

\begin{proof}[Proof of Proposition \ref{prop: averages actions}] \textbf{1.} We prove the first item, where we assume that the perturbation preserves time-reversal symmetry. With Lemma \ref{lem: derivees r omega zero} at hand, the proof is an adaptation of \cite[Proposition 9.2]{ingremeau2023emergence}, which had $n=1$. We use the sets $\widetilde J_h(\gamma_i)$ from Lemma \ref{lem: bunches periodic orbits}. We introduce shorthands
		\[\mathbf I:=\mathbf E_\omega\Big[\exp\Big(\frac{\mathrm{i}}{h}\sum k_i S_{\omega}(\gamma_i)\Big)\Big],\]
		and
	\[\Phi(\omega):=-h^\alpha\sum_{j\in J_h} \omega_j\Big(\sum_{i=1}^n k_i\int_{\gamma_i}q_j\Big), \quad R(\omega):=\sum_{i=1}^n r_\omega(\gamma_i),\]
	We let $\bm m(\omega):=\prod_{j\in J_h} m(\omega_j)$ denote the product density on $[-1,1]^{J_h}$. Then,
	\begin{equation} \label{eq: esperance is integral}\mathbf I=\int_{[-1,1]^{J_h}} \exp\Big(\frac{\mathrm{i}}{h} \big(\Phi(\omega) +R(\omega)\big)\Big) \bm m(\omega)\mathrm{d}\omega.\end{equation}
	For $i\in \{1,\ldots,n\}$, let $\theta^{(i)}= (\mathbf 1_{j\in \widetilde J_h(\gamma_i)}  k_i\theta_j(\gamma_i))_{j\in J_h}$. Observe that the vectors $\theta^{(i)}$ have mutually disjoint supports, for $i=1,\ldots,n$. By Lemma \ref{lem: Jh(gamma)} and Lemma \ref{lem: bunches periodic orbits}, we have a lower bound
	\[\|\theta^{(i)}\|^2\gtrsim k_i^2 T_{\gamma_i}h^{-\beta}h^{2\beta}\gtrsim h^{\beta}.\]
	Note that
	\[(\theta^{(i)}\cdot \nabla_{\omega}) \mathrm{e}^{\frac{\mathrm{i}}{h}\Phi}=\frac{-\mathrm{i}h^\alpha}{h} \sum_{j\in \widetilde J_h(\gamma_i)} k_i^2\theta_j(\gamma_i)^2\mathrm{e}^{\frac{\mathrm{i}}{h}\Phi}=- \mathrm{i}h^{\alpha-1}\|\theta^{(i)}\|^2\mathrm{e}^{\frac{\mathrm{i}}{h}\Phi},\]
	and more generally,
	\[(\theta^{(1)}\cdot \nabla_{\omega})\cdots(\theta^{(n)}\cdot \nabla_{\omega})\mathrm{e}^{\frac{\mathrm{i}}{h}\Phi}=\mathrm{e}^{\frac{\mathrm{i}}{h}\Phi} (-\mathrm{i}h^{\alpha-1})^n \prod_{i=1}^n \|\theta^{(i)}\|^2.\]
	By performing $n$ integrations by parts, we obtain
	\begin{multline} \label{eq: int by parts}
	\mathbf I =\frac{(\mathrm{i}h)^{n(1-\alpha)}}{\prod_{1\le i\le n} \|\theta^{(i)}\|^2}\int_{[-1,1]^{J_h}} \mathrm{e}^{\frac{\mathrm{i}}{h}\Phi(\omega)} (\theta^{(1)}\cdot \nabla_{\omega})\cdots ( \theta^{(n)}\cdot \nabla_\omega) \big(\mathrm{e}^{\frac{\mathrm{i}}{h} R(\omega)} \bm m(\omega)\big) \mathrm{d}\omega.
	\end{multline}
	We set $\bm m^i(\omega):=\prod_{j\in \widetilde J_h(\gamma_i)} m(\omega_i)$, and $\widetilde{\bm m}(\omega):=\prod_{j\notin \bigsqcup_i \widetilde J_h(\gamma_i)} m(\omega_j)$. One can then rewrite
	\[\mathrm{e}^{\frac{\mathrm{i}}{h} R(\omega)}\bm m(\omega)= \Big(\prod_{1\le i\le n} \mathrm{e}^{\frac{\mathrm{i}}{h} r_{\omega}(\gamma_i)} \bm m^i(\omega)\Big)\widetilde {\bm m}(\omega).\]
	By Lemma \ref{lem: derivees r omega zero} and the fact that the vectors $\theta^{(i)}$ have disjoint support, the derivative acts diagonally on the product:
	\[(\theta^{(1)}\cdot \nabla_{\omega})\cdots (\theta^{(n)}\cdot \nabla_\omega) \big(\mathrm{e}^{\frac{\mathrm{i}}{h} R(\omega)} \bm m(\omega)\big)=\Big(\prod_{i=1}^n (\theta^{(i)}\cdot \nabla_{\omega}) (\mathrm{e}^{\frac{\mathrm{i}}{h}r_\omega(\gamma_i)} \bm m^i(\omega))\Big) \widetilde {\bm m}(\omega).\]
	For $i\in \{1,\ldots,n\}$, define
	\[I^{(i)}:=\int_{[-1,1]^{\widetilde J_h(\gamma_i)}} \big|(\theta^{(i)}\cdot \nabla_{\omega}) \mathrm{e}^{\frac{\mathrm{i}}{h}r_\omega(\gamma_i)}\bm m^i(\omega)\big| \mathrm{d}\omega.\]
	Then, by \eqref{eq: esperance is integral} and \eqref{eq: int by parts}, 
	\begin{equation} \label{eq: bound esp by prod}|\mathbf I|\le h^{n(1-\alpha)} \prod_{1\le i\le n}\frac{|I^{(i)}|}{\|\theta^{(i)}\|^2}.\end{equation}
	We control each of the $|I^{(i)}|$ separately. The end of the argument is the same as in \cite{ingremeau2023emergence} but we reproduce it here. By the triangle inequality, one has
	\[|I^{(i)}|\le \int \big|h^{-1}\theta^{(i)}\cdot \nabla_{\omega} r_\omega(\gamma_i)|\bm m^i(\omega)\mathrm{d}\omega+\int |\theta^{(i)}\cdot \nabla_\omega \bm m^i(\omega)|\mathrm{d}\omega=:I_1^{(i)}+I_2^{(i)}.\]
	To deal with $I_1^{(i)}$, we use Proposition \ref{prop: bounds derivatives omega r} and the second part of Lemma \ref{lem: Jh(gamma)}, together with the Cauchy--Schwarz inequality, to obtain:
	\[|\theta^{(i)}\cdot \nabla_\omega r_\omega(\gamma_i)|\le \|\theta^{(i)}\|\cdot \Big(\sum_{j\in J_h} |\partial_{\omega_j} r_{\omega}(\gamma_i)|^2\Big)^{1/2}\le C_n\|\theta^{(i)}\| h^{2\alpha-\frac{3}{2}\beta}\mathrm{e}^{CL}.\]
This leads to 
	\begin{equation} \label{eq: bound I(i)1}|I_1^{(i)}|\le C_n\|\theta^{(i)}\| h^{2\alpha-1-\frac{3}{2}\beta}\mathrm{e}^{CL}.\end{equation}
	We now deal with $I_2^{(i)}$. One has
	\[(\theta^{(i)}\cdot \nabla_{\omega}) \bm m^i(\omega)=\sum_{j\in \widetilde J_h(\gamma_i)} \theta^{(i)}_j \frac{m'(\omega_j)}{m(\omega_j)}\bm m^i(\omega).\]
	By Cauchy--Schwarz inequality, we find
	\[|I_2^{(i)}|\le \Big(\int_{[-1,1]^{\widetilde J_h(\gamma_i)}} \Big|\sum_{j\in \widetilde J_h(\gamma_i)} \theta^{(i)}_j \frac{m'(\omega_j)}{m(\omega_j)}\Big|^2 \bm m^i(\omega)\mathrm{d}\omega\Big)^{1/2}.\]
	We develop the square:
	\begin{align*}\Big|\sum_{j\in \widetilde J_h(\gamma_i)} \theta^{(i)}_j \frac{m'(\omega_j)}{m(\omega_j)}\Big|^2\bm m^i(\omega)& =\sum_{j\in \widetilde J_h(\gamma_i)} (\theta^{(i)}_j)^2 \frac{m'(\omega_j)^2}{m(\omega_j)}\prod_{\ell\neq j} m(\omega_\ell) \\ & +\sum_{\substack{j,k\in \widetilde J_h(\gamma_i) \\ j\neq k}} \theta^{(i)}_j\theta^{(i)}_k m'(\omega_j)m'(\omega_k)\prod_{\ell\notin\{j,k\}} m(\omega_\ell).
	\end{align*}
	The second sum integrates to $0$ over $[-1,1]^{\widetilde J_h(\gamma_i)}$ because for $j\neq k$ we have
	\[\int m'(\omega_j)m'(\omega_k)\mathrm{d}\omega_j\mathrm{d}\omega_k=0.\]
	To control the first sum, we observe that $(m')^2/m\le 4 \|m''\|_\infty$, which leads to
	\begin{equation} \label{eq: bound on I(i)2} |I^{(i)}_2|\le \Big(\int_{[-1,1]^{\widetilde J_h(\gamma_i)}} \sum_{j\in \widetilde J_h(\gamma_i)} (\theta^{(i)}_j)^2 \frac{m'(\omega_j)^2}{m(\omega_j)}\prod_{\ell\neq j} m(\omega_\ell) \mathrm{d}\omega\Big)^{1/2}\lesssim \|\theta^{(i)}\|.\end{equation}
	By combining \eqref{eq: bound I(i)1} with \eqref{eq: bound on I(i)2}, we find $|I^{(i)}|\lesssim \|\theta^{(i)}\|(1+h^{2\alpha-1-3/2\beta}\mathrm{e}^{CL})$. Recalling \eqref{eq: bound esp by prod}, we conclude
	\[\Big|\mathbf E_\omega\Big[\exp\Big(\frac{\mathrm{i}}{h}\sum k_i S_{\omega}(\gamma_i)\Big)\Big]\Big|\lesssim (h^{1-\alpha-\beta/2}+h^{\alpha-2\beta}\mathrm{e}^{CL})^n.\]
	So, it is enough to take $\mathfrak c(\alpha,\beta,\varepsilon)$ small enough to ensure $h^{1-\alpha-\beta/2}+h^{\alpha-2\beta}\mathrm{e}^{CL}\le h^\varepsilon$ when $L\le \mathfrak c|\log h|$.
	
	\textbf{2.} Let us explain the modifications in the case where the perturbation strongly breaks time reversal symmetry. We let $\gamma_1,\ldots,\gamma_n$ be as above with $\gamma_j\notin \{\gamma_i,\gamma_i^{-1}\}$, and take integers $k_i^+,k_i^-$. This time, we let
		\[\Phi(\omega)=-h^\alpha\sum_{j\in J_h}\omega_j\Big(\sum_{i=1}^n k_i^+\int_{\gamma_i} q_j+k_i^-\int_{\gamma_i^{-1}} q_j\Big).\]
		One sets $\theta^{(i)}=\big(\mathbf 1_{j\in \widetilde J_h(\gamma_i)}(k_i^+\theta_j(\gamma_i)+k_i^-\theta_j(\gamma_i^{-1}))\big)_{j\in J_h}$. One can check that when the pair $\{k_i^+,k_i^-\}$ is of the form $\{-1,1\}$ or $\{0,k\}$ with $k\neq 0$, then
		\[|k_i^+\theta_j(\gamma)+k_i^-\theta_j(\gamma^{-1})|\ge \min(|\theta_j(\gamma)|,|\theta_j(\gamma^{-1})|, |\theta_j(\gamma)-\theta_j(\gamma^{-1})|).\]
		Using \eqref{eq: defi time sym break} and the same arguments as in \textbf{1.}, we deduce that when $\{k_i^+,k_i^-\}$ is of the above form, we have $\|\theta^{(i)}\|^2\gtrsim h^\beta$. When $\{k_i^+,k_i^-\}$ is not of the above form, we cannot say anything. By repeating the end of the proof of \textbf{1.}, we deduce
		\[\Big|\mathbf E_\omega\Big[\exp\Big(\frac{\mathrm{i}}{h}\sum_{1\le i\le n} (k_i^+ S_{\omega}(\gamma_i)+k_i^- S_\omega(\gamma_i^{-1})\Big)\Big]\Big|\le C_n h^{\varepsilon N(\vec k^{\pm})},\]
		where $N(\vec k^{\pm})$ counts the number of pairs $\{k_i^+,k_i^-\}$ of the form $\{0,k\}$ with $k\neq 0$ or $\{-1,1\}$.
\end{proof}

\section{Proof of Proposition \ref{prop: moments N_osc}: perturbations preserving time reversal symmetry}\label{sec: proof T sym} This section is devoted to the proof of Proposition \ref{prop: moments N_osc} in the case where the perturbations preserve time-reversal symmetry in the sense of Definition \ref{defi: preserve T sym}. We introduce the shorthand $a_{\gamma,\varepsilon}:=\frac{T_\gamma^\sharp}{|I-P_{\gamma}|^{1/2}} \widehat f\Big(\varepsilon\frac{T_\gamma}{L}\Big)$ and also set $A_\gamma=|a_{\gamma,+}|=|a_{\gamma,-}|$. We rewrite the oscillating term in the trace formula
\[N_{\rm osc}(\omega,\mu)=\frac 1L\sum_{\varepsilon=\pm 1}\sum_{\gamma\in \mathcal P} a_{\gamma,\varepsilon} \mathrm{e}^{\frac{\mathrm{i}}{h} \varepsilon S_\omega(\gamma)}\mathrm{e}^{\mathrm{i}\mu \varepsilon T_\gamma/L}.\]
Notice that the summands take the same value for $\gamma$ and $\gamma^{-1}$. Let $\mathcal P_0$ denote the set of primitive periodic orbits quotiented by the relation $\gamma\sim \gamma^{-1}$. By writing each periodic orbit as a power of a primitive one, we are led to
\begin{equation} \label{eq: new exp nosc}N_{\rm osc}=\frac 2L\sum_{\varepsilon=\pm 1}\sum_{\gamma\in \mathcal P_0}\sum_{k\ge 1} a_{\gamma^k,\varepsilon} \mathrm{e}^{\frac{\mathrm{i}}{h} \varepsilon k S_\omega(\gamma)}\mathrm{e}^{\mathrm{i}\mu \varepsilon k T_\gamma/L}.\end{equation}
Taking \eqref{eq: new exp nosc} to the $n$-th power yields
\begin{equation} \label{eq: Nosc power n}N_{\rm osc}^{n}=\frac{2^n}{L^n}\sum_{\mathbf t=(\vec \gamma,\vec k,\vec \varepsilon)} \Big(\prod_{1\le i\le n} a_{\gamma_i^{k_i},\varepsilon_i}\Big) \exp\Big(\frac{\mathrm{i}}{h}\sum_{1\le i\le n} \varepsilon_ik_i S_{\gamma_i}\Big) \exp\Big(\mathrm{i}\mu\sum_{1\le i\le n} \varepsilon_ik_iT_{\gamma_i}/L\Big).\end{equation}
The sum runs over tuples $\mathbf t=(\vec \gamma,\vec k,\vec \varepsilon)$ with $\vec \gamma\in \mathcal P_0^n$, $\vec k\in (\mathbf N^*)^n$ and $\vec \varepsilon\in \{\pm 1\}^n$.
\begin{defi}\label{defi: tuples} A tuple $\mathbf t$ is \emph{diagonal} if $\sum \varepsilon_i k_i \gamma_i\equiv 0$ in the $\mathbf Z$-module $\mathbf Z[\mathcal P_0]$. A tuple $\mathbf t$ is \emph{primitive} if $\vec k=(1,\ldots,1)$ and if each $\gamma_i$ appears at most twice in the tuple $\vec \gamma$.
\end{defi}
The terminology is motivated by the fact that for $n=2$, a pair $(\gamma_1,k_1,\varepsilon_1;\gamma_2,k_2,\varepsilon_2)$ is diagonal if and only if $\gamma_1=\gamma_2$ and $k_1\varepsilon_1=-k_2\varepsilon_2$. When $\mathbf t$ is a diagonal tuple, the phase in \eqref{eq: Nosc power n} is trivial. We split the sum \eqref{eq: Nosc power n} into a \say{diagonal} and an \say{off-diagonal} part:
\begin{equation} \label{eq: def N osc diag}N_{\rm osc,\rm{diag}}^n:=\frac{2^n}{L^n} \sum_{\mathbf t~ \text{diag}}\Big(\prod_{1\le i\le n} a_{\gamma_i^{k_i},\varepsilon_i}\Big), \end{equation}
\begin{equation} \label{eq: def N osc off} N_{\rm osc,\rm{off}}^n:=\frac{2^n}{L^n} \sum_{\mathbf t~\text{nondiag}}\Big(\prod_{1\le i\le n} a_{\gamma_i^{k_i},\varepsilon_i}\Big)\exp\Big(\frac{\mathrm{i}}{h}\sum_{1\le i\le n} \varepsilon_ik_i S_\omega(\gamma_i)\Big) \exp\Big(\mathrm{i}\mu\sum_{1\le i\le n} \varepsilon_ik_iT_{\gamma_i}/L\Big).\end{equation}

\subsection{The diagonal part} \label{subsec: diagonal part} We first deal with the diagonal part \eqref{eq: def N osc diag}. Note that if $\mathbf t$ is diagonal then each $\gamma_i$ appears at least twice in $\vec \gamma$. Thus one can write $\{\gamma_1,\ldots,\gamma_n\}=\{g_1,\ldots,g_j\}$ with $g_1,\ldots,g_j\in \mathcal P_0$, for some $j\le n/2$.

\subsubsection{Discarding nonprimitive contributions} \label{subsec: discard nonprim 1} We denote by $\mathcal S_{\rm np}^{\rm diag}$ the subsum of \eqref{eq: def N osc diag} that runs over nonprimitive tuples. We show that it has a negligible contribution when $L\to +\infty$.

\begin{defi}For each tuple $\mathbf t=(\vec \gamma,\vec k,\vec \varepsilon)$, we define a map $K_{\mathbf t}:\mathcal P_0\to \{\text{multisets of integers}\}$, by 
\[K_{\mathbf t}(\gamma):=\big\{k_i~:~i\in \{1,\ldots,n\}\text{ such that } {\gamma_i=\gamma}\big\}.\]
By convention, $\operatorname{Supp} K_{\mathbf t}:=\{\gamma~:~\# K_{\mathbf t}(\gamma)\neq 0\}$. We let $|K_{\mathbf t}(\gamma)|:=\sum_{k\in K_{\mathbf t}(\gamma)} k$. \end{defi} When $\mathbf t$ is a diagonal tuple, the following holds:
\begin{itemize}
\item if $K_{\mathbf t}(\gamma)\neq \varnothing$ then $\# K_{\mathbf t}(\gamma)\ge 2$.
\item $\sum_\gamma \# K_{\mathbf t}(\gamma)=n$.
\item if $\mathbf t$ is nonprimitive, then there is some $\gamma$ with $|K_{\mathbf t}(\gamma)|\ge 3$.
\end{itemize}
We can then write
\[|\mathcal S_{\rm np}^{\rm diag}|\le C_nL^{-n} \sum_{\substack{\mathbf t ~\text{diagonal} \\ \text{and nonprimitive}}}\prod_{\gamma\in \mathcal P_0}\prod_{k\in K_{\mathbf t}(\gamma)} A_{\gamma^{k}}.\]
For each function $K:\mathcal P_0\to \{\text{multisets of integers}\}$ satisfying the three bullet properties above, there are at most $\mathcal O_n(1)$ tuples $\mathbf t=(\vec \gamma,\vec k,\vec \varepsilon)$ for which $K_{\mathbf t}=K$. Indeed, for each $i\in \{1,\ldots,n\}$ we have to choose $\gamma_i$ in $\operatorname{Supp}(K)$ which has size $\le n/2$, giving $\mathcal O_n(1)$ choices in total for $\vec \gamma$. Once $\vec \gamma$ is chosen there are $\mathcal O_n(1)$ choices for $\vec k$, and $\le 2^n$ choices for $\vec \varepsilon$. It follows that
\begin{align*}
|\mathcal S_{\rm np}^{\rm diag}| & \le C_nL^{-n} \sum_{K:\mathcal P_0\to \text{multisets}} \prod_{\gamma\in \operatorname{Supp}K}\prod_{k\in K(\gamma)} A_{\gamma^{k}}.\end{align*}
The sum runs over maps $K$ satisfying the three bullet properties above. Such a map is determined by the choice of the support of $K$, which is a set $\{g_1,\ldots,g_j\}$ of size $j\le n/2$, and the images $K_i=K(g_i)$ for $i\in \{1,\ldots,j\}$, where all $K_i$ satisfy $\#K_i\ge 2$ and one has $|K_1|\ge 3$ (up to changing labels). This leads to
\[|\mathcal S_{\rm np}^{\rm diag}|\le C_n L^{-n} \sum_{1\le j\le n/2} \sum_{g_1,\ldots,g_j}\sum_{\substack{K_1,\ldots,K_j, \\ \#K_i\ge 2, \  |K_1|\ge 3 \\ \sum \#K_i=n}} \prod_{1\le i\le j}\Big(\prod_{k\in K_i} A_{g_i^k}\Big).\]
At this point we invoke Lemma \ref{lem: poincare map powers} to obtain, for $\# K_i\ge 2$:
\[\prod_{k\in K_i} A_{g_i^k}\le C_n A_{g_i}^2  \mathrm{e}^{-c(|K_i|-2)T_{g_i}}.\]
It follows that
\begin{align}
|\mathcal S_{\rm np}^{\rm diag}|  & \le  C_n L^{-n} \sum_{1\le j\le n/2} \sum_{g_1,\ldots,g_j}\Big(\prod_{1\le i\le j} A_{g_i}^2\Big)\cdot \sum_{\substack{K_1,\ldots,K_j, \\ \#K_i\ge 2, \  |K_1|\ge 3 \\ \sum \# K_i=n}} \Big(\prod_{1\le i\le j} \mathrm{e}^{-c(|K_i|-2)T_{g_i}}\Big) \nonumber \\
& \le C_n L^{-n}\sum_{1\le j\le n/2} \sum_{g_1,\ldots,g_j} \Big(\prod_{1\le i\le j} A_{g_i}^2\Big) \mathrm{e}^{-cT_{g_1}} \nonumber \\
&\le C_n L^{-n}\sum_{1\le j\le n/2} \Big(\sum_{g_1} A_{g_1}^2\mathrm{e}^{-cT_{g_1}}\Big)\Big(\sum_{g} A_g^2\Big)^{j-1} \nonumber \\
& \le C_n L^{-n} L^{2\lfloor n/2-1\rfloor} \le C_nL^{-2}. \label{eq: bound Snp}\end{align}
To pass from the third line to the fourth line we have used the fact that $\sum_g A_g^2\lesssim L^2$ and $\sum_g A_g^2\mathrm{e}^{-cT_g}\lesssim 1$, which follows from \eqref{eq: formula for GOE}.

\subsubsection{Dealing with primitive contributions} We now study the subsum of \eqref{eq: def N osc diag} that consists of primitive diagonal tuples $\mathbf t=(\vec \gamma,\vec k,\vec \varepsilon)$ in the sense of Definition \ref{defi: tuples}, for which each $\gamma_i$ appears exactly twice in $\vec \gamma$, and $k_i=1$ for all $i$. This part is denoted by $\mathcal S_{\rm p}^{\rm diag}$. We can assume that $n=2p$ is even otherwise $\mathcal S_{\rm p}^{\rm diag}=0$. If $\mathbf t$ is a primitive diagonal tuple, then:
\begin{itemize}
\item $\operatorname{Supp} K_{\mathbf t}$ has cardinal $p$.
\item $\# K_{\mathbf t}(\gamma)=|K_{\mathbf t}(\gamma)|=2$ on $\operatorname{Supp}(K_{\mathbf t})$.
\end{itemize} 
Actually, $K_{\mathbf t}$ is completely determined by the function $\# K_{\mathbf t}$, which is in turn uniquely determined by its support, which has size $p$. For any subset $\Gamma\subset \mathcal P_0$ of cardinal $p$, there are $(2p)!$ primitive diagonal tuples $\mathbf t$ such that $\operatorname{Supp} K_{\mathbf t}=\Gamma$. Indeed, there are
\[\binom{2p}{2}\binom{2p-2}{2}\cdots\binom{2}{2}=\frac{(2p)!}{2^p}\]
choices for the tuple $(\gamma_1,\ldots,\gamma_{2p})$; and $2^p$ choices for $(\varepsilon_1,\ldots,\varepsilon_{2p})$ associated with it. This leads to
\[\mathcal S_{\rm p}^{\rm diag}=\frac{2^{2p}}{L^{2p}}\frac{(2p)!}{p!}\sum_{\substack{(\gamma_1,\ldots,\gamma_p)\in \mathcal P_0^p \\ \text{all distinct}}} \Big(\prod_{1\le i\le p}A_{\gamma_i}^2\Big).\]
We have divided by $p!$ to account for the fact that there are $p!$ bijective maps $\{1,\ldots,p\}\to \Gamma$. We deduce from \eqref{eq: formula for GOE} that
\begin{equation}
	\label{eq: bound Sp} \mathcal S_{\rm p}^{\rm diag}=\frac{(2p)!}{2^pp!} \Big(8L^{-2}\sum_{\gamma\in \mathcal P_0} A_\gamma^2\Big)^p+\mathcal O_p(L^{-2})=\frac{(2p)!}{2^pp!}\Sigma_{\rm GOE}^{2p}+\mathcal O_p(L^{-2}).
\end{equation}
Combining \eqref{eq: bound Snp} with \eqref{eq: bound Sp} leads to the following proposition.
\begin{prop} \label{prop: esperance diag} For all $n\ge 1$,
\[N_{\rm osc, \rm diag}^{n}=m_{n}+\mathcal O_n(L^{-2}).\]
\end{prop}
Note that until now, we have not exploited any cancellation mechanism.

\subsection{The off-diagonal part} We now deal with the nondiagonal part \eqref{eq: def N osc off}. By integrating \eqref{eq: def N osc off} over $\mu$, we obtain
\[\mathbf E_\mu(N_{\rm osc, \rm off}^n)=\sum_{\substack{\mathbf t=(\vec \gamma,\vec k,\vec \varepsilon)\\ \text{non-diagonal}}} \prod_{1\le i\le n} a_{\gamma_i^{k_i},\varepsilon_i} \exp\Big(\frac{\mathrm{i}}{h}\sum_{1\le i\le n} \varepsilon_i k_iS_\omega(\gamma_i)\Big)\operatorname{sinc}\Big(\Lambda\sum_{1\le i\le n} \varepsilon_ik_iT_{\gamma_i}/L\Big).\]
We introduce a refined version of the map $K_{\mathbf t}$ defined in the previous subsection. 

\begin{defi}\label{defi: refined K tuple} Let $\mathbf t=(\vec \gamma,\vec k,\vec \varepsilon)$.

\begin{itemize}
\item Define $K^{\pm}_{\mathbf t}(\gamma):=\{k_j~:~ \gamma_j = \gamma\ \text{and} \ \varepsilon_j=\pm 1\}$. 
\item Set $q^{\pm}_{\mathbf t}(\gamma):=\# K^{\pm}_{\mathbf t}(\gamma)$ and $k^{\pm}_{\mathbf t}(\gamma):=|K^{\pm}_{\mathbf t}(\gamma)|$
\item Lastly, set $q_{\mathbf t}:=q_{\mathbf t}^{+}+q_{\mathbf t}^{-}$, $k_{\mathbf t}=k_{\mathbf t}^++k_{\mathbf t}^-$, and $\Delta k_{\mathbf t}=k_{\mathbf t}^+-k_{\mathbf t}^-$.
\end{itemize}
One has $\operatorname{Supp}(k_{\mathbf t}^{\pm})=\operatorname{Supp}(q_{\mathbf t}^{\pm})$ and $\sum_\gamma q_{\mathbf t}(\gamma)=n$.\end{defi}
For any pair of maps $K^{\pm}:\mathcal P_0\to \{\text{multisets of integers}\}$, we associate functions $k^{\pm},q^{\pm},\Delta k$ as above. Introducing the shorthand $a_{K^{\pm}}(\gamma):= \prod_{k\in K^+(\gamma)} a_{\gamma^k,1}\prod_{k\in K^-(\gamma)} a_{\gamma^k,-1}$ (with the convention that an empty product equals $1$), one can then rewrite
\begin{equation} \label{eq: new exp N osc using aK}\mathbf E_\mu(N_{\rm osc, \rm off}^n)=\sum_{\mathbf t\ \ \text{nondiag}} \Big(\prod_\gamma a_{K_{\mathbf t}}(\gamma)\Big)\exp\Big(\frac{\mathrm{i}}{h}\sum_\gamma \Delta k_{\mathbf t}(\gamma)S_\omega(\gamma)\Big)\operatorname{sinc}\Big(\Lambda\sum_{\gamma}\Delta k_{\mathbf t}(\gamma)T_\gamma/L\Big).\end{equation}
We take the modulus squared of \eqref{eq: new exp N osc using aK}, then compute the $\omega$-average. By Proposition \ref{prop: averages actions}, given two tuples $\mathbf t,\mathbf t'$, we have
\[\mathbf E_\omega\Big[\exp\Big(\frac{\mathrm{i}}{h}\sum_\gamma \Delta k_{\mathbf t}(\gamma)S_\omega(\gamma)\Big)\exp\Big(-\frac{\mathrm{i}}{h}\sum_\gamma \Delta k_{\mathbf t'}(\gamma)S_\omega(\gamma)\Big)\Big]=\mathcal O_n\big(h^{\varepsilon \#\{\gamma~:~\Delta k_{\mathbf t}(\gamma)\neq \Delta k_{\mathbf t'}(\gamma)\}}\big).\]
Introducing another shorthand
\begin{equation} \label{eq: b kt kt'}b_{K_{\mathbf t}^{\pm},K_{\mathbf t'}^{\pm}}:=\operatorname{sinc}\Big(\Lambda\sum_{\gamma}\Delta k_{\mathbf t}(\gamma)T_\gamma/L\Big)\operatorname{sinc}\Big(\Lambda\sum_{\gamma}\Delta k_{\mathbf t'}(\gamma)T_\gamma/L\Big),\end{equation}
one finds
\begin{equation} \label{eq: intro another shorthand}\mathbf E_\omega\big[|\mathbf E_\mu^{\rm off}(N_{\rm osc}^n)|^2\Big]\le \frac{C_n}{L^{2n}} \sum_{\substack{\mathbf t,\mathbf t'\\ \text{nondiag}}}\Big(\prod_\gamma a_{K_{\mathbf t}}(\gamma)a_{K_{\mathbf t'}}(\gamma)\Big) h^{\varepsilon \#\{\gamma~:~\Delta k_{\mathbf t}(\gamma)\neq \Delta k_{\mathbf t'}(\gamma)\}} b_{K_{\mathbf t},K_{\mathbf t'}}. \end{equation}
For each pair of maps $K=(K^{\pm})$ satisfying the properties of Definition \ref{defi: refined K tuple}, there are $\mathcal O_n(1)$ tuples $\mathbf t=(\vec \gamma,\vec k,\vec \varepsilon)$ such that $K_{\mathbf t}^{\pm}=K^{\pm}$. One is led to
\begin{equation} \label{eq: new express N osc off n}\mathbf E_\omega\big[|\mathbf E_\mu^{\rm off}(N_{\rm osc}^n)|^2\Big]\le C_nL^{-2n} \sum_{K,K^{'}}\Big(\prod_{\gamma\in \mathcal P_0} a_{K}(\gamma)a_{K'}(\gamma)\Big) h^{\varepsilon \#\{\gamma~:~\Delta k(\gamma)\neq \Delta k'(\gamma)\}} b_{K,K'}.\end{equation}
The sum runs over maps $K=K^{\pm}\to \{\text{multisets of integers}\}$ satisfying the properties of Definition \ref{defi: refined K tuple} and similarly for $K'$. The right-hand side of \eqref{eq: new express N osc off n} is denoted by $\mathcal S^{\rm off}$. We will show that it becomes negligible in the large $L$ limit.

\subsubsection{The part $\Delta k\neq \Delta k'$}\label{subsec: Delta k neq Delta k'}

We first consider a \say{nondiagonal} part of $\mathcal S^{\rm off}$, denoted by $\mathcal S^{\rm off}_{\rm nd}$, consisting of pairs $(K,K')=(K^{\pm},{K'}^{\pm})$ for which $\Delta k\neq \Delta k'$. We use the trivial bound $b_{K^{\pm},K^{'\pm}}\le 1$. Up to replacing $\varepsilon$ by $\varepsilon/2$, since $\#\{\gamma~:~\Delta k(\gamma)\neq \Delta k'(\gamma)\}\ge 1$ for each $(K,K')$ in the sum, we have
\[|\mathcal S^{\rm off}_{\rm nd}| \le C_n L^{-2n} h^{\varepsilon}\sum_{\substack{(K,K')\\ \Delta k\neq \Delta k'}} \Big(\prod_{(q+q')(\gamma)\ge 1} a_{K}(\gamma)a_{K'}(\gamma)\Big) h^{\varepsilon \#\{\gamma~:~\Delta k(\gamma)\neq \Delta k'(\gamma)\}}.\]
Now, observe that in the notation of Definition \ref{defi: refined K tuple}, $(q+q')(\gamma)=1$ implies $|\Delta k(\gamma)-\Delta k'(\gamma)|=1$. When $(q+q')(\gamma)=1$ we simply bound $a_K(\gamma)a_{K'}(\gamma)\lesssim 1$. Also, by Lemma \ref{lem: poincare map powers}, we have
\[\Big(\prod_{(q+q')(\gamma)\ge 2} a_K(\gamma)a_{K'}(\gamma)\Big)\le C_n\prod_{(q+q')(\gamma)\ge 2}A_\gamma^2 \mathrm{e}^{-c(k(\gamma)+k'(\gamma))}.\]
We thus obtain
\begin{equation} \label{eq: we thus obtain bis} |\mathcal S^{\rm off}_{\rm nd}|\le C_n L^{-2n} h^{\varepsilon}\sum_{\substack{K,K'\\ \Delta k\neq \Delta k'}} \Big(\prod_{(q+q')(\gamma)\ge 2}A_\gamma^2 \mathrm{e}^{-c(k+k')(\gamma)}\Big)h^{\varepsilon \#\{\gamma~:~(q+q')(\gamma)=1\}}.\end{equation}
\begin{defi}\label{defi: Ns Qs}
To each pair $\mathbf s=(K,K')$ we associate new functions defined as follows:
\begin{itemize}
\item We set $N_{\mathbf s}:=(k+k')\mathbf 1_{q+q'\ge 2}$.
\item We set $Q_{\mathbf s}:=(q+q')\mathbf 1_{q+q'\ge 2}$.
\item For future use below we set $d_{\mathbf s}(\gamma):=\Delta k(\gamma) \mathbf 1_{(q+q')(\gamma)\ge 2}$.
\end{itemize}\end{defi}
Note that $\sum Q_{\mathbf s}(\gamma)=2n-\#\{\gamma~:~q(\gamma)+q'(\gamma)=1\}$. The triple $(N_{\mathbf s}, Q_{\mathbf s},d_{\mathbf s})$ enjoys the following properties:
\begin{enumerate}
\item \label{prop 1 Ns} $N_{\mathbf s}$ has the same support as $Q_{\mathbf s}$, of size $\le n$.
\item \label{prop 2 Qs}$\sum_\gamma Q_{\mathbf s}(\gamma)\le 2n$.
\item \label{prop 3 ds} $|d_{\mathbf s}(\gamma)|\le N_{\mathbf s}(\gamma)$.
\end{enumerate}
With this new notation, one can rewrite \eqref{eq: we thus obtain bis} as
\begin{equation} \label{eq: bound Acal1 intermediate}|\mathcal S^{\rm off}_{\rm nd}|\le C_nL^{-2n}h^\varepsilon\sum_{\mathbf s=(K,K')} \Big(\prod_{(q+q')(\gamma)\ge 2} A_{\gamma}^{2}\mathrm{e}^{-cN_{\mathbf s}(\gamma)}\Big)\cdot h^{\varepsilon(2n-\sum_\gamma Q_{\mathbf s}(\gamma))}.\end{equation}
For each pair $(N,Q)$ of functions satisfying the properties (\ref{prop 1 Ns}),(\ref{prop 2 Qs}),(\ref{prop 3 ds}) above, let us count the number of pairs $\mathbf s=(K,K')$ such that $(N,Q)=(N_{\mathbf s},Q_{\mathbf s})$. For each $\gamma\in \operatorname{Supp}(Q)$, we have to choose multisets $K^{\pm}(\gamma),K^{'\pm}(\gamma)$, knowing that
\[|K(\gamma)|+|K'(\gamma)|\le N(\gamma), \quad \# K(\gamma)+\#K'(\gamma)\le n.\]
The number of such choices is $\mathcal O_n(N(\gamma)^n)$. Outside $\operatorname{Supp}(Q)$, letting $r=2n	-\sum_\gamma Q(\gamma)$, we have to choose distinct elements $\gamma_1,\ldots,\gamma_r$ of period $\le L$, such that
\[\big(\operatorname{Supp}(K)\cup \operatorname{Supp}(K')\big)\backslash \operatorname{Supp}(Q)=\{\gamma_1,\ldots,\gamma_r\}.\]
There are $\mathcal O_n(\mathrm{e}^{CL r})$ such possible choices, because there are $\mathcal O(\mathrm{e}^{CL})$ periodic orbits of period $\le L$. Then, for each $i=1,\ldots,r$, we have to choose the values of $K^{+}(\gamma_i)$, $K^{-}(\gamma_i)$, $K^{'+}(\gamma_i)$, and $K^{'-}(\gamma_i)$. All of these four sets are empty except one that contains an integer $\le L$. Thus there are $\mathcal O_n(L^r)$ possible choices in total. Up to increasing the constant $C$ in the exponential, we get
\[\#\big\{\mathbf s=(K,K')~:~(N_{\mathbf s},Q_{\mathbf s})=(N,Q)\big\}\le C_n \mathrm{e}^{C L(2n-\sum_\gamma Q(\gamma))} \prod_\gamma (1+N(\gamma))^n.\]
Going back to \eqref{eq: bound Acal1 intermediate}, we deduce the following bound on $\mathcal S^{\rm off}_{\rm nd}$:
\[|\mathcal S^{\rm off}_{\rm nd}|\le C_nL^{-2n} h^\varepsilon\sum_{(N,Q)} \Big(\prod_{\gamma\in \operatorname{Supp}N} A_\gamma^{2}\mathrm{e}^{-cN(\gamma)}(1+N(\gamma)^{n})\Big)\cdot \mathrm{e}^{CL(2n-\sum Q(\gamma))} h^{\varepsilon(2n-\sum Q(\gamma))}.\]
We take $L\le \mathfrak c|\log h|$ with $\mathfrak c$ small enough to ensure $\mathrm{e}^{CL}h^\varepsilon<1$. Also, we use the fact that for each $N$ there are $\mathcal O_n(1)$ corresponding functions $Q$ such that the properties (\ref{prop 1 Ns}),(\ref{prop 2 Qs}), are satisfied. This leads to
\[|\mathcal S^{\rm off}_{\rm nd}|\le C_n L^{-2n}h^\varepsilon\sum_{N} \Big(\prod_{\gamma\in \operatorname{Supp}N} A_\gamma^{2}\mathrm{e}^{-cN(\gamma)}(1+N(\gamma)^{n})\Big),\]
where the sum runs over functions $N:\mathcal P_0\to \mathbf N$ with support of size $\le n$. Such a function is determined by the choice of its support $\{\gamma_1,\ldots,\gamma_k\}$ with $k\le n$, and the choice of the images $N_i=N(\gamma_i)$. It follows that
\begin{align}
|\mathcal S^{\rm off}_{\rm nd}|& \le C_nh^\varepsilon L^{-2n}\sum_{k=1}^n\sum_{(\gamma_1,\ldots,\gamma_k)}\Big(\prod_{1\le i\le k} A_{\gamma_i}^2\Big)\Big(\sum_{N_1,\ldots,N_k\ge 0} \prod_{1\le i\le k} \mathrm{e}^{-cN_i} (1+N_i)^n\Big) \nonumber\\
& \le C_nh^\varepsilon L^{-2n}\sum_{k=1}^n\Big(\sum_{\gamma} A_{\gamma}^2\Big)^{k} \cdot\Big(\sum_{N_1\ge 0} \mathrm{e}^{-cN_1}(1+N_1)^{n}\Big)^k \nonumber \\
& \le  C_nh^{\varepsilon}L^{-2n}\sum_{k=1}^n L^{2k}\le C_n h^\varepsilon, \label{eq: bound Acal1}
\end{align}
where the constant $C_n$ might have changed from line to line. To obtain the last line, we have used again the bound $\sum A_\gamma^2\lesssim L^2$ from \eqref{eq: formula for GOE}.
\subsubsection{The part $\Delta k=\Delta k'\neq 0$} We now consider pairs $(K,K')$ with $\Delta k=\Delta k'\neq 0$. Note that the case $\Delta k=\Delta k'=0$ corresponds to diagonal tuples, which have been dealt with in \S \ref{subsec: diagonal part} The corresponding sum is denoted by $\mathcal S^{\rm off}_{\rm d}$. For these pairs there is no cancellation mechanism coming from Proposition \ref{prop: averages actions}, but there is a gain coming from the energy averaging. Start from
\begin{equation} \label{eq: start from A2 equal}\mathcal S^{\rm off}_{\rm d}\le C_nL^{-2n}\sum_{\substack{(K,K') \\ \Delta  k=\Delta k'\neq 0}}\Big(\prod_\gamma a_K(\gamma)a_{K'}(\gamma)\Big)\operatorname{sinc}\Big(\frac{\Lambda}{L}\sum_\gamma \Delta k(\gamma)T_\gamma\Big)^2, \end{equation}
where the sum runs over maps $K^{\pm},K^{'\pm}$ such that $\operatorname{Supp}(k^{\pm})=\operatorname{Supp}q^{\pm}$ and $\sum q(\gamma)=n$, and similarly for $k^{'\pm},q^{'\pm}$. Observe that $\Delta k=\Delta k'\neq 0$ forces $q+q'\ge 2$ on $\operatorname{Supp}(q+q')$. In particular, $k+k'$ has support of size $\le n$, and $k+k'\ge 2$ on its support.

Let $\alpha\ge 1$, that will be chosen later. We consider a dyadic decomposition of the sum on the right-hand side of \eqref{eq: start from A2 equal}.
\begin{itemize}
\item A first piece $\mathcal S_{-1}$ consists of pairs $(K,K')$ with $|\sum_\gamma \Delta k(\gamma)T_\gamma|\le \alpha$.
\item For $j\ge 0$, we let $\mathcal S_j$ be the subsum associated with pairs $(K,K')$ such that  $|\sum_\gamma \Delta k(\gamma)T_\gamma|\in (2^j\alpha,2^{j+1}\alpha]$. 
\end{itemize}
Then, $\mathcal S_{\rm d}^{\rm off}=\mathcal S_{-1}+\sum_{j=0}^{+\infty}\mathcal S_j$. We control each of these terms separately.
\subsubsection{Controlling $\mathcal S_{-1}$} \label{subsec: controlling S-1} We use the functions $(N_{\mathbf s},Q_{\mathbf s},d_{\mathbf s})$ from Definition \ref{defi: Ns Qs}. Since $q+q'\ge 2$ on its support we have $Q_{\mathbf s}=q+q'$ and $N_{\mathbf s}=k+k'$. Then, we bound successively:
\begin{equation} \label{eq: bounds S -1}\begin{aligned}
		\mathcal S_{-1}& \le C_nL^{-2n}\sum_{\substack{(K,K') \\ \Delta  k=\Delta k'\neq 0}}\Big(\prod a_K(\gamma)a_{K'}(\gamma)\Big)\mathbf 1_{|\sum \Delta k(\gamma)T_\gamma|\le \alpha} \\
		& \le C_nL^{-2n}\sum_{\mathbf s=(K,K'), \ d_{\mathbf s}\neq 0} \Big(\prod_{\gamma\in \operatorname{Supp}(N_{\mathbf s})} A_{\gamma}^2 \mathrm{e}^{-cN_{\mathbf s}(\gamma)}\Big)\mathbf 1_{|\sum d_{\mathbf s}(\gamma)T_\gamma|\le \alpha} \\
		&\le C_n L^{-2n} \sum_{N,d\neq 0} \Big(\prod_{\gamma\in \operatorname{Supp}(N)} A_{\gamma}^2 \mathrm{e}^{-cN(\gamma)}(1+N(\gamma))^n\Big)\mathbf 1_{|\sum d(\gamma)T_\gamma|\le \alpha} \\
		&\le C_n L^{-2n} \sum_{k=1}^n\sum_{(\gamma_1,\ldots,\gamma_k)}\sum_{\substack{N_1,\ldots,N_k\ge 0 \\ d_1,\ldots,d_k \\ |d_i|\le N_i, d_1\neq 0}}\prod_{1\le i\le k} A_{\gamma_i}^2 \mathrm{e}^{-cN_i}(1+N_i)^n \mathbf 1_{|\sum d_iT_{\gamma_i}|\le \alpha}.
\end{aligned}\end{equation}

In the third line, the sum runs over pairs of functions $(N,d):\mathcal P_0\to \mathbb N^2$ such that $\operatorname{Supp}(N)$ has size $\le n$, $N(\gamma)\ge 2$ on the support, $|d(\gamma)|\le N(\gamma)$ and $d\neq 0$. We have used the fact that for each pair $(N,d)$ with these properties, there are $\mathcal O(\prod_\gamma (1+N(\gamma)^n))$ pairs $\mathbf s$ such that $(N_{\mathbf s},d_{\mathbf s})=(N,d)$. In the last line, we have parameterized such pairs of maps $(N,d)$ by the choice of the support of $N$, denoted $\{\gamma_1,\ldots,\gamma_k\}$, which has size $k\le n$; and the choice of $N_i:=N(\gamma_i)$ and $d_i=d(\gamma_i)$. The restriction $d_1\neq 0$ in the sum comes from the fact that the map $d$ is nonzero in the third line.

Given tuples $(N_i)_{1\le i\le k}$ and $(d_i)_{1\le i\le k}$ satisfying $|d_i|\le N_i$ and $d_1\neq 0$, we estimate
\begin{align*}
\sum_{(\gamma_1,\ldots,\gamma_k)\in \mathcal P_0^k} \Big(\prod_{1\le i\le k} A_{\gamma_i}^2\Big)\mathbf 1_{|\sum d_iT_{\gamma_i}|\le \alpha} & \le  \sum_{\gamma_2,\ldots,\gamma_k} \Big(\prod_{2\le i\le k}A_{\gamma_i}^2\Big) \sum_{\gamma_1~:~T_{\gamma_1}\le L} A_{\gamma_1}^2\mathbf 1_{|T_{\gamma_1}+\sum_i \frac{d_i}{d_1}T_{\gamma_i}|\lesssim \frac{\alpha}{d_1}} \\
& \lesssim \sum_{\gamma_2,\ldots,\gamma_k} \Big(\prod_{2\le i\le k}A_{\gamma_i}^2\Big)\cdot L\alpha \lesssim L^{2(k-1)}L\alpha.
\end{align*}
To pass from the first to second line we have used the following consequence of Proposition \ref{prop: equidistribution result}:
\begin{equation} \label{eq: we have used}\forall \alpha,T\ge 1, \ \sum_{T_\gamma\in [T,T+\alpha], \ T_\gamma\le L} A_\gamma^2\lesssim L\alpha.\end{equation}
Going back to \eqref{eq: bounds S -1}, we have
\begin{align*}\mathcal S_{-1}& \lesssim C_n L^{-2n} \sum_{k=1}^n L^{2k-1}\alpha \sum_{\substack{N_1,\ldots,N_k\ge 0 \\ d_1,\ldots,d_k, \ |d_i|\le N_i}} \Big(\prod_{1\le i\le k}\mathrm{e}^{-cN_i}(1+N_i)^n\Big) \\
& \lesssim C_nL^{-2n} L^{2n-1}\alpha\le C_n L^{-1}\alpha. \end{align*}
\subsubsection{Controlling $\mathcal S_j$ for $j\ge 0$} \label{subsec: controlling Sj} To deal with $\mathcal S_j$, we repeat the calculations carried above for $\mathcal S_{-1}$, but instead use the upper bound
\[\operatorname{sinc}\Big(\frac{\Lambda}{L}|\sum \Delta k(\gamma)T_\gamma|\Big)^2\le \Big(\frac{L}{\Lambda 2^j\alpha}\Big)^2,\]
which is valid for $|\sum \Delta k(\gamma)T_\gamma|\ge 2^j\alpha$; along with the estimate
\begin{equation} \label{eq: sum intermediaire} \sum_{T_\gamma\in [T,T+2^j\alpha], \ T_\gamma\le L} A_\gamma^2\lesssim L2^j\alpha,\end{equation}
which is used in place of \eqref{eq: we have used}. This leads to
\[|\mathcal S_j|\le \Big(\frac{L}{\Lambda 2^j\alpha}\Big)^2\cdot \frac{2^j\alpha}{L}=\frac{L}{\Lambda^2 2^j\alpha}.\]
Summing over $j$ leads to
\[|\mathcal S^{\rm off}_{\rm d}|\le |\mathcal S_{-1}|+\sum_{j=0}^{+\infty}|\mathcal S_j|\le C_n \Big(\frac{\alpha}{L}+\frac{L}{\Lambda^2 \alpha}\Big).\]
We optimize the right-hand side by taking $\alpha=\frac{L}{\Lambda}$ which is $\ge 1$ by assumption, which gives
\begin{equation} \label{eq: bound Acal2}|\mathcal S^{\rm off}_{\rm d}|\le \frac{C_n}{\Lambda}.\end{equation}
Combining \eqref{eq: bound Acal1} with \eqref{eq: bound Acal2}, we end up with the following proposition.
\begin{prop} \label{prop: the nondiag cont} Assume $1\le \Lambda\le L\le \mathfrak c|\log h|$. The nondiagonal contribution is controlled by
\begin{equation} \label{eq: bound E omega E mu N osc} \mathbf E_\omega(|\mathbf E_\mu(N_{\rm osc,\rm off}^n)|^2) \le C_n\Big(\frac{1}{\Lambda}+h^\varepsilon\Big)\le \frac{C_n}{\Lambda}. \end{equation}\end{prop}
\subsubsection{End of the proof of Proposition \ref{prop: moments N_osc}} 
We combine Proposition \ref{prop: esperance diag} with Proposition \ref{prop: the nondiag cont} and use the fact that $1\le \Lambda\le L$ to obtain
\begin{align*} \mathbf E_\omega\big(|\mathbf E_\mu(N_{\rm osc}^n)-m_n|^2\big) & \le 2\mathbf E_\omega\big(\mathbf E_\mu(|N_{\rm osc,\rm diag}^n-m_n|^2)\big)+2\mathbf E_\omega\big(|\mathbf E_\mu(N_{\rm osc,\rm off}^n)|^2\big) \\
& \le C_n\big(\Lambda^{-1}+L^{-2}\big)\le C_n \Lambda^{-1}.
\end{align*}

\subsubsection{Deducing Theorem \ref{thm: as GOE fluctuations} from Proposition \ref{prop: moments N_osc}} We now explain how to obtain Theorem \ref{thm: as GOE fluctuations} from Proposition \ref{prop: moments N_osc}.  By Proposition \ref{prop: moments N_osc}, there is a family of constants $(C_k)$ such that for all $k\ge 1$, one has
\[\mathbf E_\omega(|\mathbf E_\mu(N_{\rm osc}^k)-m_k|^2)\le \frac{C_k}{\Lambda}.\]
Fix another family $(D_k)$. Let $\delta>0$. By the Markov inequality, for all $k$ there is a set $B_k$ of probability $\ge 1-\frac{C_k}{\Lambda D_k\delta}$ such that
\[\forall \omega\in B_k, \qquad |\mathbf E_\mu(N_{\rm osc}^k)-m_k|^2\le D_k\delta.\]
Take $D_k=2^kC_k$ then set $B=\bigcap_{k\ge 1} B_k$. One has $\mathbf P_\omega(B)\ge 1-\frac{1}{\Lambda \delta}$, moreover,
\[\forall \omega\in B, \forall k\ge 1, \qquad |\mathbf E_\mu(N_{\rm osc}^k)-m_k|^2\le D_k\delta.\]
Recall that $\widetilde N_{\rm osc}=N_{\rm osc}-\mathbf E_\mu(N_{\rm osc})$. If $\omega\in B$ then $|\mathbf E_\mu(N_{\rm osc})|^2\le D_1\delta$. Also, recall from Proposition \ref{prop: long time gutz} that $|\widetilde N-\widetilde N_{\rm osc}|\lesssim \frac{\Lambda}{L}$. Thus, up to increasing the constants $D_k$, one finds
\[\forall \omega\in B, \forall k\ge 1, \qquad |\mathbf E_\mu(\widetilde N^k)-m_k|^2\le D_k\Big(\delta+\frac{\Lambda}{L}\Big).\]

Let $1\ll \Lambda_n\ll L_n\le \mathfrak c|\log h_n|$, and define $\delta_n:=\Lambda_n^{-1/2}$. Define $B_n$ as $B$ above; then one has $\mathbf P_{\omega_n}(B_n)\ge 1-\frac{1}{\Lambda_n\delta_n}\underset{n\to +\infty}\longrightarrow 1$, and
\[\forall \omega_n\in B_n, \forall k\ge 1, \qquad |\mathbf E_\mu(\widetilde N_{\rm osc}^k(\omega_n,\cdot))-m_k|^2\le D_k\Big(\delta_n+\frac{\Lambda_n}{L_n}\Big).\]
For any fixed $k$, the right-hand side goes to $0$ as $n\to +\infty$. This implies uniform convergence in law of $\widetilde N(\omega_n,\cdot)$ to a centered Gaussian random variable $\mathcal Z\sim \mathcal N(0,\Sigma_{\rm GOE}^2)$ along sequences $\omega_n\in B_n$, by a standard application of the method of moments.

\section{Proof of Proposition \ref{prop: moments N_osc}: perturbations breaking time reversal symmetry} Without repeating the whole argument, we explain the required modifications to the proof of the previous section to obtain Proposition \ref{prop: moments N_osc} for random perturbations that strongly break time reversal symmetry. In this case, the perturbed actions of the orbits $\gamma$ and $\gamma^{-1}$ become almost uncorrelated, which results in a smaller energy variance. Although the proofs follow the same lines, we need to understand more cancellation mechanisms. As in the previous section, some cancellations happen at the level of the perturbed action differences and some happen at the level of the energy averaging.

Let $\mathcal P_0$ denote the set of primitive periodic orbits quotiented by the relation $\gamma\sim \gamma^{-1}$. By writing each periodic orbit as a power of a primitive one, one can write the oscillating term in the trace formula as follows.
\begin{equation} N_{\rm osc}(\omega,\mu)=\frac 1L\sum_{\varepsilon=\pm 1}\sum_{\eta=\pm 1}\sum_{\gamma\in \mathcal P_0}\sum_{k\ge 1} a_{\gamma^k,\varepsilon} \mathrm{e}^{\frac{\mathrm{i}}{h} \varepsilon k S_\omega(\gamma^\eta)}\mathrm{e}^{\mathrm{i}\mu \varepsilon k T_\gamma/L}.\end{equation}
Taking \eqref{eq: new exp nosc} to the $n$-th power yields
\begin{equation} N_{\rm osc}^{n}=\frac{1}{L^n}\sum_{\mathbf t,\vec \eta} \Big(\prod_{1\le i\le n} a_{\gamma_i^{k_i},\varepsilon_i}\Big) \exp\Big(\frac{\mathrm{i}}{h}\sum_{1\le i\le n} \varepsilon_ik_i S_\omega(\gamma_i^{\eta_i})\Big) \exp\Big(\mathrm{i}\mu\sum_{1\le i\le n} \varepsilon_ik_iT_{\gamma_i}/L\Big).\end{equation}
The sum runs over pairs $\mathbf p=(\mathbf t,\vec \eta)$ where $\mathbf t=(\vec \gamma,\vec k,\vec \varepsilon)$ with $\vec \gamma\in \mathcal P_0^n$, $\vec k\in (\mathbf N^*)^n$ and $\vec \varepsilon,\vec \eta\in \{\pm 1\}^n$.
\begin{defi}The pair $\mathbf p=(\mathbf t,\vec \eta)$ is \emph{diagonal} if $\mathbf t$ is diagonal in the sense of Definition \ref{defi: tuples} and $\vec \eta$ is constant equal to $(1,\ldots,1)$ or $(-1,\ldots,-1)$.
\end{defi}
 We split the sum \eqref{eq: Nosc power n} into a \say{diagonal} and an \say{off-diagonal} part:
\begin{equation} N_{\rm osc,\rm{diag}}^n:=\frac{1}{L^n} \sum_{\mathbf p~ \text{diag}}\Big(\prod_{1\le i\le n} a_{\gamma_i^{k_i},\varepsilon_i}\Big), \end{equation}
\begin{equation}  N_{\rm osc,\rm{off}}^n:=\frac{1}{L^n} \sum_{\mathbf p~\text{nondiag}}\Big(\prod_{1\le i\le n} a_{\gamma_i^{k_i},\varepsilon_i}\Big)\exp\Big(\frac{\mathrm{i}}{h}\sum_{1\le i\le n} \varepsilon_ik_i S_\omega(\gamma_i^{\eta_i})\Big) \exp\Big(\mathrm{i}\mu\sum_{1\le i\le n} \varepsilon_ik_iT_{\gamma_i}/L\Big).\end{equation}
The diagonal contribution is evaluated as in \S \ref{sec: proof T sym}, and one shows 
\begin{prop} \label{prop: esperance diag GUE} For all $n\ge 1$,
	\[N_{\rm osc, \rm diag}^{n}=m_{n}+\mathcal O_n(L^{-2}),\]
	where $(m_n)$ is the sequence of moments of $\mathcal N(0,\Sigma_{\rm GUE}^2)$.
\end{prop}

\subsection{The nondiagonal contribution} We now consider the sum over pairs $\mathbf p,\mathbf p'$ where both $\mathbf p$ and $\mathbf p'$ are nondiagonal.

\begin{defi}\label{defi: new nondiag} For each pair $\mathbf p,\mathbf p':=(\mathbf t,\vec \eta),(\mathbf t',\vec \eta')$, we define functions $w_{\mathbf p,\mathbf p'}^{\pm}:\mathcal P_0\to \mathbf Z$ by:
	\[w_{\mathbf p,\mathbf p'}^{\pm}(\gamma)=\sum_{\gamma_i=\gamma,\eta_i=\pm1}\varepsilon_i k_i-\sum_{\gamma_i'=\gamma,\eta_i'=\pm1} \varepsilon_i'k_i'.\]
	We also let $N_{\mathbf p,\mathbf p'}$ denote the number of pairs $\{w_{\mathbf p,\mathbf p'}^+(\gamma),w_{\mathbf p,\mathbf p'}^-(\gamma)\}$ of the form $\{-1,1\}$ or $\{0,k\}$ with $k\neq 0$, for $\gamma\in \mathcal P_0$.

	\end{defi}
This time, one finds, using Proposition \ref{prop: averages actions}:
\begin{equation} \label{eq: this time one finds}\mathbf E_\omega\big[|\mathbf E_\mu^{\rm off}(N_{\rm osc}^n)|^2\Big]\le \frac{C_n}{L^{2n}} \sum_{\substack{\mathbf p=(\mathbf t,\vec \eta),\ \mathbf p'=(\mathbf t',\vec \eta') \\ \text{non diagonal}}}\Big(\prod_\gamma a_{K_{\mathbf t}}(\gamma)a_{K_{\mathbf t'}}(\gamma)\Big) h^{\varepsilon N_{\mathbf p,\mathbf p'}} b_{K_{\mathbf t},K_{\mathbf t'}},\end{equation}
which should be compared with \eqref{eq: intro another shorthand}.

\subsection{The part $N_{\mathbf p,\mathbf p'}\neq 0$} We first consider the subsum of \eqref{eq: this time one finds}, that we denote $\mathcal S^{\rm off}_{\rm nd}$, that consists of pairs $(\mathbf p,\mathbf p')$ such that $N_{\mathbf p,\mathbf p'}\neq 0$. This is analogous to \S \ref{subsec: Delta k neq Delta k'} where we considered subsums consisting of pairs $(\mathbf t,\mathbf t')$ such that $\Delta k_{\mathbf t}\neq \Delta k_{\mathbf t'}$. Using the same method, we obtain
\[|\mathcal S^{\rm off}_{\rm nd}|\le C_n h^\varepsilon.\]

\subsection{The part $N_{\mathbf p,\mathbf p'}=0$} We now consider the subsum of \eqref{eq: this time one finds}, denoted $\mathcal S^{\rm off}_{\rm d}$, that consists of pairs $(\mathbf p,\mathbf p')$ such that $N_{\mathbf p,\mathbf p'}=0$, and $\mathbf p=(\mathbf t,\vec \eta)$ and $\mathbf p'=(\mathbf t',\vec \eta)$ are both nondiagonal in the sense of Definition \ref{defi: new nondiag}. Note that this condition forces $q_{\mathbf t}+q_{\mathbf t}'\ge 2$ on the support. In particular, $q_{\mathbf t}+q_{\mathbf t'}$ has support of size $\le n$.

\subsubsection{Discarding nonprimitive contributions} Consider the subsum of $\mathcal S^{\rm off}_{\rm d}$, denoted $\mathcal S^{\rm off}_{\rm d,\rm np}$, that consists of pairs $(\mathbf p,\mathbf p')$ such that $k_{\mathbf t}(\gamma)+k_{\mathbf t'}(\gamma)\ge 3$ for some $\gamma\in \mathcal P_0$. It is bounded by
\[|\mathcal S_{\rm d,\rm np}^{\rm off}|\le \frac{C_n}{L^{2n}} \sum_{\substack{\mathbf p=(\mathbf t,\vec \eta),\ \mathbf p'=(\mathbf t',\vec \eta') \\ \text{non diagonal}\\ \exists \gamma~:~ k_{\mathbf t}+k_{\mathbf t'}(\gamma)\ge 3} }\Big(\prod_\gamma a_{K_{\mathbf t}}(\gamma)a_{K_{\mathbf t'}}(\gamma)\Big).\]
This term is controlled as $\mathcal S_{\rm np}^{\rm diag}$ in \S \ref{subsec: discard nonprim 1} and one finds
\[|\mathcal S_{\rm d,\rm np}^{\rm off}|\le C_nL^{-2}.\]

\subsubsection{Dealing with primitive contributions} We now consider the subsum of $\mathcal S^{\rm off}_{\rm d}$, denoted $\mathcal S^{\rm off}_{\rm d,\rm p}$, that consists of pairs $(\mathbf p,\mathbf p')=((\mathbf t,\vec \eta),(\mathbf t',\vec \eta'))$ such that $N_{\mathbf p,\mathbf p'}=0$, $k_{\mathbf t}(\gamma)+k_{\mathbf t'}(\gamma)=2$ on its support, and such that $\mathbf p,\mathbf p'$ are nondiagonal, in the sense of Definition \ref{defi: new nondiag}. In particular, each $\gamma\in \mathcal P_0$ appears exactly twice in $\vec \gamma\cup \vec \gamma'$.

\begin{lem}Under the above assumptions on the pair $(\mathbf p,\mathbf p')$, one has $\Delta k_{\mathbf t}\neq 0$ and $\Delta k_{\mathbf t'}\neq 0$.
\end{lem}

\begin{proof}By symmetry, it is enough to prove $\Delta k_{\mathbf t}\neq 0$. By contradiction, assume $\Delta k_{\mathbf t}=0$. Then each $\gamma\in \mathcal P_0$ appears either twice with signs $+1$ and $-1$ or none in $\mathbf t$, in particular $\mathbf t$ is diagonal in the sense of Definition \ref{defi: tuples}. Since $\mathbf p$ is nondiagonal in the sense of Definition \ref{defi: new nondiag}, we deduce that the vector $\vec \eta$ is nonconstant. In particular, one can find indices $i_1,i_2\in \{1,\ldots,n\}$ such that $(\gamma_{i_1},\varepsilon_{i_1},\eta_{i_1})=(\gamma_{i_2},-\varepsilon_{i_2},-\eta_{i_2})$. It follows that $|\Delta k_{\mathbf t}(\gamma_{i_1})|=2$, which contradicts our assumption. 
\end{proof}
With this at hand, we find
\[|\mathcal S^{\rm off}_{\rm d,\rm p}|\le \frac{C_n}{L^{2n}}\sum_{\substack{(\mathbf t,\mathbf t')  \\ \Delta k_{\mathbf t}\neq 0, \ \Delta k_{\mathbf t'}\neq 0 \\ q_{\mathbf t}+q_{\mathbf t'}\in \{0,2\}}} \Big(\prod_\gamma a_{K_{\mathbf t}}(\gamma)a_{K_{\mathbf t'}}(\gamma)\Big) b_{K_{\mathbf t},K_{\mathbf t'}}.\]
Up to permuting indices of the tuples, a pair $(\mathbf t,\mathbf t')$ satisfying the above properties is completely determined by:
\begin{itemize}
	\item the choice of $\operatorname{Supp}(q_{\mathbf t}+q_{\mathbf t'})$, which is a set $\Gamma:=\{g_1,\ldots,g_n\}$ of size $n$.
	\item the choices of maps $\{1,\ldots,n\}\to \Gamma$ which define the vectors $\vec \gamma$ and $\vec \gamma'$. Here there are $\mathcal O_n(1)$ choices.
	\item the choices of values $d_i:=\Delta k_{\mathbf t}(g_i)$ and $d_i'=\Delta k_{\mathbf t'}(g_i)$ in $\{-2,-1,0,1,2\}$ for each $i\in \{1,\ldots,n\}$, such that $\vec d\neq 0$ and $\vec d'\neq 0$. Up to exchanging indices (which adds $\mathcal O_n(1)$ choices) we might assume $d_1\neq 0$ and $d_1'\neq 0$.
	\item the choices of $\vec k,\vec \varepsilon$ for which there are $\mathcal O_n(1)$ choices.
\end{itemize}
It follows that
\[|\mathcal S^{\rm off}_{\rm d,\rm p}|\le \frac{C_n}{L^{2n}}\sum_{g_1,\ldots,g_n}\sum_{\substack{d_1,\ldots,d_n \\ d_1',\ldots,d_n' \\ d_1\neq 0, \ d_1'\neq 0}} \Big(\prod_{1\le i\le n} A_{g_i}^2\Big) \operatorname{sinc}\Big(\frac{\Lambda}{L} \sum_{1\le i\le n} d_iT_{g_i}\Big)\operatorname{sinc}\Big(\frac{\Lambda}{L}\sum_{1\le i\le n} d_i'T_{g_i}\Big).\]
Let $\alpha=\frac{L}\Lambda$ as before. We decompose the sum on the right-hand side as follows.
\begin{itemize}
	\item A first piece $\mathcal S_{-1}$ consists of pairs $(\vec d,\vec d')$ such that $\big|\sum d_iT_{g_i}\big|\le \alpha$ or $\big|\sum d_i'T_{g_i}\big|\le \alpha$.
	\item For $j,j'\ge 0$, we denote by $\mathcal S_{j,j'}$ the subsum consisting of pairs $(\vec d,\vec d')$ such that $\big|\sum d_iT_{g_i}\big|\in (2^j \alpha,2^{j+1}\alpha]$ and $\big|\sum d_i'T_{g_i}\big|\in (2^{j'} \alpha,2^{j'+1}\alpha]$
\end{itemize}

\subsubsection{Dealing with $\mathcal S_{-1}$} This is done exactly as in \S \ref{subsec: controlling S-1} and one finds
\[|\mathcal S_{-1}|\le C_nL^{-1}\alpha.\]

\subsubsection{Dealing with $\mathcal S_{j,j'}$ for $j,j'\ge 0$} Here we adapt the proof of \S \ref{subsec: controlling Sj}. One finds
\[|\mathcal S_{j,j'}|\le C_n \frac{L}{\Lambda 2^j\alpha} \frac{L}{\Lambda 2^{j'}\alpha}\cdot L2^{\min(j,j')}\alpha\le C_n \frac{L}{\Lambda^2\alpha} 2^{\min(j,j')-j-j'}.\]

By summing all the contributions, we obtain analogously to \eqref{eq: bound Acal2}:
\[|\mathcal S^{\rm off}_{\rm d,\rm p}|\le |\mathcal S_{-1}|+\sum_{j,j'\ge 0} |\mathcal S_{j,j'}|\le \frac{C_n}{\Lambda}.\]

\section{Convergence to a Gaussian random field}
\subsection{Proof of Theorem \ref{thm: convergence random function}} \label{sec: random function 1} In this subsection, we prove Theorem \ref{thm: convergence random function} on the convergence of the sequence of random functions $\widetilde N(\omega,\cdot)$ where $\omega$ is picked at random in $\Omega_h$. The proof of Theorem \ref{thm: conv random fun energy} where the energy is picked at random is slightly more technical and will be proved in the next subsection. We prove Theorem \ref{thm: convergence random function} for perturbations that preserve time reversal symmetry. The case of perturbations that strongly break time-reversal symmetry is dealt with in a similar way. We first establish convergence of moments. 
\begin{prop}\label{prop: moments method higher dim} Let $\mathfrak f_{\rm GOE}$ denote the centered Gaussian random field on $\mathbf R$ with covariance kernel $K_{\rm GOE}(\mu,\mu')$. Let $\vec \mu=(\mu_1,\ldots,\mu_n)\in \mathbf R^n$. Then, in the regime $1\le L\le \mathfrak c|\log h|$, one has
	\begin{equation} \label{eq: convergence multivar}\mathbf E_\omega\Big[\prod_{i=1}^n N_{\rm osc}(\omega,\mu_i)\Big]=\mathbf E\Big[\prod_{i=1}^n \mathfrak f_{\rm GOE}(\mu_i)\Big]+\mathcal O_{\vec \mu}(L^{-2}).\end{equation}
	Moreover the error is uniform in $\vec \mu$ provided the differences $|\mu_i-\mu_j|$ remain in a compact set.
\end{prop}
Since $\widetilde N(\omega,\mu)=\widetilde N_{\rm osc}(\omega,\mu)+\mathcal O_{\vec \mu}(L^{-1})$, a corollary is that \eqref{eq: convergence multivar} holds with $N_{\rm osc}$ replaced by $\widetilde N$, up to replacing the error by $\mathcal O_{\vec \mu}(L^{-1})$. Using the multivariate method of moments \cite[Theorem 6.1]{puder2022local}, Proposition \ref{prop: moments method higher dim} implies that for any $\mu_1,\ldots,\mu_n\in \mathbf{R}$, one has convergence in distribution, as $h\to 0$:
\[\big(\widetilde N(\omega,\mu_1),\ldots,\widetilde N(\omega,\mu_n)\big)\longrightarrow \big(\mathfrak f_{\rm GOE}(\mu_1),\ldots,\mathfrak f_{\rm GOE}(\mu_n)\big).\]
As argued in \cite[\S 9]{ingremeau2023emergence}, thanks to Prokhorov's theorem, it is then enough to prove that the sequence of random functions $\widetilde N(\omega,\cdot)$ is \emph{tight} in $C^\infty(\mathbf{R})$ to obtain Theorem \ref{thm: convergence random function}.

\begin{proof}[Proof of Proposition \ref{prop: moments method higher dim}]By \cite[Lemma 11.6.1]{RandomFields}, one has
	\begin{equation} \label{eq: expected prod GFF}\mathbf E\Big[\prod_{i=1}^n \mathfrak f_{\rm GOE}(\mu_i)\Big]=\sum_{\mathbf A} \prod_{\{i,j\}\in \mathbf A} K_{\rm GOE}(\mu_i-\mu_j),\end{equation}
	where the sum runs over \emph{perfect matchings} $\mathbf A$ of $\{1,\ldots,n\}$. Recall that a \emph{perfect matching} on $\{1,\ldots,n\}$ is a partition $\mathbf A$ of this set into pairs. We write $i\sim j$ if $\{i,j\}\in \mathbf A$.
	
Let us denote $\mathbf I=\mathbf E_\omega\Big[\prod_{i=1}^n N_{\rm osc}(\omega,\mu_i)\Big]$. We need to show that $\mathbf I$ approaches \eqref{eq: expected prod GFF} as $h\to 0$. Start from
\[\mathbf I=\frac{2^n}{L^n}\sum_{(\vec \gamma,\vec k,\vec \varepsilon)} \prod_{1\le i\le n} a_{\gamma_i^{k_i},\varepsilon_i}\mathbf E_\omega\Big[\exp\Big(\frac{\mathrm{i}}{h}\sum_{1\le i\le n} \varepsilon_i k_iS_\omega(\gamma_i)\Big)\Big]\exp\Big(\frac{\mathrm{i}}{L}\sum_{1\le i\le n} \varepsilon_ik_i \mu_iT_{\gamma_i}\Big).\]
One can discard tuples that are nonprimitive or nondiagonal, which contribute to an error $\mathcal O(L^{-2})$. That is,
\begin{equation} \label{eq: one can discard}\mathbf I=\frac{2^n}{L^n}\sum_{\substack{\mathbf t=(\vec \gamma,\vec \varepsilon) \\ \text{diagonal} \\ \text{and primitive}}} \Big(\prod_{1\le i\le n} a_{\gamma_i,\varepsilon_i}\Big)\exp\Big(\frac{\mathrm{i}}{L}\sum_{1\le i\le n} \varepsilon_i\mu_iT_{\gamma_i}\Big)+\mathcal O(L^{-2}).\end{equation}
Since $(\vec \gamma,\vec \varepsilon)$ being diagonal and primitive forces $n$ to be even, we now consider the case $n=2p$ with $p\ge 1$. We can then rewrite \eqref{eq: one can discard} as
\begin{equation} \label{eq:we can then rewrite}\mathbf I=\frac{2^{2p}}{L^{2p}}\sum_{\substack{\mathbf A\  \text{perfect} \\ \text{ matching}}} \sum_{\substack{\gamma_1,\ldots,\gamma_{2p}\in \mathcal P_0 \\ \varepsilon_1,\ldots,\varepsilon_{2p}\in \{\pm 1\} \\ (\gamma_j,\varepsilon_j)=(\gamma_i,-\varepsilon_i) \ \text{if $i\sim j$}}} \Big(\prod_{i=1}^{2p} a_{\gamma_i,\varepsilon_i}\Big)\exp\Big(\frac{\mathrm{i}}{L}\sum_{i=1}^{2p} \varepsilon_i\mu_iT_{\gamma_i}\Big)+\mathcal O(L^{-2}).\end{equation}
Here we have omitted the vector $\vec k=(1,\ldots,1)$ from the notation. For each pair $\mathbf a=\{i,j\}\in \mathbf A$ one must choose an element $\gamma_{\mathbf a}\in \mathcal P_0$ such that $\gamma_i=\gamma_j=\gamma_{\mathbf a}$. Then if $\vec \varepsilon$ is chosen so that $\varepsilon_j=-\varepsilon_i$ when $\{i,j\}\in \mathbf A$, one has 
\[\prod_{i=1}^{2p} a_{\gamma_i,\varepsilon_i}=\prod_{\mathbf a\in \mathbf A}A_{\gamma_{\mathbf a}}^2.\]
Moreover, given a perfect matching $\mathbf A$, one has
\[\sum_{\substack{\varepsilon_1,\ldots,\varepsilon_{2p}\in \{\pm 1\}\\ \varepsilon_j=-\varepsilon_i \ \text{if $i\sim j$}}} \exp\Big(\frac{\mathrm{i}}{L}\sum_{\mathbf a=\{i,j\}\in \mathbf A} (\varepsilon_i\mu_i+\varepsilon_j\mu_j)T_{\gamma_{\mathbf a}}\Big)=2^p\prod_{\mathbf a=\{i,j\}\in \mathbf A} \cos\Big((\mu_i-\mu_j)\frac{T_{\gamma_{\mathbf a}}}{L}\Big).\]
Then, \eqref{eq:we can then rewrite} becomes
\begin{equation}\label{eq: write as perfect} \begin{aligned}\mathbf I & =\frac{2^{3p}}{L^{2p}}\sum_{\substack{\mathbf A\  \text{perfect} \\ \text{ matching}}} \Big[\sum_{\substack{\gamma:\mathbf A\to \mathcal P_0 \\ \text{injective}}} \Big(\prod_{\mathbf a=\{i,j\}\in \mathbf A}A_{\gamma_{\mathbf a}}^2\cos((\mu_i-\mu_j)T_{\gamma_{\mathbf a}}/L)\Big)\Big]+\mathcal O(L^{-2})\\
& =\sum_{\substack{\mathbf A\  \text{perfect} \\ \text{ matching}}}\Big[ \prod_{\mathbf a=\{i,j\}\in \mathbf A} \Big(\frac{8}{L^2}\sum_{\gamma\in \mathcal P_0} A_{\gamma }^2\cos((\mu_i-\mu_j)T_{\gamma}/L)\Big)\Big]+\mathcal O(L^{-2}).\end{aligned}\end{equation}
Here, to pass from the first line to the second line we have replaced the sum over injective maps $\mathbf A\to \mathcal P_0$ by the sum over all maps, and we used the fact that noninjective maps $\mathbf A\to \mathcal P_0$ give a contribution $\mathcal O(L^{-2})$. Now for each pair $\{i,j\}$, one has from Proposition \ref{prop: equidistribution result}:
\[\frac{8}{L^2}\sum_{\gamma\in \mathcal P_0} A_{\gamma }^2\cos((\mu_i-\mu_j)T_{\gamma}/L)=4\int_{0}^{+\infty}t|\widehat f(t)|^2 \cos((\mu_i-\mu_j)t)\mathrm{d}t+\mathcal O_{\mu_i,\mu_j}(L^{-2}).\]
Here the error is uniform in $\vec \mu$ provided the differences $|\mu_i-\mu_j|$ lie in a compact set. The leading term on the right-hand side is nothing but $k_{\rm GOE}(|\mu_i-\mu_j|)$. Plugging these asymptotics in \eqref{eq: write as perfect} leads to the result. \end{proof}
It remains to prove the tightness of the sequence of random functions $\widetilde N(\omega,\cdot)$. Since $\widetilde N(\omega,\cdot)$ is only defined on $[-\Lambda,\Lambda]$, we introduce a smooth cutoff to make it defined on the whole real line.
\begin{prop}\label{prop: tightness 1} Let $\chi\in C^\infty_{\rm comp}([-1,1],\mathbf{R})$. The sequence of random functions $\varphi_h:\mu\mapsto \chi(\mu/\Lambda)\widetilde N(\omega,\mu)$ is tight in $C^\infty(\mathbf{R})$. That is, for any $\varepsilon>0$, there is a compact subset $\mathcal K_\varepsilon\subset C^\infty(\mathbf{R})$ such that
\[\sup_{h\in (0,1]} \mathbf P_\omega\big(\varphi_h\in \mathcal K_\varepsilon\big)>1-\varepsilon.\] \end{prop}
\begin{proof}Following \cite[\S 9]{ingremeau2023emergence}, we consider compact subsets of the form
	\[\mathcal K(a)=\{f\in C^\infty(\mathbf{R})~:~ \forall \ell,k\ge 0, \ \|f\|_{C^\ell([-k,k])}\le a_{\ell,k}\},\]
	where $a=\{a_{\ell,k}\}_{\ell,k\ge 0}$ is a family of real numbers. It is easy to check
	\[\mathcal K(a)^c\subset \bigcup_{k,\ell} \{f~:~ \|f\|_{C^\ell([-k,k])}>a_{\ell,k}\}.\]
	In turn,
	\begin{equation}\label{eq: by markov} \begin{aligned}\mathbf P_\omega\big(\varphi_h\in \mathcal K(a)^c\big)& \le \sum_{k,\ell\ge 0} \mathbf P_\omega\big(\|\varphi_h\|_{C^\ell([-k,k])}>a_{\ell,k}\big)\end{aligned}\end{equation}
	To study derivatives of $\widetilde N(\omega,\mu)$ in the $\mu$ variable, we go back to the definition
	\[\partial_\mu^\ell N(\omega,\mu)=(-1)^\ell\operatorname{Tr}f^{(\ell)}\Big(\frac{P_h^\omega-E}{h}L-\mu\Big),\]
	where $f^{(\ell)}$ denotes the derivative of $f$ of order $\ell$, which has Fourier transform $\widehat{f^{(\ell)}}(\xi)=(-\mathrm{i}\xi)^\ell \widehat f(\xi)$. By applying the trace formula from Proposition \ref{prop: long time gutz} to $f^{(\ell)}$, one finds, for $\ell\ge 0$, $\mu\in [-\Lambda,\Lambda]$ and $\Lambda\le L$:
	\[\partial_\mu^\ell N(\omega,\mu)=\frac{(-1)^\ell C_\omega}{hL} \widehat{f^{(\ell)}}(0)+\partial_\mu^\ell N_{\rm osc}(\omega,\mu)+R_\ell(\omega,\mu),\]
	where the error is bounded deterministically by $|R_\ell(\omega,\mu)|\lesssim_{\!\! f,\ell} \Lambda/L$. In turn,
	\begin{equation} \label{eq: partial mu ell tilde N}\partial_\mu^\ell \widetilde N(\omega,\mu)=\partial_\mu^\ell \widetilde N_{\rm osc}(\omega,\mu)+\widetilde R_\ell(\omega,\mu),\end{equation}
	where we write $\widetilde Y=Y-\mathbf E_\mu[Y]$. By the proof of Proposition \ref{prop: moments N_osc}, we know that there is a family of $h$-independent constants $C_{\ell}\ge 1$, such that for any $0\le |k|\le \Lambda$ and $\ell\ge 0$, one has
	\[\mathbf E_\omega\Big[\int_{k-1}^{k+1} |\partial_\mu^\ell N_{\rm osc}(\omega,\mu)|^2\mathrm{d}\mu\Big]\le C_{\ell}.\]
	It follows from the Sobolev embedding theorem that for all $0\le k\le \Lambda$ and $\ell\ge 0$, one has
	\begin{equation} \label{eq: esperance C ell norm} \mathbf E_\omega(\|N_{\rm osc}(\omega,\cdot)\|_{C^\ell([-k,k])})\le C_{\ell,k},\end{equation}
	for some family of constants $C_{\ell,k}\ge 1$. Fix a family of numbers $(b_{\ell,k})$, to be chosen below. Using \eqref{eq: esperance C ell norm}, the Markov inequality shows that there exists a subset $B_{\ell,k}\subset \Omega_h$ with probability $\ge 1-\frac{C_{\ell,k}}{b_{\ell,k}}$ such that
	\[\forall \omega\in B_{\ell,k}, \qquad \|N_{\rm osc}(\omega,\cdot)\|_{C^\ell([-k,k])}\le b_{\ell,k}.\]
	Since $\mathbf E_\omega(|\mathbf E_\mu(N_{\rm osc})|^2)\le C_0\Lambda^{-1}$, up to replacing $B_{\ell,k}$ by a subset of probability $\ge 1-\frac{2 C_{\ell,k}}{b_{\ell,k}}$, we have
	\[\forall \omega\in B_{\ell,k}, \qquad \|\widetilde N_{\rm osc}(\omega,\cdot)\|_{C^\ell([-k,k])}\le 2 b_{\ell,k}.\]
	Recalling \eqref{eq: partial mu ell tilde N} and the bound $|\widetilde R_\ell|\lesssim \Lambda/L\lesssim 1$, we deduce
	\[\forall \omega\in B_{\ell,k},  \qquad \|\widetilde N(\omega,\cdot)\|_{C^\ell([-k,k])}\le 2 b_{\ell,k}+C_\ell.\]
	Multiplying $\widetilde N$ by the cutoff $\chi(\cdot/\Lambda)$ does not affect the estimates and we find another family of constants $(b_{\ell,k}')$ depending on $(b_{\ell,k})$, such that
	\[\forall \ell,k\ge 0, \forall \omega\in B_{\ell,k},\ \qquad \|\varphi_h\|_{C^\ell([-k,k])}\le b_{\ell,k}'.\]
	Thus,
	\[\sum_{k,\ell\ge 0} \mathbf P_\omega\big(\|\varphi_h\|_{C^\ell([-k,k])}>b_{\ell,k}'\big)\le \sum_{k,\ell\ge 0} (1-\mathbf P_{\omega}(B_{\ell,k}))\le \sum_{k,\ell\ge 0} \frac{2C_{\ell,k}}{b_{\ell,k}}.\]
	We now take $b_{\ell,k}= 8C_{\ell,k} 2^{k+\ell} \varepsilon^{-1}$ to make the right-hand side smaller than $\varepsilon$. Going back to \eqref{eq: by markov}, it is enough to take $a_{\ell,k}=b_{\ell,k}'$ to conclude that
	\[\mathbf P_\omega(\varphi_h\in \mathcal K(a)^c)\le \varepsilon\]
	uniformly in $h\in (0,1]$.
	 \end{proof}
	 
	 \subsection{Proof of Theorem \ref{thm: conv random fun energy}} This final section is devoted to the proof of Theorem \ref{thm: conv random fun energy}, which combines ideas from the proofs of Theorem \ref{thm: as GOE fluctuations} and Theorem \ref{thm: convergence random function}. The next proposition is analogous to Proposition \ref{prop: moments N_osc} and is proved in a similar way. 
	 \begin{prop} \label{prop: moments method gen} Endow the interval $[-\Lambda,\Lambda]$ with the uniform probability measure $\mathbf P_\lambda$. Let $b,d\in \mathbf N^*$. For any $\vec \mu=(\mu_1,\ldots,\mu_d)\in [-b,b]^d$, define
	 	\[G(\omega,\vec \mu):=\mathbf E_\lambda\Big[\prod_{i=1}^dN_{\rm osc}(\omega,\lambda+\mu_i)\Big]-\mathbf E\Big[\prod_{i=1}^d \mathfrak f_{\rm GOE}(\mu_i)\Big].\]
	 	Then, one has
	 	\[\mathbf E_\omega\big(|G(\omega,\vec \mu)|^2\big)=\mathcal O_{b,d}(\Lambda^{-1}). \]
	 \end{prop}
	 \begin{proof}As in Section \ref{sec: proof T sym}, we decompose
	 	\[\prod_{i=1}^dN_{\rm osc}(\omega,\lambda+\mu_i)=:\mathcal N_{\rm diag}+\mathcal N_{\rm osc}\]
	 	as the sum of a diagonal part and a nondiagonal part. The diagonal part $\mathcal N_{\rm diag}$ is in fact independent of $\omega$ and $\lambda$, and is given by the formula \eqref{eq: one can discard} (with $d$ in place of $n$). It was shown in Section \ref{sec: random function 1} to satisfy
	 	\[\mathcal N_{\rm diag}=\mathbf E\Big[\prod_{i=1}^d \mathfrak f_{\rm GOE}(\mu_i)\Big]+\mathcal O(L^{-2}).\]
	 	Thus, one has
	 	\[\mathbf E_\omega\big(|G(\omega,\vec \mu)|^2\big)\lesssim \mathbf E_\omega\big(|\mathbf E_\lambda[\mathcal N_{\rm off}]|^2\big)+\mathcal O(L^{-2}).\]
	 	The nondiagonal part $\mathcal N_{\rm off}$ is defined as \eqref{eq: def N osc off}, except that one should replace $n$ by $d$ there, and write
	 	\[\exp\Big(\mathrm{i}\sum_{i=1}^d \mu_i \varepsilon_ik_iT_{\gamma_i}/L\Big)\quad \text{instead of} \quad \exp\Big(\mathrm{i}\mu\sum_{i=1}^n \varepsilon_ik_iT_{\gamma_i}/L\Big).\]
	 	By repeating the proof of Proposition \ref{prop: the nondiag cont}, we find
	 	\[\mathbf E_\omega\big(|\mathbf E_\lambda[\mathcal N_{\rm off}]|^2\big)\lesssim \Lambda^{-1}.\]
	 \end{proof}
	 Fix $\vec \mu\in \mathbf R^d$. Proposition \ref{prop: moments method gen} shows the existence of a sequence of subsets $B_n(\vec \mu)\subset \Omega_{h_n}$ with $\mathbf P_\omega(B_n(\vec \mu))\to 1$, such that for any sequence $(\omega_n\in B_n(\vec \mu))$, the sequence of random vectors $(N_{\rm osc}(\omega_n,\lambda+\mu_1),\ldots,N_{\rm osc}(\omega_n,\lambda+\mu_d))$ (where $\lambda$ is picked at random) converges in law to a centered Gaussian random vector. However, we want to control the convergence of this sequence of random vectors not only for one value of $\vec \mu$, but for all values of $\vec \mu$ in a compact set of $\mathbf{R}^d$. We cannot naively take the intersection of the sets $B_n(\vec \mu)$ because we have no control on the measure of this set.
	 
	 To remedy this issue, we fix a compact set $[-b,b]^d$ and set up a fine enough net of points in $[-b,b]^d$. We apply Proposition \ref{prop: moments method gen} for each vector in the net, then control the adjacent vectors by estimating the variations of $G(\omega,\vec \mu)$ with $\vec \mu$ using the next lemma.
	 
	 \begin{lem}\label{lem: esperance cell bis} There is a family of $h$-independent constants $C_{\ell,b,d}\ge 1$ such that for any $\ell,b,d\ge 1$, one has
	 \[\mathbf E_\omega\big(\big\|G(\omega,\vec \mu)\big\|_{C^\ell_{\vec \mu}([-b,b]^d)}^2\big)\le C_{\ell,b,d}.\]\end{lem}
	\begin{proof} Arguing as in Section \ref{sec: random function 1}, we show that for any $\ell,b,d\in \mathbf N^*$, one has
		\begin{equation} \label{eq: enough to get tightness}\mathbf E_\omega\Big(\Big\|\mathbf E_\lambda\Big[\prod_{i=1}^dN_{\rm osc}(\omega,\lambda+\mu_i)\Big] \Big\|^2_{C^\ell_{\vec \mu}([-b,b]^d)}\Big)\le C_{\ell,b,d},\end{equation}
	 for some family of $h$-independent constants $C_{\ell,b,d}\ge 1$. Up to increasing $C_{\ell,b,d}$, we also have
	 \[\Big\|\mathbf E\Big[\prod_{i=1}^d \mathfrak f_{\rm GOE}(\mu_i)\Big] \Big\|^2_{C^\ell_{\vec \mu}([-b,b]^d)}\le C_{\ell,b,d},\]
	 which leads to the result, up to increasing $C_{\ell,b,d}$ again. \end{proof}
	 Let $r>0$, and consider a net of points $\Gamma\subset [-b,b]^d$ of size $\#\Gamma \lesssim_{\!\! b,d} r^{-d}$, such that
	 \[[-b,b]^d\subset \bigcup_{\vec \mu_\star\in \Gamma} B(\vec \mu_\star,r).\]
	 Then, for any $\vec \mu\in [-b,b]^d$, one can find $\vec \mu_\star\in \Gamma$ such that $|\vec \mu-\vec \mu_\star|\le r$. In turn,
	 \[|G(\omega,\vec \mu)|^2\le 2|G(\omega,\vec \mu_\star)|^2+2r^2\|G(\omega,\cdot)\|_{C^1([-b,b]^d)}^2\]
	 Taking the supremum over $\vec \mu\in [-b,b]^d$ leads to
	 \[\big\|G(\omega,\vec \mu)^2\big\|_{L^\infty([-b,b]^d)}\le 2\Big(\sum_{\vec \mu_\star\in \Gamma} |G(\omega,\vec \mu_\star)|^2\Big) +2r^2\|G(\omega,\cdot)\|_{C^1([-b,b]^d)}^2.\]
	 Taking the expectations and recalling $\#\Gamma\lesssim_{\!\! b,d}r^{-d}$, Proposition \ref{prop: moments method gen} and Lemma \ref{lem: esperance cell bis} lead to the following improvement of Proposition \ref{prop: moments method gen}.
	 \begin{prop}There exists a family of $h$-independent constants $C_{b,d}\ge 1$ such that for any $r>0$, one has
	 \[\mathbf E_\omega\Big(\big\|G(\omega,\vec \mu)^2\big\|_{L^\infty([-b,b]^d)}\Big)\le C_{b,d}\Big(\frac{r^{-d}}{\Lambda}+r^2\Big).\]\end{prop}
	Let $\varepsilon=\frac{r^{-d}}{\Lambda}+r^2$. Fix $\delta>0$, and a family of constants $m_{b,d}>0$. By the Markov inequality, there exists a subset $B_{b,d}\subset \Omega_h$ of probability $\ge 1-\frac{C_{b,d} \varepsilon}{m_{b,d} \delta}$ such that
	\[\forall \omega\in B_{b,d}, \forall \vec \mu\in [-b,b]^d, \qquad |G(\omega,\vec \mu)|^2\le m_{b,d} \delta.\]
	Let $m_{b,d}=2^{b+d}C_{b,d}$, and set $B:=\bigcap_{b,d\ge 1} B_{b,d}$. Then $\mathbf P_\omega(B)\ge 1-\frac{\varepsilon}{\delta}$ and
	\[\forall \omega\in B, \forall d\ge 1, \forall b\ge 1, \forall \vec \mu\in [-b,b]^d, \qquad |G(\omega,\vec \mu)|^2\le 2^{b+d}C_{b,d} \delta.\]
	Let $Y(\omega)=\mathbf E_\lambda[N_{\rm osc}(\omega,\lambda)]$. Recall that there is a subset $B'$ of probability $\ge 1-\frac{C}{\Lambda \delta}$ such that $|Y(\omega)|^2\le \delta$ for all $\omega\in B'$. Also recall that $\widetilde N=\widetilde N_{\rm osc}+\widetilde R$ with $|\widetilde R|\lesssim \Lambda/L$. Replace $B$ by $B\cap B'$. Then $\mathbf P_\omega(B)\ge 1-\frac{\varepsilon}{\delta}-\frac{C}{\Lambda \delta}$ and for all $\omega\in B$, $b,d\ge 1$ and $\vec \mu\in [-b,b]^d$, one has
	\[\Big|\mathbf E_\lambda\Big[\prod_{i=1}^d\widetilde N(\omega,\lambda+\mu_i)\Big]-\mathbf E\Big[\prod_{i=1}^d \mathfrak f_{\rm GOE}(\mu_i)\Big]\Big|^2\lesssim_{\!\! b,d}\delta+\frac{\Lambda^2}{L^2}.\]
	
	Now let $h=h_n\to 0$, and let $1\ll \Lambda_n\ll L_n\le \mathfrak c|\log h_n|$. Take $r_n=\Lambda_n^{-1/2d}$ so that $\varepsilon_n\to 0$, and finally let $\delta_n=\varepsilon_n^{1/2}$ so that $\varepsilon_n/\delta_n+1/\Lambda_n\delta_n\to 0$. We end up with the following proposition, where the constants need not be the same as above.
	\begin{prop}\label{prop: moments method for blabla} Assume $1\ll \Lambda_n\le L_n\le \mathfrak c|\log h_n|$. Then there exists a sequence $\delta_n\to 0$ and a family of measurable subsets $B_n\subset \Omega_{h_n}$ such that the following holds.
		\begin{itemize}
  \item $\mathbf P_\omega(B_n)\longrightarrow 1$.
  \item There exists a family of constants $C_{b,d}$ such that for all $\omega_n\in B_n$, for all $b,d\ge 1$, for all $\vec \mu\in [-b,b]^d$, one has
  \[\Big|\mathbf E_\lambda\Big[\prod_{i=1}^d\widetilde N(\omega_n,\lambda+\mu_i)\Big]-\mathbf E\Big[\prod_{i=1}^d \mathfrak f_{\rm GOE}(\mu_i)\Big]\Big|\le C_{b,d} \delta_n.\]
\end{itemize}
	\end{prop}
	Proposition \ref{prop: moments method for blabla} is already enough to show that for any sequence $(\omega_n\in B_n)$, for any $\vec \mu\in \mathbf{R}^d$, the sequence of random vectors $(\widetilde N(\omega_n,\lambda+\mu_1),\ldots,\widetilde N(\omega_n,\lambda+\mu_d))$, where $\lambda$ is picked at random in $[-\Lambda_n,\Lambda_n]$, converges in law to a centered Gaussian random vector with covariance matrix $(K_{\rm GOE}(\mu_i,\mu_j))_{1\le i,j\le d}$. To prove convergence of the sequence of random functions $N(\omega_n,\lambda+\cdot)$ and obtain Theorem \ref{thm: conv random fun energy}, we are left with proving tightness of the sequence, but this basically follows from the estimate \eqref{eq: enough to get tightness} and arguing as in the proof of Proposition \ref{prop: tightness 1}.
	\printbibliography
	
\end{document}